\documentclass[10pt]{amsart}

\usepackage[dvipdfmx]{graphicx, color}
\usepackage[all]{xy}
\usepackage{amssymb}
\usepackage{amscd}
\usepackage{amsmath,amsfonts,amsthm,amssymb,amscd}
\usepackage{mathrsfs}
\usepackage{extpfeil}
\usepackage[T1]{fontenc}
\usepackage{enumerate}

\usepackage{graphicx}

\makeatletter
\@namedef{subjclassname@2010}{%
  \textup{2010} Mathematics Subject Classification}
\makeatother

\newtheorem{Thm}{Theorem}[section]
\newtheorem{Cor}[Thm]{Corollary}
\newtheorem{Prop}[Thm]{Proposition}
\newtheorem{Lem}[Thm]{Lemma}
\newtheorem{Rem}[Thm]{Remark}
\newtheorem{Def}[Thm]{Definition}
\newtheorem{Asum}[Thm]{\textbf{Assumption}}

\numberwithin{equation}{section}

\newcommand{\Ker}{{\rm{Ker}}}

\newcommand{\Spec}{{\rm{Spec}}}

\newcommand{\resp}{{\rm{resp}}}
\newcommand{\Imm}{{\rm{Im}}}

\newcommand{\dlog}{{\rm{dlog}}}
\newcommand{\tcf}{{\rm{the\;class\;of}}}

\newcommand{\gr}{{\rm{gr}}}
\newcommand{\pr}{{\rm{pr}}}

\newcommand{\Coker}{{\rm{Coker}}}
\newcommand{\Cone}{{\rm{Cone}}}
\newcommand{\id}{{\rm{id}}}
\newcommand{\Symb}{{\rm{Symb}}}
\newcommand{\Res}{{\rm{Res}}}
\newcommand{\Supp}{{\rm{Supp}}}

\begin{document}


\baselineskip=19pt



\title[On syntomic complex with modulus for semi-stable reduction case  ]{On  syntomic complex with modulus for semi-stable reduction case }

\author[K. Yamamoto]{Kento Yamamoto}
\address{Department of Mathematics,Chuo University 1-13-27 Kasuga, Bunkyo-Ku, Tokyo 112-8551, Japan}
\email{k-yamamo@gug.math.chuo-u.ac.jp}

\date{\today}

\begin{abstract}
In this paper, we define syntomic complex for modulus pair $(X,D)$, where $X$ is regular semi-stable family and  $D$ is an effective Cartier divisor on $X$. We compute its cohomology sheaves.
\end{abstract}

\subjclass[2010]{14F30}

\keywords{syntomic complex with modulus}

\maketitle
\tableofcontents

\section{\textbf{Introduction}}

In their paper \cite{KSY}, Bruno Kahn, Hiroyasu Miyazaki, Shuji Saito and Takao Yamazaki study to construct a triangulated category of
motives with modulus $``MDM^{eff}
_{gm}"$ over a field $k$ that extends Voevodsky's category $DM^{eff}
_{gm}$ with non $\mathbb{A}^1$-homotopy invariant property. While
the Voevodsky's category $DM^{eff}_{gm}$ is
constructed from smooth $k$-varieties, the category of
motives with modulus $``MDM^{eff}
_{gm}"$ is expected to be constructed from proper modulus pairs $(X,D)$, that is, pairs of a proper $k$-variety $X$ and
an effective divisor $D$ on $X$ such that $X-|D|$ is smooth. 

Let $K$ be a $p$-adic field, and let $\mathscr{O}_K$ be its valuation ring with $k$ the residue field. Let $X$ be a regular semistable family over $\mathscr{O}_K$ and put $Y:=X\otimes_{\mathscr{O}_K}k$. Let $D$ be an effective Cartier divisor which is flat over $\mathscr{O}_K$ and such that  $Y \cup D_{{\rm red}}$ has normal crossings on $X$.
The first aim of this paper is to define the syntomic complex $\mathscr{S}_n(r)_{X|D}$ with modulus for such pairs ($X,D$) for $n\geq 1$ and $0 \leq r \leq p-1$, which is a generalization of Tsuji's syntomic complex $\mathcal{S}_n(r)_{(X,M_X)}$\;(cf.\ \cite{Ka1}, \cite {Ka2}, \cite{Ku}, \cite{Tsu0}, \cite{Tsu1}, \cite{Tsu2} etc.). More explicitly, we have $\mathscr{S}_n(r)_{X|\emptyset}=\mathcal{S}_n(r)_{(X,M_X)}$. In \cite{Tsu0}, \cite{Tsu1} and \cite{Tsu2}, Tsuji constructed the symbol map 
\[{\rm Symb}_X: (M_{X_n}^{gp})^{\otimes r} \longrightarrow \mathcal{H}^r(\mathcal{S}_n(r)_{(X,M_X)})\]
and proved the surjectivity of this map.
The second aim of this paper is to construct a symbol map for $\mathscr{S}_n(r)_{X|D}$ and to investigate its surjectivity.
We will prove the following main result:
\begin{Thm}{\rm(Theorem \ref{main result})} Let $n \geq 1$ be an integer. If $0 \leq r \leq p-2$, $p \geq 3$,
the cokernel of the symbol map \[ {\rm Symb}_{X|D}:(1+I_{D_{n+1}})^{\times} \otimes (M_{X_{n+1}}^{gp})^{\otimes{r-1}} \longrightarrow \mathcal{H}^r\left(\mathscr{S}_n(r)_{X|D}\right)\] is Mittag-Leffler zero with respect to the multiplicities of the prime components of $D$. Here $X_n:=X\otimes \mathbb{Z}/p^n\mathbb{Z}$, $D_n:=D\otimes \mathbb{Z}/p^n\mathbb{Z}$ and $I_{D_{n+1}}(\subset \mathscr{O}_{X_{n+1}})$ is the definition ideal of $D_{n+1}$;\;$M_X$ denotes the log structure on $X$ associated with $Y \cup D_{{\rm red}}$, and $M_{X_n}$ is the inverse image of $M_X$ onto $X_n$.
\end{Thm}

In fact, the cokernel of ${\rm Symb}_{X|D}$ is non-zero unless $D$ is zero or reduced, and deeply depends on the multiplicities of the prime components of $D$.
Nevertheless our main result asserts that those cokernels are Mittag-Leffler zero as a projective system. A key fact to understand this phenomenon is a Cartier inverse isomorphism in a modulus situation (see Lemma \ref{Omega} below).  From this key lemma, we will obtain an explicit description of the cokernel of the symbol map in a sufficiently local situation.

As an application of the subject of this paper, we will consider a $p$-adic \'etale Tate twists for a modulus pair $(X, D)$ in a forthcoming paper \cite{Y}, which is a generalization of Sato's $p$-adic \'etale Tate twists (\cite{Sat}). We will show that our new object is a ``dual'' of the usual $p$-adic \'etale Tate twists of $X-D$.\\

\noindent
\textbf{Notation and conventions.}\begin{enumerate}[\upshape (i)]
\item Throughout this paper, $p$ denotes a prime number and $K$ denotes a henselian discrete valuation field of characteristic $0$ whose residue field $k$ is a perfect field of characteristic $p$. We write  $\mathscr{O}_K$ for the integer ring of $K$, and $\pi$ denotes a prime element of $\mathscr{O}_K$. We denote by $\hat{K}$ the completion of $K$ with respect to the discrete valuation and by $\mathscr{O}_{\hat{K}}$ its ring of integers.
\item\ Throughout this paper, we assume that a scheme $X$ is always separated over $\mathscr{O}_K$. 
\item For a scheme $X$, we put $X_n:=X \otimes \mathbb{Z}/p^n\mathbb{Z}$.
\item \label{Not(iv)}Let $X$ be a pure-dimensional scheme which is flat of finite type over $\Spec(\mathscr{O}_K)$. We call X a $regular\;semistable\;family$ over  $\Spec(\mathscr{O}_K)$, if it is regular and 
evrywhere \'etale locally isomorphic to 
\[\Spec\big(\mathscr{O}_K[T_0,\dotsc,T_d]/(T_0 \dotsb T_a-\pi)\big),\] 
\noindent
for some $a$ such that $0\leq a \leq d:=\dim(X/\mathscr{O}_K)$.
\end{enumerate}

\section{\textbf{syntomic complex with modulus}}\label{Tate modulus}\mbox{}
In this section, we will define syntomic complex with modulus $\mathscr{S}_n(q)_{X|D}$ for $0 \leq q \leq p-1$.

\noindent
\textbf{Setting}:
\begin{itemize}
\item Let $X$ be a $regular\;semistable\;family$ over  $\Spec(\mathscr{O}_K)$. We set $Y:=X \otimes_{\mathscr{O}_K}k$ and $X_K:=X \otimes_{\mathscr{O}_K}K$. Let $D \subset X$ be an effective Cartier divisor on $X$ which is flat over $\Spec(\mathscr{O}_K)$ and $Y \cup D_{{\rm red}}$ has normal crossings on $X$.
\item Let $M_X$ be a logarithmic structure on $X$ associated with $Y \cup D_{{\rm red}} $. Let $M_D$ be a logarithmic structure on $D$ defined as the restriction of $M_X$. For $n\geq 1$, we write $M_{X_n}$ for the inverse image of log structure of $M_X$ onto $X_n$.
Let $(Y, M_Y)$ be the reduction mod $\pi$ of $(X, M_X)$.
\end{itemize}

\subsection{Local construction}
To define the syntomic complex with modulus in a sufficiently local situation, we assume the existence of the following data:
\begin{Asum}\label{Asum}\mbox{}

\begin{itemize}
\item\ There exist exact closed immersions \[\beta_n: (X_n,M_{X_n}) \hookrightarrow(Z_n,M_{Z_n})\ \ {\rm and}\ \  \beta_{n,D}: (D_n,M_{D_n}) \hookrightarrow(\mathscr{D}_n,M_{\mathscr{D}_n})\] of log schemes for $n \geq 1$ such that $(Z_n,M_{Z_n})$ and $(\mathscr{D}_n,M_{\mathscr{D}_n})$ are smooth over $W:=W(k)$,  and such that the following diagram is commutative:
\[
\xymatrix@M=10pt{
X_n \ar@<-0.3ex>@{^{(}->}[r]^-{\beta_n}&  Z_n &\\
D_n \ar@<-0.3ex>@{^{(}->}[u]\ar@<-0.3ex>@{^{(}->}[r]^-{\beta_{D,n}}& \mathscr{D}_n \ar[u]&  
}
\]
\item\ There exist a compatible system of liftings of  Frobenius endomorphisms $\{F_{Z_n}:(Z_n,M_{Z_n})\rightarrow (Z_n,M_{Z_n})\}$ and $\{F_{\mathscr{D}_n}:(\mathscr{D}_n,M_{\mathscr{D}_n})\rightarrow (\mathscr{D}_n,M_{\mathscr{D}_n})\}$ for each $n \in \mathbb{N}$ {\rm(\cite[p.71, (2.1.1)--(2.1.3)]{Tsu0})}.
\item\  The systems $\{F_{Z_n}:(Z_n,M_{Z_n})\rightarrow (Z_n,M_{Z_n})\}$ and $\{F_{\mathscr{D}_n}:(\mathscr{D}_n,M_{\mathscr{D}_n})\rightarrow (\mathscr{D}_n,M_{\mathscr{D}_n})\}$ fit into the following commutative diagram for for each $n \in \mathbb{N}$: 
\[
\xymatrix@M=10pt{
(Z_n,M_{Z_n}) \ar[r]^-{F_{Z_n}}& (Z_n,M_{Z_n})&\\
(\mathscr{D}_n,M_{\mathscr{D}_n}) \ar[u]\ar[r]^-{F_{\mathscr{D}_n}}&(\mathscr{D}_n,M_{\mathscr{D}_n})  \ar[u]&  
}\]
\end{itemize}
\end{Asum}
 Let $(\mathscr{E}_n, M_{\mathscr{E}_n})$ be the PD-envelope of $(X_n,M_{X_n})$ in $(Z_n,M_{Z_n})$ which is compatible with the canonical PD-structure on the ideal ($p$)$\subset \mathbb{Z}/p^n\mathbb{Z}$. Let $(\mathscr{E}_{n,D}, M_{\mathscr{E}_{n,D}})$ be the PD-envelope of $(D_n,M_{D_n})$ in $(\mathscr{D}_n, M_{\mathscr{D}_n})$. By the assumption that $D$ is flat over $\Spec(\mathscr{O}_K)$, we have $\mathscr{E}_{n,D} \cong \mathscr{E}_{n} \otimes_{Z_n} \mathscr{D}_n$.
The morphism $F_{Z_n}$ induces a lifting of Frobenius $F_{\mathscr{E}_n}$ of $(\mathscr{E}_n, M_{\mathscr{E}_n})$.
 For $i \geq 1$, let $ J_{\mathscr{E}_n}^{[i]} \subset \mathscr{O}_{\mathscr{E}_n}$ be the $i$-th devided power of the ideal $J_{\mathscr{E}_n}:=\Ker \Big(\mathscr{O}_{\mathscr{E}_n}\rightarrow \mathscr{O}_{X_n}\Big)$. For $i \leq 0$, we put $J_{\mathscr{E}_n}^{[i]}:=\mathscr{O}_{\mathscr{E}_n}$.
We put \[\omega_{Z_n}^q:=\Omega_{Z_n/W_n}^q\big(\log M_{Z_n}\big),\quad \omega_{Z_n|\mathscr{D}_n}^{q}:=\omega_{Z_n}^{q}\otimes_{\mathscr{O}_{Z_n}} \mathscr{O}_{Z_n}(-\mathscr{D}_n)\quad(q\geq 0) \] which are locally free $\mathscr{O}_{Z_n}$-modules. 

Let us recall that the syntomic complex is defined as follows:
\[\mathcal{S}_n(q)_{(X_n,M_{X_n}),(Z_n,M_{Z_n})}:={\rm Cone}\big(1-p^{-q}\varphi:  \omega_{Z_n}^{\text{\large$\cdot$}}\otimes_{\mathscr{O}_{Z_n}} J_{\mathscr{E}_n}^{[q-\text{\large$\cdot$}]}\rightarrow   \omega_{Z_n}^{\text{\large$\cdot$}}\otimes _{\mathscr{O}_{Z_n}}\mathscr{O}_{\mathscr{E_n}} \big)[-1]
\]
for $0 \leq q \leq p-1$ (cf.\ \cite{Tsu0}, \cite{Tsu1}, \cite{Tsu2}{\rm).

\begin{Prop}\label{distmn}{\rm(\cite[Proposition 2.2.10]{AS})} For $m, n \geq 1$, there is a short exact sequence of complexes on $X_{1, \acute{e}t}$
\[0\longrightarrow \mathcal{S}_n(q)_{X_n, Z_n}\overset{\times p^m}{\longrightarrow} \mathcal{S}_{n+m}(q)_{X_{n+m}, Z_{n+m}}\longrightarrow \mathcal{S}_m(q)_{X_m, Z_m}\longrightarrow 0,\] 
where we abbreviate $(X_n,M_{X_n}),(Z_n,M_{Z_n})$ to $X_n, Z_n$.
\end{Prop}

\noindent
\begin{Def}{\rm(}syntomic complex with modulus, sufficiently local case{\rm)}\\
We assume $q \leq p-1$. We define
\begin{multline}
{\mathscr{S}_n(q)}^{loc}_{X|D, (Z_n,M_{Z_n}),(\mathscr{D}_n,M_{\mathscr{D}_n})}\\:={\rm Cone}\big(\mathcal{S}_n(q)_{(X_n,M_{X_n}),(Z_n,M_{Z_n})}\longrightarrow \mathcal{S}_n(q)_{(D_n,M_{D_n}),(\mathscr{D}_n,M_{\mathscr{D}_n})}\big)[-1] \in D^+(X_{1,\acute{e}t}, \mathbb{Z}/p^n\mathbb{Z})
\end{multline} under the \textbf{Assumption} \ref{Asum}.
\end{Def}
\begin{Rem} If we define the syntomic complex $\mathcal{S}_n(q)_{(X_n,M_{X_n}),(Z_n,M_{Z_n})}$, we do not need the assumption that  $\mathscr{D}_n$ is an effective Cartier divisor on $Z_n$ for the global construction of the syntomic complex below. If we calculate the syntomic complex in the local situation, we need the assumption that  $\mathscr{D}_n$ is an effective Cartier divisor on $Z_n$.
\end{Rem}
\begin{Lem} The syntomic complex with modulus 
${\mathscr{S}_n(q)}^{loc}_{X|D, (Z_n,M_{Z_n}),(\mathscr{D}_n,M_{\mathscr{D}_n})}$ is independent of the choice of $(Z_n,M_{Z_n})$ and $(\mathscr{D}_n,M_{\mathscr{D}_n})$. 
\end{Lem}
\begin{proof}
Choose another $(Z'_n,M_{Z'_n})$ and $(\mathscr{D}'_n,M_{\mathscr{D}'_n})$, and consider the following commutative diagrams:
\[
\xymatrix@M=10pt{
(X_n, M_{X_n}) \ar[d]^{\id} \ar@<-0.3ex>@{^{(}->}[r]^-{\beta'_n}&  (Z'_n, M_{Z'_n}) \ar[d]&\\
(X_n, M_{X_n})  \ar@<-0.3ex>@{^{(}->}[r]^-{\beta_{n}}& (Z_n, M_{Z_n}), &  
}  \xymatrix@M=10pt{
(D_n, M_{D_n}) \ar[d]^{\id} \ar@<-0.3ex>@{^{(}->}[r]^-{\beta'_{D,n}}& (\mathscr{D}'_n,M_{\mathscr{D}'_n})\ar[d]&\\
(D_n, M_{D_n})  \ar@<-0.3ex>@{^{(}->}[r]^-{\beta_{D,n}}& (\mathscr{D}_n,M_{\mathscr{D}_n}),&  
}
\]
where $\beta'_n$, $\beta_n$, $\beta'_{D,n}$ and $\beta_{D,n}$ are exact closed immersions. Let $(\mathscr{E}_{X,n}, M_{\mathscr{E}_{X,n}})$, $(\mathscr{E}'_{X,n}, M_{\mathscr{E}'_{X,n}})$ (resp. $(\mathscr{E}_{D,n}, M_{\mathscr{E}_{D,n}})$, $(\mathscr{E}'_{D,n}, M_{\mathscr{E}'_{D,n}})$) denote the PD-envelopes of $\beta_n$ and $\beta'_n$ (resp. $\beta_{D,n}$ and $\beta'_{D,n}$). From \cite[Corollary 1.11]{Tsu2}, we have quasi-isomorphisms
\[\mathcal{S}_n(q)_{(X_n,M_{X_n}),(Z_n,M_{Z_n})}  \overset{\mathrm{qis}}{\cong} \mathcal{S}_n(q)_{(X_n,M_{X_n}),(Z'_n,M_{Z'_n})}, \]
\[ \mathcal{S}_n(q)_{(D_n,M_{D_n}),(\mathscr{D}_n,M_{\mathscr{D}_n})} \overset{\mathrm{qis}}{\cong}  \mathcal{S}_n(q)_{(D_n,M_{D_n}),(\mathscr{D}'_n,M_{\mathscr{D}'_n})}. \]
We put $A:=\mathcal{S}_n(q)_{(X_n,M_{X_n}),(Z_n,M_{Z_n})}$, $A':= \mathcal{S}_n(q)_{(X_n,M_{X_n}),(Z'_n,M_{Z'_n})}$, $B:= \mathcal{S}_n(q)_{(D_n,M_{D_n}),(\mathscr{D}_n,M_{\mathscr{D}_n})}$ and $B':= \mathcal{S}_n(q)_{(D_n,M_{D_n}),(\mathscr{D}'_n,M_{\mathscr{D}'_n})}$ and morphisms $u: A^{\cdot} \rightarrow B^{\cdot}$ and $v: A'^{\cdot} \rightarrow B'^{\cdot}$. Then we have the following commutative diagram of the long exact sequences 
{\tiny\[
\xymatrix@=30pt{\cdots \ar[r] & \mathcal{H}^i(A^{\cdot})\ar[d]^-{\cong} \ar[r] & \mathcal{H}^i(B^{\cdot})\ar[d]^-{\cong} \ar[r] &\mathcal{H}^i(\Cone(u))\ar[d] \ar[r] & \mathcal{H}^{i+1}(A^{\cdot})\ar[d]^-{\cong}\ar[r]& \mathcal{H}^{i+1}(B^{\cdot})\ar[d]^-{\cong} \ar[r]& \cdots\\
\cdots \ar[r] & \mathcal{H}^i(A'^{\cdot}) \ar[r] & \mathcal{H}^i(B'^{\cdot}) \ar[r] &\mathcal{H}^i(\Cone(v)) \ar[r] & \mathcal{H}^{i+1}(A'^{\cdot})\ar[r]& \mathcal{H}^{i+1}(B'^{\cdot}) \ar[r]& \cdots.} 
\]}
Thus we have an isomorphism $\mathcal{H}^i(\Cone(u)) \cong \mathcal{H}^i(\Cone(v))$ i.e.,
a quasi-isomorphism \[\mathscr{S}_n(q)^{loc}_{X|D, (Z_n,M_{Z_n}),(\mathscr{D}_n,M_{\mathscr{D}_n})}  \overset{\mathrm{qis}}{\cong}  \mathscr{S}_n(q)^{loc}_{X|D, (Z'_n,M_{Z'_n}),(\mathscr{D}'_n,M_{\mathscr{D}'_n})}.\]
This completes the proof.
\end{proof}

\subsection{Construction in the general case} We keep the \textbf{Assumption 2.1}.
In the general case,  we define $\mathscr{S}_n(q)_{X|D} \in D^+\big(X_{1, \acute{e}t}, \mathbb{Z}/p^n\mathbb{Z}\big)$ by gluing the local complexes:        
We choose a hyper-covering $X^{\bullet}$ of $X$(resp.\;$D^{\bullet}$ of $D$) in the \'etale topology and a closed immersions $\beta_n^{\bullet}:(X_n^{\bullet}, M_{X_n^{\bullet}})\longrightarrow (Z_n^{\bullet}, M_{Z_n^{\bullet}})$ (resp.\;${\beta_{n,D}}^{\bullet}:(D_n^{\bullet}, M_{D_n^{\bullet}})\longrightarrow (\mathscr{D}_n^{\bullet}, M_{\mathscr{D}_n^{\bullet}})$), with the property that, for each integer $\mu \geq 0$, ${\beta_n}^{\mu}$(resp.\\ \;${\beta_{n,D}}^{\mu}$) is an immersion of log schemes and $(Z^{\mu}_n, M_{Z^{\mu}_n})$ (resp\;$(\mathscr{D}^{\mu}_n, M_{\mathscr{D}^{\mu}_n})$) is a smooth log scheme over $W$, in such a way that  there exists a compatible system of liftings of frobenius $\{F_{{Z^{\bullet}}_n}:(Z_n^{\bullet}, M_{Z_n^{\bullet}})\rightarrow (Z_n^{\bullet}, M_{Z_n^{\bullet}})\}$ (resp.\;$\{F_{{\mathscr{D}^{\bullet}}_n}:(\mathscr{D}_n^{\bullet}, M_{\mathscr{D}_n^{\bullet}})\rightarrow (\mathscr{D}_n^{\bullet}, M_{\mathscr{D}_n^{\bullet}})\}$).

\begin{Lem}\label{Lemfunc}{\rm (The functoriality of $\mathscr{S}_n(q)^{loc}_{X|D, (Z_n,M_{Z_n}),(\mathscr{D}_n,M_{\mathscr{D}_n})} $)}
Suppose that we are given two data \[(X, D, Z_n, \mathscr{D}_n, \beta_n, \beta_{D,n})\ \ {\rm and}\ \  (X', D', Z'_n, \mathscr{D}'_n, \beta'_n, \beta_{D',n})\] in Assumption \ref{Asum} which fit into the following commutative squares and a cartesian square:
\[
\xymatrix@M=10pt{
 ({X'}_n, M_{X'_n}) \ar[d]^{f_X} \ar@<-0.3ex>@{^{(}->}[r]^-{\beta'_n}&  (Z'_n, M_{Z'_n}) \ar[d]&\\
(X_n, M_{X_n})  \ar@<-0.3ex>@{^{(}->}[r]^-{\beta_{n}}& (Z_n, M_{Z_n}), &  
}  \xymatrix@M=10pt{
({D'}_n, M_{D'_n}) \ar[d]^{f_D} \ar@<-0.3ex>@{^{(}->}[r]^-{\beta'_{D,n}}& (\mathscr{D}'_n,M_{\mathscr{D}'_n})\ar[d]&\\
(D_n, M_{D_n})  \ar@<-0.3ex>@{^{(}->}[r]^-{\beta_{D,n}}& (\mathscr{D}_n,M_{\mathscr{D}_n}),&  
}
\]
\[
\xymatrix@M=10pt{
({D'}_n, M_{D'_n}) \ar@{}[rd]|{\square} \ar[d]^{f_D} \ar@<-0.3ex>@{^{(}->}[r]&({X'}_n, M_{X'_n})\ar[d]^-{f_X}&\\
(D_n, M_{D_n})  \ar@<-0.3ex>@{^{(}->}[r]& (X_n, M_{X_n}),&  
} \]
where we assume that $f_X : X'_n\rightarrow X_n$ and $f_D: D'_n \rightarrow D_n$ are \'etale, $f_X^*M_{X_n}\cong M_{{X'}_n}$ and that $f_D^*M_{D_n}\cong M_{{D'}_n}$.
Then the natural homomorphism of complexes:
\[(+)\ \ f_X^{*}\mathscr{S}_n(q)^{loc}_{X|D, (Z_n,M_{Z_n}),(\mathscr{D}_n,M_{\mathscr{D}_n})} \longrightarrow \mathscr{S}_n(q)^{loc}_{X'|D', ({Z'}_n,M_{{Z'}_n}),({\mathscr{D}'}_n,M_{{\mathscr{D}'}_n})} \]
is a quasi-isomorphism for any $q \geq 0$.
\end{Lem}
\begin{proof} We have the following quasi-isomorphisms by using \cite[Corollary 1.11]{Tsu2}:
\[f_X^*\mathcal{S}_n(q)_{(X_n,M_{X_n}),(Z_n,M_{Z_n})}\cong \mathcal{S}_n(q)_{({X'}_n,M_{{X'}_n}),({Z'}_n,M_{{Z'}_n})},\]
\[f_X^*\mathcal{S}_n(q)_{(D_n,M_{D_n}),(\mathscr{D}_n,M_{\mathscr{D}_n})} \cong f_D^*\mathcal{S}_n(q)_{(D_n,M_{D_n}),(\mathscr{D}_n,M_{\mathscr{D}_n})}\cong \mathcal{S}_n(q)_{({D'}_n,M_{{D'}_n}),({\mathscr{D}'}_n,M_{{\mathscr{D}'}_n})}.\]
Hence we obtain the quasi-isomorphism $(+)$.
\end{proof}
We obtain the complex ${\mathscr{S}_n(q)}^{loc}_{X^{\bullet}|D^{\bullet}, (Z^{\bullet}_n,M_{Z^{\bullet}_n}),(\mathscr{D}^{\bullet}_n,M_{\mathscr{D}^{\bullet}_n})}$ on $X^{\bullet}_{1,\acute{e}t}$ by the functoriality of the local syntomic complex with modulus from the above Lemma \ref{Lemfunc}.
\begin{Def}{\rm(}syntomic complex with modulus, the general case; cf. \cite[p. 540]{Tsu2} {\rm)} We define the syntomic complex with modulus  $\mathscr{S}_n(q)_{X|D}$  to be the object
\[ R\theta_*\left({\mathscr{S}_n(q)}^{loc}_{X^{\bullet}|D^{\bullet}, (Z^{\bullet}_n,M_{Z^{\bullet}_n}),(\mathscr{D}^{\bullet}_n,M_{\mathscr{D}^{\bullet}_n})}\right)\]
of $D^{+}(X_{1,\acute{e}t}, \mathbb{Z}/p^n\mathbb{Z})$, where $\theta$ denotes the canonical morphism of topoi $(X_1^{\bullet})^{\tilde{}}_{\acute{e}t}\longrightarrow X_{1,\acute{e}t}$.
 \end{Def} 
 
 \begin{Prop}
 The syntomic complex with modulus $\mathscr{S}_n(q)_{X|D}$ is independent of the choice of hyper coverings $X^{\bullet}$ and $D^{\bullet}$ up to a canonical isomorphism.
 \end{Prop}
 \begin{proof}
 If we choose another $X'^{\bullet}$, $\beta_n^{\bullet}: (X'^{\bullet}_n,M_{X'^{\bullet}_n}) \hookrightarrow(Z'^{\bullet}_n,M_{Z'^{\bullet}_n})$, $\beta_{n,D}^{\bullet}: (D_n'^{\bullet},M_{D_n'^{\bullet}}) \hookrightarrow({\mathscr{D}_n^{'\bullet}},M_{{\mathscr{D}_n^{'\bullet}}})$, $\{F_{Z_n}\}_n$ and $\{F_{\mathscr{D}_n}\}_n$, then by taking the fiber products
 \[X''^{\bullet}:=X^{\bullet}\times_{X}X'^{\bullet},\; (Z''^{\bullet}_n,M_{Z''^{\bullet}_n}):=(Z^{\bullet}_n,M_{Z^{\bullet}_n})\times_{\mathbb{Z}/p^n\mathbb{Z}} (Z'^{\bullet}_n,M_{Z'^{\bullet}_n})\]
\[D''^{\bullet}:=D^{\bullet}\times_{D}D'^{\bullet},\;({\mathscr{D}^{''\bullet}}_n,M_{{\mathscr{D}^{''\bullet}}_n}):=({\mathscr{D}^{\bullet}}_n,M_{{\mathscr{D}^{\bullet}}_n})\times_{\mathbb{Z}/p^n\mathbb{Z}} ({\mathscr{D}_n^{'\bullet}},M_{{\mathscr{D}_n^{'\bullet}}}),\] 
\[F_{Z''^{\bullet}_n}:=F_{Z^{\bullet}_n}\times F_{Z'^{\bullet}_n},\;F_{{\mathscr{D}^{''\bullet}}_n}:=F_{{\mathscr{D}^{\bullet}}_n}\times F_{{\mathscr{D}^{'\bullet}}_n}.\]
We put  $\pr$, $\pr'$, $\theta'$, $\theta''$ the canonical morphism of topoi
\[({X''}_1^{\bullet})^{\tilde{}}_{\acute{e}t}\longrightarrow (X_1^{\bullet})^{\tilde{}}_{\acute{e}t},\;({X''}_1^{\bullet})^{\tilde{}}_{\acute{e}t}\longrightarrow ({X'}_1^{\bullet})^{\tilde{}}_{\acute{e}t},\]
\[({X'}_1^{\bullet})^{\tilde{}}_{\acute{e}t}\longrightarrow ({X_1}')^{\tilde{}}_{\acute{e}t},\;({X''}_1^{\bullet})^{\tilde{}}_{\acute{e}t}\longrightarrow ({X_1}'')^{\tilde{}}_{\acute{e}t}.\]
By using \cite[Corollary1.11]{Tsu2}, we obtain canonical quasi-isomorphisms
\[\pr^{-1}\mathcal{S}_n(r)_{(X_n^{\bullet},M_{X_n^{\bullet}}),(Z_n^{\bullet}, M_{Z_n^{\bullet}})}\longrightarrow \mathcal{S}_n(r)_{({X''}_n^{\bullet},M_{{X''}_n^{\bullet}}),({Z''}_n^{\bullet}, M_{{Z''}_n^{\bullet}})},\]
\[{\pr'}^{-1}\mathcal{S}_n(r)_{({X'}_n^{\bullet},M_{{X'}_n^{\bullet}}),({Z'}_n^{\bullet}, M_{{Z'}_n^{\bullet}})}\longrightarrow \mathcal{S}_n(r)_{({X''}_n^{\bullet},M_{{X''}_n^{\bullet}}),({Z''}_n^{\bullet}, M_{{Z''}_n^{\bullet}})},\]
 \[\pr^{-1}\mathcal{S}_n(r)_{(D_n^{\bullet},M_{D_n^{\bullet}}),(\mathscr{D}_n^{\bullet}, M_{\mathscr{D}_n^{\bullet}})}\longrightarrow \mathcal{S}_n(r)_{({D''}_n^{\bullet},M_{{D''}_n^{\bullet}}),({\mathscr{D}''}_n^{\bullet}, M_{{\mathscr{D}'}_n^{\bullet}})},\]
\[{\pr'}^{-1}\mathcal{S}_n(r)_{({D'}_n^{\bullet},M_{{D'}_n^{\bullet}}),({\mathscr{D}'}_n^{\bullet}, M_{{\mathscr{D}'}_n^{\bullet}})}\longrightarrow \mathcal{S}_n(r)_{({D''}_n^{\bullet},M_{{D''}_n^{\bullet}}),({\mathscr{D}''}_n^{\bullet}, M_{{\mathscr{D}''}_n^{\bullet}})}\]
and a canonical quasi-isomorphisms
\begin{align*}
(\alpha)\ \ R\theta_*\Big(\mathcal{S}_n(r)_{(X_n^{\bullet},M_{X_n^{\bullet}}),(Z_n^{\bullet}, M_{Z_n^{\bullet}})}\Big)&\xrightarrow{\cong} R{\theta''}_*\Big(\mathcal{S}_n(r)_{({X''}_n^{\bullet},M_{{X''}_n^{\bullet}}),({Z''}_n^{\bullet}, M_{{Z''}_n^{\bullet}})}\Big)\\
&\xleftarrow{\cong} R{\theta'}_*\Big(\mathcal{S}_n(r)_{({X'}_n^{\bullet},M_{{X'}_n^{\bullet}}),({Z'}_n^{\bullet}, M_{{Z'}_n^{\bullet}})}\Big),
\end{align*}
\begin{align*}
(\beta)\ \ R\theta_*\Big(\mathcal{S}_n(r)_{(D_n^{\bullet},M_{D_n^{\bullet}}),(\mathscr{D}_n^{\bullet}, M_{\mathscr{D}_n^{\bullet}})}\Big)&\xrightarrow{\cong} R{\theta''}_*\Big(\mathcal{S}_n(r)_{({D''}_n^{\bullet},M_{{D''}_n^{\bullet}}),({\mathscr{D}''}_n^{\bullet}, M_{{\mathscr{D}''}_n^{\bullet}})}\Big)\\
&\xleftarrow{\cong} R{\theta'}_*\Big(\mathcal{S}_n(r)_{({D'}_n^{\bullet},M_{{D'}_n^{\bullet}}),({\mathscr{D}'}_n^{\bullet}, M_{{\mathscr{D}'}_n^{\bullet}})}\Big),
\end{align*}
Hence we obtain 
\begin{multline*}
R\theta_*\Big({\mathscr{S}_n(q)}^{loc}_{X^{\bullet}|D^{\bullet},(Z^{\bullet}_n,M_{Z^{\bullet}_n}),(\mathscr{D}^{\bullet}_n,M_{\mathscr{D}^{\bullet}_n}) }\Big)\xrightarrow{\cong} R{\theta''}_*\Big({\mathscr{S}_n(q)}^{loc}_{{X''}^{\bullet}|{D''}^{\bullet},({Z''}^{\bullet}_n,M_{{Z''}^{\bullet}_n}),({\mathscr{D}''}^{\bullet}_n,M_{{\mathscr{D}''}^{\bullet}_n}) }\Big)\\
\xleftarrow{\cong} R{\theta'}_*\Big({\mathscr{S}_n(q)}^{loc}_{X'^{\bullet}|D'^{\bullet},({Z'}^{\bullet}_n,M_{{Z'}^{\bullet}_n}),({\mathscr{D}'}^{\bullet}_n,M_{{\mathscr{D}'}^{\bullet}_n}) }\Big).
\end{multline*}
This quasi-isomorphism satisfies the transitivity because the quasi-isomorphisms $(\alpha)$ and $(\beta)$ satisfy the transitivity (\cite[p.542, l.7]{Tsu2}). This completes the proof.  \end{proof}
 
\begin{Lem}{\rm (cf. Proposition \ref{distmn})}\label{dist} Let $n \geq 1$ be an integer. We have a distinguished triangle
\[\mathscr{S}_n(q)_{X|D}\overset{\times p}{\longrightarrow} \mathscr{S}_{n+1}(q)_{X|D}\longrightarrow \mathscr{S}_1(q)_{X|D}\longrightarrow \mathscr{S}_n(q)_{X|D}[1].\]
\end{Lem}
\begin{proof}It is enough to show the existence of the following distinguished triangle
\begin{multline*}  \mathscr{S}_n(q)^{loc}_{X|D, (Z_n,M_{Z_n}),(\mathscr{D}_n,M_{\mathscr{D}_n})}\overset{\times p}{\longrightarrow} \mathscr{S}_{n+1}(q)^{loc}_{X|D, (Z_{n+1},M_{Z_{n+1}}),(\mathscr{D}_{n+1},M_{\mathscr{D}_{n+1}})}\\ \longrightarrow \mathscr{S}_1(q)^{loc}_{X|D, (Z_1,M_{Z_1}),(\mathscr{D}_1,M_{\mathscr{D}_1})}\longrightarrow \mathscr{S}_n(q)^{loc}_{X|D, (Z_n,M_{Z_n}),(\mathscr{D}_n,M_{\mathscr{D}_n})}[1].\end{multline*} For $m=1$ in Proposition \ref{distmn}, we have the following distinguished triangles:
\[\mathcal{S}_n(q)_{X_n, Z_n}\overset{\times p}{\longrightarrow} \mathcal{S}_{n+1}(q)_{X_{n+1}, Z_{n+1}}\longrightarrow \mathcal{S}_1(q)_{X_1, Z_1}\longrightarrow \mathcal{S}_n(q)_{X_n, Z_n}[1],\]
\[ \mathcal{S}_n(q)_{D_n, \mathscr{D}_n}\overset{\times p}{\longrightarrow} \mathcal{S}_{n+1}(q)_{D_{n+1}, \mathscr{D}_{n+1}}\longrightarrow \mathcal{S}_1(q)_{D_1, \mathscr{D}_1}\longrightarrow \mathcal{S}_n(q)_{D_n, \mathscr{D}_n}[1],\]
where we abbreviate $(X_n,M_{X_n}),(Z_n,M_{Z_n})$ to $X_n, Z_n$ and $(D_n,M_{D_n}),(\mathscr{D}_n,M_{\mathscr{D}_n})$ to $D_n, \mathscr{D}_n$.
Here we put $A^{\cdot}_n:=\mathcal{S}_n(q)_{X_n, Z_n}$ and $B^{\cdot}_n:= \mathcal{S}_n(q)_{D_n, \mathscr{D}_n}$ for simplicity.
We put $C^{\cdot}_n:=\Cone(A^{\cdot}_n \rightarrow B^{\cdot}_n)(=\mathscr{S}_n(q)^{loc}_{X|D, (Z_n,M_{Z_n}),(\mathscr{D}_n,M_{\mathscr{D}_n})}[1])$. We have \[\Cone(C^{\cdot}_n \overset{\times p}{\longrightarrow} C^{\cdot}_{n+1})\cong\Cone\left(\Cone(A^{\cdot}_n \longrightarrow A^{\cdot}_{n+1})\rightarrow \Cone(B^{\cdot}_n \longrightarrow B^{\cdot}_{n+1})\right)\cong C^{\cdot}_1.\] 
Then we have the distinguished triangle \[C^{\cdot}_n \overset{\times p}{\longrightarrow} C^{\cdot}_{n+1} \longrightarrow C^{\cdot}_1 \longrightarrow C^{\cdot}_n[1].\] 
This completes the proof. \end{proof}

\noindent

\subsection{Definition of another syntomic complex with modulus}
We assume the following assumption when we use another syntomic complex with modulus $s_n(r)_{X|D}$, which is defined below.
\begin{Asum}\label{Asum*}\mbox{}

\begin{itemize}
\item There exist exact closed immersions \[\beta_n: (X_n,M_{X_n}) \hookrightarrow(Z_n,M_{Z_n})\ \ {\rm and}\ \ \beta_{n,D}: (D_n,M_{D_n}) \hookrightarrow(\mathscr{D}_n,M_{\mathscr{D}_n})\] of log schemes for $n \geq 1$ such that $(Z_n,M_{Z_n})$ and $(\mathscr{D}_n,M_{\mathscr{D}_n})$ are smooth over $W:=W(k)$, and such that the following diagram is cartesian :
\[
\xymatrix@M=10pt{
X_n \ar@{}[rd]|{\square} \ar@<-0.3ex>@{^{(}->}[r]^-{\beta_n}&  Z_n &\\
D_n \ar@<-0.3ex>@{^{(}->}[u]\ar@<-0.3ex>@{^{(}->}[r]^-{\beta_{D,n}}& \mathscr{D}_n \ar@<-0.3ex>@{^{(}->}[u]&  
}
\]

\item There exist a compatible system of liftings of  Frobenius endomorphisms $\{F_{Z_n}:(Z_n,M_{Z_n})\rightarrow (Z_n,M_{Z_n})\}$ and $\{F_{\mathscr{D}_n}:(\mathscr{D}_n,M_{\mathscr{D}_n})\rightarrow(\mathscr{D}_n,M_{\mathscr{D}_n})\}$ for each $n \in \mathbb{N}$ {\rm(\cite[p.71, (2.1.1)--(2.1.3)]{Tsu0})}.
\item\  The systems $\{F_{Z_n}:(Z_n,M_{Z_n})\rightarrow (Z_n,M_{Z_n})\}$ and $\{F_{\mathscr{D}_n}:(\mathscr{D}_n,M_{\mathscr{D}_n})\rightarrow (\mathscr{D}_n,M_{\mathscr{D}_n})\}$fit into the following commutative diagram for for each $n \in \mathbb{N}$: 
\[
\xymatrix@M=10pt{
(Z_n,M_{Z_n}) \ar[r]^-{F_{Z_n}}& (Z_n,M_{Z_n})&\\
(\mathscr{D}_n,M_{\mathscr{D}_n}) \ar[u]\ar[r]^-{F_{\mathscr{D}_n}}&(\mathscr{D}_n,M_{\mathscr{D}_n})  \ar[u]&  
}\]

\item  $\mathscr{D}_n$ is an effective Cartier divisor on $Z_n$ such that $\beta_{D,n}^*\mathscr{D}_n=D_n$ and $F_{Z_n}$ which induces a morphism $\mathscr{D}_n \longrightarrow \mathscr{D}_n$.
\end{itemize}
\end{Asum}

 For an effective Cartier divisor $\mathscr{D}_n:=\sum_{\lambda \in \Lambda} m_\lambda \mathscr{D}_\lambda$,  we define $\mathscr{D}'_n:=\sum_{\lambda \in \Lambda} m'_\lambda \mathscr{D}_\lambda$, where $m'_{\lambda}:=\{l \in \mathbb{N} \mid p\cdot l \geq m_{\lambda}\}$.

We denote $\varphi:J_{\mathscr{E}_n}^{[r]}\otimes_{\mathscr{O}_{Z_n}} \mathscr{O}_{Z_n}(-\mathscr{D}_n) \longrightarrow J_{\mathscr{E}_n}^{[r]}\otimes_{\mathscr{O}_{Z_n}} \mathscr{O}_{Z_n}(-\mathscr{D}_n)\;; a\otimes b \mapsto \varphi(a)\otimes \varphi(b)$, where the homomorphism $\varphi$ induced by $F_{\mathscr{E}_n}$. 
We will define the Frobenius morphism ``devided by $p^r$''  $p^{-r}\varphi$ (or $\varphi_r$) $: J_{\mathscr{E}_n}^{[r-\text{\large$\cdot$}]}\otimes_{\mathscr{O}_{Z_n}} \mathscr{O}_{Z_n}(-\mathscr{D}_n)) \rightarrow \mathscr{O}_{\mathscr{E}_n}\otimes_{\mathscr{O}_{Z_n}} \mathscr{O}_{Z_n}(-\mathscr{D}_n)$ in the following:

We have \begin{equation}\label{p^r}\varphi(J_{\mathscr{E}_n}^{[r]}\otimes_{\mathscr{O}_{Z_n}} \mathscr{O}_{Z_n}(-\mathscr{D}_n))\subset p^r (\mathscr{O}_{\mathscr{E}_n}\otimes_{\mathscr{O}_{Z_n}} \mathscr{O}_{Z_n}(-\mathscr{D}_n)),\end{equation} (cf. \cite[I, Lemma 1.3 (1)]{Ka1}).
On the other hand, $J_{\mathscr{E}_n}^{[r]}$ is flat over $\mathbb{Z}/p^n\mathbb{Z}$ and \[\big(J_{\mathscr{E}_{n+1}}^{[r]}\otimes \mathscr{O}_{Z_{n+1}}(-\mathscr{D}_{n+1})\big) \otimes \mathbb{Z}/p^n\mathbb{Z}\cong J_{\mathscr{E}_{n}}^{[r]}\otimes  \mathscr{O}_{Z_{n}}(-\mathscr{D}_{n})\] for every $n \geq 1$ and $r \geq 0$. Hence, for $0 \leq r \leq p-1$, there exists a unique homomorphism \[\varphi_r:J_{\mathscr{E}_n}^{[r]}\otimes_{\mathscr{O}_{Z_n}} \mathscr{O}_{Z_n}(-\mathscr{D}_n)\rightarrow\mathscr{O}_{\mathscr{E}_n}\otimes_{\mathscr{O}_{Z_n}} \mathscr{O}_{Z_n}(-\mathscr{D}_n)\] which makes the following diagram commute:
\[
\xymatrix@M=10pt{
J_{\mathscr{E}_{n+r}}^{[r]}\otimes_{\mathscr{O}_{Z_{n+r}}} \mathscr{O}_{Z_{n+r}}(-\mathscr{D}_{n+r})  \ar[r]^{\varphi} \ar[d]&{\mathscr{O}}_{\mathscr{E}_{n+r}}\otimes_{\mathscr{O}_{Z_{n+r}}} \mathscr{O}_{Z_{n+r}}(-\mathscr{D}_{n+r}) \\
J_{\mathscr{E}_{n}}^{[r]}\otimes_{\mathscr{O}_{Z_{n}}} \mathscr{O}_{Z_{n}}(-\mathscr{D}_{n}) \ar[r]^{\varphi_r}&{\mathscr{O}}_{\mathscr{E}_{n}}\otimes_{\mathscr{O}_{Z_{n}}} \mathscr{O}_{Z_{n}}(-\mathscr{D}_{n}) \ar[u]^-{p^r}.
}
\]
From the fact that \[d\varphi(\omega_{Z_n/W_n}^1)\subset p\cdot\omega_{Z_n/W_n}^1\quad(n \in \mathbb{N},\;n>0),\]
we can define a frobenius ``divided by $p$'' \begin{equation}\label{frob}\frac{d\varphi}{p}:\omega_{Z_n/W_n}^1\longrightarrow \omega_{Z_n/W_n}^1.\end{equation}
\noindent
\begin{Def}{\rm(}another syntomic complex with modulus, sufficiently local case{\rm)}\\
We assume $r \leq p-1$. We define
\[
s_n(r)_{X|D, (Z_n,M_{Z_n}), (\mathscr{D}_n,M_{\mathscr{D}_n})}:={\rm Cone}\big(1-\varphi_{r}:  J_{\mathscr{E}_n}^{[r-\text{\large$\cdot$}]}\otimes_{\mathscr{O}_{Z_n}}\omega_{Z_n|\mathscr{D}_n}^{\text{\large$\cdot$}} \rightarrow   \mathscr{O}_{\mathscr{E}_n}\otimes_{\mathscr{O}_{Z_n}}\omega_{Z_n|\mathscr{D}_n}^{\text{\large$\cdot$}} \big)[-1],
\]
where $\varphi_r=\varphi_{r-q}\otimes \land^{q} \frac{d\varphi}{p}$ in degree $q$. We denote this complex by $s_n(r)_{X|D}$ for simplicity.
\end{Def}
 
\begin{Lem}\label{quasi-isom} Under \textbf{Assumption} \ref{Asum*},
$s_n(q)_{X|D}$ and ${\mathscr{S}_n(q)}^{loc}_{X|D, (Z_n,M_{Z_n}),(\mathscr{D}_n,M_{\mathscr{D}_n})}$ are naturally quasi-isomorphic.
\end{Lem}
\begin{proof}
By the definition of ${\mathscr{S}_n(q)}^{loc}_{X|D}$, we have a quasi-isomorphism \[\Ker\big(\mathcal{S}_n(q)_{(X_n,M_{X_n}),(Z_n,M_{Z_n})}\longrightarrow \mathcal{S}_n(q)_{(D_n,M_{D_n}),(\mathscr{D}_n,M_{\mathscr{D}_n})}
\big) \overset{\mathrm{qis}}{\cong}{\mathscr{S}_n(q)}^{loc}_{X|D,(Z_n,M_{Z_n}),(\mathscr{D}_n,M_{\mathscr{D}_n})}.\]
We will show that there exists an isomorphism of complexes \[\Ker\big(\mathcal{S}_n(q)_{(X_n,M_{X_n}),(Z_n,M_{Z_n})}\longrightarrow \mathcal{S}_n(q)_{(D_n,M_{D_n}),(\mathscr{D}_n,M_{\mathscr{D}_n})}
\big)  \cong s_n(q)_{X|D}.\]
To show this isomorphism, we will show  the isomorphism
\begin{align*} &\left(J_{\mathscr{E}_n}^{[q-j]}\otimes_{\mathscr{O}_{Z_n}}\omega_{Z_n|\mathscr{D}_n}^{j} \right)\oplus \left( \mathscr{O}_{\mathscr{E}_n}\otimes_{\mathscr{O}_{Z_n}}\omega_{Z_n|\mathscr{D}_n}^{j-1} \right) \\
&\cong \Ker \left( \left(J_{\mathscr{E}_n}^{[q-j]}\otimes_{\mathscr{O}_{Z_n}}\omega_{Z_n}^{j} \right)\oplus \left( \mathscr{O}_{\mathscr{E}_n}\otimes_{\mathscr{O}_{Z_n}}\omega_{Z_n}^{j-1}\right) \longrightarrow  \left(J_{\mathscr{E}_{n,D}}^{[q-j]} \otimes_{\mathscr{O}_{\mathscr{D}_n}}\omega_{\mathscr{D}_n}^{j} \right)\oplus \left( \mathscr{O}_{\mathscr{E}_{n,D}}\otimes_{\mathscr{O}_{\mathscr{D}_n}}\omega_{\mathscr{D}_n}^{j-1} \right) \right) 
\end{align*}
for each degree $j \geq 0$.
It sufficies to show an isomorphism \begin{equation}\label{isom}J_{\mathscr{E}_n}^{[q-j]}\otimes_{\mathscr{O}_{Z_n}} \omega_{Z_n|\mathscr{D}_n}^{j} \xrightarrow{\cong} \Ker\Big(J_{\mathscr{E}_n}^{[q-j]}\otimes_{\mathscr{O}_{Z_n}} \omega_{Z_n}^{j}\longrightarrow J_{\mathscr{E}_n}^{[q-j]}\otimes_{\mathscr{O}_{Z_n}} \omega_{Z_n}^{j}\otimes_{\mathscr{O}_{Z_n}}\mathscr{O}_{\mathscr{D}_n}\Big),\end{equation} where $J_{\mathscr{E}_{n,D}}^{[q-j]} \otimes_{\mathscr{O}_{\mathscr{D}_n}}\omega_{\mathscr{D}_n}^{j}  \cong J_{\mathscr{E}_n}^{[q-j]}\otimes_{\mathscr{O}_{Z_n}} \omega_{Z_n}^{j}\otimes_{\mathscr{O}_{Z_n}}\mathscr{O}_{\mathscr{D}_n}$.
By tensoring $J_{\mathscr{E}_n}^{[q-j]}\otimes_{\mathscr{O}_{Z_n}} \omega_{Z_n}^j$ to the short exact sequence \[0\rightarrow\mathscr{O}_{Z_n}(-\mathscr{D}_n)\rightarrow\mathscr{O}_{Z_n}\rightarrow\mathscr{O}_{\mathscr{D}_n}\rightarrow 0,\] 
we have the exact sequence
\[J_{\mathscr{E}_n}^{[q-j]}\otimes_{\mathscr{O}_{Z_n}} \omega_{Z_n|\mathscr{D}_n}^{j} \rightarrow J_{\mathscr{E}_n}^{[q-j]}\otimes_{\mathscr{O}_{Z_n}} \omega_{Z_n}^{j} \rightarrow J_{\mathscr{E}_n}^{[q-j]}\otimes_{\mathscr{O}_{Z_n}} \omega_{Z_n}^{j}\otimes_{\mathscr{O}_{Z_n}}\mathscr{O}_{\mathscr{D}_n} \rightarrow 0.\]
Thus we have the surjectivity of the morphism \eqref{isom}.
We will prove the injectivity of this morphism. By the smoothness of $Z_n$, $\omega_{Z_n}^{\cdot}$ is locally free sheaf on $\mathscr{O}_{Z_n}$.
Hence it suffices to show that the morphism 
\[(*)\ \ \ J_{\mathscr{E}_n}^{[q-j]}\otimes_{\mathscr{O}_{Z_n}}\mathscr{O}_{Z_n}(-\mathscr{D}_n) \longrightarrow J_{\mathscr{E}_n}^{[q-j]}\]
is injective. It suffices to show that the injectivity of the morphism $(*)$ locally, so we assume that $\mathscr{D}_n$ is defined by the one equation $f$.
Then the inclusion map $\mathscr{O}_{Z_n}(-\mathscr{D}_n) \longrightarrow \mathscr{O}_{Z_n}$ is identified with the map $\mathscr{O}_{Z_n} \xrightarrow{\times f} \mathscr{O}_{Z_n}$. So, we can identify $(*)$ with the map 
\[(**)\ \ \ J_{\mathscr{E}_n}^{[q-j]} \xrightarrow{\times f} J_{\mathscr{E}_n}^{[q-j]}. \]
Since $J_{\mathscr{E}_n}^{[q-j]}$ is the subsheaf of $\mathscr{O}_{\mathscr{E}_n}$, it suffices to show that $\mathscr{O}_{\mathscr{E}_n}\xrightarrow{\times f}\mathscr{O}_{\mathscr{E}_n}$ is injective. The problem is reduced to the case
 \[X_n= \Spec(\mathscr{O}_K/p^n[t_1,\dotsc,t_d]/(t_i\cdots t_d-\pi)),\] \[Z_n=\Spec(W_n[T_0, T_1,\dotsc ,T_d,T_{\infty}]),\] \[D_n=\{t_{i}^{m_{i}}\cdots t_{i+j}^{m_{i+j}}=0\},\] \[\mathscr{D}_n=\{T_{i}^{m_{i}}\cdots T_{i+j}^{m_{i+j}}=T_{\infty}\},\]
 \[\psi:W_n[T_0, T_1,\dotsc ,T_d,T_{\infty}]\rightarrow \mathscr{O}_K/p^n[t_1,\dotsc,t_d]/(t_i\cdots t_d-\pi);\]
\[T_0\mapsto \pi,\;T_i\mapsto t_i \;(1\leq i \leq d),\; T_{\infty}\mapsto 0.\]
In this case, $f=T_{i}^{m_{i}}\cdots T_{i+j}^{m_{i+j}}-T_{\infty}$, and the kernel of the ring homomorphism $\psi$ is \[J:=(T_{\infty},\;T_1\cdots T_d-T_0).\]
We put $g_1:=T_{\infty}$ and $g_2:= T_1\cdots T_d-T_0$.
The affine ring $\mathcal{A}_n$ of $\mathscr{E}_n$ is generated by $\Delta_n:=\{g_1^{[m_1]}\cdot g_2^{[m_2]}\;|\;m_1+m_2=n, m_i \in \mathbb{N} \}$ as a $W_n[T_0, T_1,\dotsc ,T_d,T_{\infty}]$-module. Then any element of $\mathcal{A}_n$ can be written as $\Sigma_{i,\;n \geq 1,\;x \in \Delta_n}\;a_ix$, where $a_i \in W_n[T_0, T_1,\dotsc ,T_d,T_{\infty}]$. The generators $g_1, g_2$ of $J$ are linearly independent on $W_n[T_0, T_1,\dotsc ,T_d,T_{\infty}]$. Thus $\Delta_n$ are basis for $\mathcal{A}_n$ as a $W_n[T_0, T_1,\dotsc ,T_d,T_{\infty}]$-module (\cite[p.31, 1.4.2 and Corollarie 2.3.2 (ii)]{Ber}). Since the polynomial $f$ is a non-zero divisor on $W_n[T_0, T_1,\dotsc ,T_d,T_{\infty}]$, $f$ is a non-zero divisor on $\mathcal{A}_n$. This completes the proof.
\end{proof}

In what follows, we will use the complex $s_n(q)_{X|D}$ when we compute the cohomology sheaf of the syntomic complex with modulus in sufficiently local situation.
By definition, $s_n(q)_{X|D}$ is concentrated in [$0,q$]. Note that $s_n(q)_{X|\emptyset}=\mathcal{S}_n(q)_{(X,M_X)}$,  the syntomic complex defined in \cite{Tsu1}, \cite{Tsu2}.
\begin{Lem}\label{prod s_n}{\rm (cf. \cite{Tsu1}, \cite{Tsu2})}
For $q, q' \geq 0$, there is a morphism in $D^+(X_{1,\acute{e}t}, \mathbb{Z}/p^n\mathbb{Z})$:
\[(\star)\ \   s_n(q)_{X|D} \otimes^{\mathbb{L}} \mathcal{S}_n(q')_{(X_n,M_{X_n}),(Z_n,M_{Z_n})} \longrightarrow s_n(q+q')_{X|D} \] by \[(x, y)\otimes (x', y') \mapsto (xx', (-1)^j  xy' + y \varphi_{q'}(x'))\]
\[(x,y) \in s^j_n(q)_{X|D}=\left(J^{[q-j]}_{\mathscr{E}_n} \otimes_{\mathscr{O}_{Z_n}} \omega^j_{Z_n|\mathscr{D}_n}\right)\oplus \left( \mathscr{O}_{\mathscr{E}_n}\otimes_{\mathscr{O}_{Z_n}} \omega^{j-1}_{Z_n|\mathscr{D}_n} \right)\]
\[(x', y') \in  \mathcal{S}^{j'}_n(q')_{(X_n,M_{X_n}),(Z_n,M_{Z_n})}\left(J^{[q'-j']}_{\mathscr{E}_n} \otimes_{\mathscr{O}_{Z_n}} \omega^{j'}_{Z_n}\right)\oplus \left( \mathscr{O}_{\mathscr{E}_n}\otimes_{\mathscr{O}_{Z_n}} \omega^{j'-1}_{Z_n} \right). \]\end{Lem}
\begin{proof}
There is a product morphism (cf. [Tsu2  $\S$2.2] )
\[ \left(J^{[q-\cdot]}_{\mathscr{E}_n} \otimes_{\mathscr{O}_{Z_n}} \omega^{\cdot}_{Z_n}\right) \oplus \left(J^{[q'-\cdot]}_{\mathscr{E}_n} \otimes_{\mathscr{O}_{Z_n}} \omega^{\cdot}_{Z_n}\right)\longrightarrow \left(J^{[q+q'-\cdot]}_{\mathscr{E}_n} \otimes_{\mathscr{O}_{Z_n}} \omega^{\cdot}_{Z_n}\right).\]
This morphism induces
\[ \left(J^{[q-\cdot]}_{\mathscr{E}_n} \otimes_{\mathscr{O}_{Z_n}} \omega^{\cdot}_{Z_n|\mathscr{D}_n}\right) \oplus \left(J^{[q'-\cdot]}_{\mathscr{E}_n} \otimes_{\mathscr{O}_{Z_n}} \omega^{\cdot}_{Z_n}\right)\longrightarrow \left(J^{[q+q'-\cdot]}_{\mathscr{E}_n} \otimes_{\mathscr{O}_{Z_n}} \omega^{\cdot}_{Z_n|\mathscr{D}_n}\right).\]
Thus we can define the above product morphism $(\star)$.
\end{proof}

\subsection{Construction of the symbol map in local case}\label{symbol section}
In this subsection, we assume \textbf{Assumption} \ref{Asum*}. Let us define a symbol map 
\begin{equation}\label{symb} (1+I_{D_{n+1}})^{\times} \otimes (M_{X_{n+1}}^{gp})^{\otimes{q-1}} \longrightarrow \mathcal{H}^q(\mathscr{S}^{loc}_n(q)_{X|D})
\end{equation}
for $q \geq 0$. Here $I_{D_{n+1}} \subset \mathscr{O}_{X_{n+1}}$ is the definition ideal of $D_{n+1}$ and \[(1+I_{D_{n+1}})^{\times}:=(1+I_{D_{n+1}})\cap \mathscr{O}_{X_{n+1}}^{\times}.\]

We construct a symbol map in the local situation in the following. By taking $R\theta_*$, we immediately obtain its global case.\\
Recall that $(X_n, M_{X_n})$ denotes the reduction mod $p^n$ of $(X, M_{X})$. 
Let $C_{n+1}$ be the complex
\begin{equation}
\big(1+J_{\mathscr{E}_{n+1}})\cap \big(1+\mathscr{O}_{Z_{n+1}}(-\mathscr{D}_{n+1})\big)^{\times}   \longrightarrow \big(1+\mathscr{O}_{Z_{n+1}}(-\mathscr{D}_{n+1})\big)^{\times}\big.
\end{equation}
\[\deg.\;0\quad\quad\quad\quad\quad\quad\quad\quad\quad\quad\quad\quad\quad\quad\quad\quad\deg.\;1\]
We define the morphism of complexes $C_{n+1} \longrightarrow s_n(1)_{X|D}$ by 
\begin{equation}
(1+J_{\mathscr{E}_{n+1}})\cap (1+\mathscr{O}_{Z_{n+1}}(-\mathscr{D}_{n+1}))^{\times} \longrightarrow (s_n(1)_{X|D})^0=J_{\mathscr{E}_{n}}\otimes_{\mathscr{O}_{Z_n}} \mathscr{O}_{Z_n}(-\mathscr{D}_n); 
\end{equation}
\[a \mapsto \log a \mod p^n\]
at degree $0$ and
\begin{equation}
\big(1+\mathscr{O}_{Z_{n+1}}(-\mathscr{D}_{n+1})\big)^{\times} \longrightarrow (s_n(1)_{X|D})^1=\big(\mathscr{O}_{\mathscr{E}_n}\otimes _{\mathscr{O}_{Z_n}}\omega_{Z_n|\mathscr{D}_n}^1\big)\oplus\big( \mathscr{O}_{\mathscr{E}_n}\otimes _{\mathscr{O}_{Z_n}} \mathscr{O}_{\mathscr{Z}_n}(-\mathscr{D}_n) \big);
\end{equation}
\[b\mapsto \Big(d\log b \mod p^n, p^{-1}\log(b^p \varphi_{ \mathscr{E}_{n+1} }(b)^{-1})\Big),\]
at degree $1$, where $\varphi_{ \mathscr{E}_{n} }: \mathscr{O}_{\mathscr{E}_{n}}\rightarrow \mathscr{O}_{\mathscr{E}_{n}}$ denotes the Frobenius operator induced by $\{F_{Z_n}\}$
and we have used the fact that $\log(b^p \varphi_{ \mathscr{E}_{n+1} }(b)^{-1})$ is contained in \[(\flat)\ \ \ p( \mathscr{O}_{\mathscr{E}_{n+1} }\otimes _{\mathscr{O}_{Z_{n+1}}} \mathscr{O}_{Z_{n+1}}(-\mathscr{D}_{n+1}) )\xleftarrow{\cong} \mathscr{O}_{\mathscr{E}_{n} }\otimes _{\mathscr{O}_{Z_{n}}} \mathscr{O}_{Z_{n}}(-\mathscr{D}_{n}),\] since $b^p \varphi_{ \mathscr{E}_{n+1} }(b)^{-1}\in 1+ p( \mathscr{O}_{\mathscr{E}_{n+1} }\otimes _{\mathscr{O}_{Z_{n+1}}} \mathscr{O}_{Z_{n+1}}(-\mathscr{D}_{n+1}))$. 
We will show  the isomorphism $(\flat)$ below:
\begin{Lem}\label{p-sequence} We have the following exact sequence 
\begin{multline*} \mathscr{O}_{\mathscr{E}_{n+m}}\otimes_{ \mathscr{O}_{Z_{n+m}}} \mathscr{O}_{Z_{n+m}}(-\mathscr{D}_{n+m}) \overset{\times p^n}{\longrightarrow} \mathscr{O}_{\mathscr{E}_{n+m}}\otimes_{ \mathscr{O}_{Z_{n+m}}} \mathscr{O}_{Z_{n+m}}(-\mathscr{D}_{n+m}) \overset{\times p^m}{\longrightarrow}\\ \mathscr{O}_{\mathscr{E}_{n+m}}\otimes_{ \mathscr{O}_{Z_{n+m}}} \mathscr{O}_{Z_{n+m}}(-\mathscr{D}_{n+m}) \longrightarrow \mathscr{O}_{\mathscr{E}_{m}}\otimes_{\mathscr{O}_{Z_m}} \mathscr{O}_{Z_m}(-\mathscr{D}_{m}) \longrightarrow 0.\end{multline*}
\end{Lem}
\begin{proof} We have the following exact sequence (\cite{Tsu1}) 
\[ \mathscr{O}_{\mathscr{E}_{n+m}}\overset{\times p^n}{\longrightarrow} \mathscr{O}_{\mathscr{E}_{n+m}} \overset{\times p^m}{\longrightarrow} \mathscr{O}_{\mathscr{E}_{n+m}}.\]
The module  $\mathscr{O}_{Z_{n+m}}(-\mathscr{D}_{n+m})$ is a flat $\mathscr{O}_{Z_{n+m}}$-module. Then we have the following exact sequence
\begin{multline*} \mathscr{O}_{\mathscr{E}_{n+m}}\otimes_{ \mathscr{O}_{Z_{n+m}}} \mathscr{O}_{Z_{n+m}}(-\mathscr{D}_{n+m}) \overset{\times p^n}{\longrightarrow} \mathscr{O}_{\mathscr{E}_{n+m}}\otimes_{ \mathscr{O}_{Z_{n+m}}} \mathscr{O}_{Z_{n+m}}(-\mathscr{D}_{n+m}) \overset{\times p^m}{\longrightarrow}\\ \mathscr{O}_{\mathscr{E}_{n+m}}\otimes_{ \mathscr{O}_{Z_{n+m}}} \mathscr{O}_{Z_{n+m}}(-\mathscr{D}_{n+m}) \longrightarrow \Coker(p^m)\longrightarrow 0.\end{multline*}
Here we have $\Coker(p^m) \cong \mathscr{O}_{\mathscr{E}_{m}}\otimes_{\mathscr{O}_{Z_m}} \mathscr{O}_{Z_m}(-\mathscr{D}_{m})$. This completes the proof.
\end{proof}

\begin{Cor} We have an isomorphism
\[(\flat)\ \ \ p( \mathscr{O}_{\mathscr{E}_{n+1} }\otimes _{\mathscr{O}_{Z_{n+1}}} \mathscr{O}_{Z_{n+1}}(-\mathscr{D}_{n+1}) )\xleftarrow{\cong} \mathscr{O}_{\mathscr{E}_{n} }\otimes _{\mathscr{O}_{Z_{n}}} \mathscr{O}_{Z_{n}}(-\mathscr{D}_{n}).\] 
\end{Cor}
\begin{proof}If $n=n, m=1$ in Lemma \ref{p-sequence}, we have the exact sequence
\[ \mathscr{O}_{\mathscr{E}_{n+1}}\otimes_{ \mathscr{O}_{Z_{n+m}}} \mathscr{O}_{Z_{n+m}}(-\mathscr{D}_{n+1}) \overset{\times p^n}{\longrightarrow} \mathscr{O}_{\mathscr{E}_{n+m}}\otimes_{ \mathscr{O}_{Z_{n+1}}} \mathscr{O}_{Z_{n+1}}(-\mathscr{D}_{n+1}) \overset{\times p}{\longrightarrow}\\ \mathscr{O}_{\mathscr{E}_{n+1}}\otimes_{ \mathscr{O}_{Z_{n+1}}} \mathscr{O}_{Z_{n+m}}(-\mathscr{D}_{n+1}).\]
Then we have $\Coker(p^n) \cong p( \mathscr{O}_{\mathscr{E}_{n+1} }\otimes _{\mathscr{O}_{Z_{n+1}}} \mathscr{O}_{Z_{n+1}}(-\mathscr{D}_{n+1})).$
On the other hand, if $n=1, m=n$ in Lemma \ref{p-sequence}, we have the exact sequence
\[\mathscr{O}_{\mathscr{E}_{n+1}}\otimes_{ \mathscr{O}_{Z_{n+1}}} \mathscr{O}_{Z_{n+1}}(-\mathscr{D}_{n+1}) \overset{\times p^n}{\longrightarrow}\\ \mathscr{O}_{\mathscr{E}_{n+m}}\otimes_{ \mathscr{O}_{Z_{n+1}}} \mathscr{O}_{Z_{n+1}}(-\mathscr{D}_{n+1}) \longrightarrow \mathscr{O}_{\mathscr{E}_{n}}\otimes_{\mathscr{O}_{Z_n}} \mathscr{O}_{Z_n}(-\mathscr{D}_{n}) \longrightarrow 0.\]
Then we have $\mathscr{O}_{\mathscr{E}_{n}}\otimes_{\mathscr{O}_{Z_n}} \mathscr{O}_{Z_n}(-\mathscr{D}_{n}) \cong \Coker(p^n)$. We obtain the isomorphism $(\flat)$. \end{proof}

Taking $\mathcal{H}^1$, we obtain 
\begin{equation}
{\rm Symb}_{X|D}: \big(1+I_{D_{n+1}}\big)^{\times}= \mathcal{H}^1(C_{n+1})\longrightarrow \mathcal{H}^1(s_n(1)_{X|D})\cong  \mathcal{H}^1(\mathscr{S}_n(1)^{loc}_{X|D}).
\end{equation}
We obtain the symbol map (\ref{symb}) as the following composite maps:
\begin{equation}
 (1+I_{D_{n+1}})^{\times}\otimes (M_{X_{n+1}}^{gp})^{\otimes{q-1}} \xrightarrow{{\rm Symb}_{X|D}\otimes {\rm Symb}_X} \mathcal{H}^1(s_n(1)_{X|D}) \otimes \mathcal{H}^{q-1}(\mathcal{S}_n(q-1)_{(X_n, M_{X_n}), (Z_n, M_{Z_n})})
\end{equation} 
\begin{equation*}
\longrightarrow \mathcal{H}^q(s_n(q)_{X|D})\cong  \mathcal{H}^q(\mathscr{S}_n(q)^{loc}_{X|D}).
\end{equation*}
Here ${\rm Symb}_X:(M_{X_{n+1}}^{gp})^{\otimes{q-1}}\rightarrow  \mathcal{H}^{q-1}(\mathcal{S}_n(q-1)_{(X_n, M_{X_n}), (Z_n, M_{Z_n})})$ is symbol map defined by \cite[\S2]{Tsu1}. The second morphism is product structure $s_n(1)_{X|D} \otimes^{\mathbb{L}} \mathcal{S}_n(q-1)_{(X, M_X)} \rightarrow s_n(q)_{X|D}$.

\subsection{Global construction of the symbol map}
We will construct the symbol map in global. First we show that the local symbol map is independent of the choice of the embedding system below.
 We consider the following three diagrams:
\[
\xymatrix@M=10pt{
X_{n+1} \ar@{}[rd]|{\square\ \ (1)} \ar@<-0.5ex>@{^{(}->}[r]^-{\beta_{n+1}}& Z_{n+1}&X_{n+1} \ar@{}[rd]|{\square\ \ (2)} \ar@<-0.5ex>@{^{(}->}[r]^-{\beta'_{n+1}}&  Z'_{n+1} &\\
D_{n+1} \ar@<-0.5ex>@{^{(}->}[u]\ar@<-0.5ex>@{^{(}->}[r]^-{\beta_{D,{n+1}}}& \mathscr{D}_{n+1}, \ar@<-0.5ex>@{^{(}->}[u]& D_{n+1} \ar@<-0.5ex>@{^{(}->}[u]\ar@<-0.5ex>@{^{(}->}[r]^-{\beta'_{D,{n+1}}}& \mathscr{D}'_{n+1}, \ar@<-0.5ex>@{^{(}->}[u]}\]

\[
\xymatrix@M=10pt{
X_{n+1} \ar@{}[rd]|{\square\ \ (3)} \ar@<-0.3ex>@{^{(}->}[r]^-{\beta''_{n+1}}&  Z_{n+1}\times Z'_{n+1} &\\
D_{n+1} \ar@<-0.3ex>@{^{(}->}[u]\ar@<-0.3ex>@{^{(}->}[r]^-{\beta''_{D,{n+1}}}& \mathscr{D}_{n+1} \times  \mathscr{D}'_{n+1}.  \ar[u]&  
}
\]
Note that $ \mathscr{D}_n \times  \mathscr{D}'_n$ is not effective Cartier divisor on $ Z_n\times Z'_n$. 
We put the projections \[p_{Z_{n+1}} : Z_{n+1}\times Z'_{n+1}  \rightarrow Z_{n+1},\quad p_{Z'_{n+1}} : Z_{n+1}\times Z'_{n+1}  \rightarrow Z'_{n+1},\]
\[p_{\mathscr{D}_{n+1}} : \mathscr{D}_{n+1}\times \mathscr{D}'_{n+1}  \rightarrow \mathscr{D}_{n+1},\quad p_{\mathscr{D}'_{n+1}} : \mathscr{D}_{n+1}\times \mathscr{D}'_{n+1}  \rightarrow \mathscr{D}'_{n+1}.\]
We denote by \[J_{\mathscr{E}_{n+1}, ?}\ (?=Z_{n+1}, Z'_{n+1}, \mathscr{D}_{n+1}, \mathscr{D}'_{n+1}, Z_{n+1}\times Z'_{n+1}, \mathscr{D}_{n+1}\times  \mathscr{D}'_{n+1} )\]
the ideal $\Ker \left(\mathscr{O}_{\mathscr{E}_{n+1}, ?} \longrightarrow \mathscr{O}_{*}  \right)$, where $\mathscr{E}_{n+1, ?}$ is the PD-envelope of \[*=X_{n+1}, X'_{n+1}, D_{n+1}, D'_{n+1}, X_{n+1}, D_{n+1}\]
in $?=Z_{n+1}, Z'_{n+1}, \mathscr{D}_{n+1}, \mathscr{D}'_{n+1}, Z_{n+1}\times Z'_{n+1}, \mathscr{D}_{n+1}\times  \mathscr{D}'_{n+1}$ respectively.
We put 
\[C^I_{n+1}:=\left[\big(1+J_{\mathscr{E}_{n+1}})\cap \big(1+\mathscr{O}_{Z_{n+1}}(-\mathscr{D}_{n+1})\big)^{\times}   \longrightarrow \big(1+\mathscr{O}_{Z_{n+1}}(-\mathscr{D}_{n+1})\big)^{\times}\right],\]
\[C^{II}_{n+1}:=\left[\big(1+J_{\mathscr{E}'_{n+1}})\cap \big(1+\mathscr{O}_{Z'_{n+1}}(-\mathscr{D}'_{n+1})\big)^{\times}   \longrightarrow \big(1+\mathscr{O}_{Z'_{n+1}}(-\mathscr{D}'_{n+1})\big)^{\times}\right],\]
\[C_{n+1}(X_{n+1}, Z_{n+1}):=\left[1+J_{\mathscr{E}_{n+1}, Z_{n+1}} \rightarrow M^{gp}_{\mathscr{E}_{n+1},Z_{n+1}}\right]\]
\[C_{n+1}(X'_{n+1}, Z'_{n+1}):=\left[1+J_{\mathscr{E}'_{n+1},Z'_{n+1}} \rightarrow M^{gp}_{\mathscr{E}'_{n+1},Z'_{n+1}}\right]\]
\[C_{n+1}(D_{n+1}, \mathscr{D}_{n+1}):=\left[1+J_{\mathscr{E}_{n+1},\mathscr{D}_{n+1}} \rightarrow M^{gp}_{\mathscr{E}_{n+1},\mathscr{D}_{n+1}}\right]\]
\[C_{n+1}(D'_{n+1}, \mathscr{D}'_{n+1}):=\left[1+J_{\mathscr{E}_{n+1},\mathscr{D}'_{n+1}} \rightarrow M^{gp}_{\mathscr{E}_{n+1},\mathscr{D}'_{n+1}}\right]\]
\[\tilde{C}^{I}_{n+1}:=\Cone\left(C_{n+1}(X_{n+1}, Z_{n+1})\longrightarrow C_{n+1}(D_{n+1}, \mathscr{D}_{n+1}) \right)[-1],\]
\[\tilde{C}^{II}_{n+1}:=\Cone\left(C_{n+1}(X'_{n+1}, Z'_{n+1})\longrightarrow C_{n+1}(D'_{n+1}, \mathscr{D}'_{n+1}) \right)[-1],\]
\[s^I_n(1)_{X|D}:=s_n(1)_{X|D, (Z_n,M_{Z_n}), (\mathscr{D}_n,M_{\mathscr{D}_n})},\]
\[s^{II}_n(1)_{X|D}:=s_n(1)_{X|D, (Z'_n,M_{Z'_n}), (\mathscr{D}'_n,M_{\mathscr{D}'_n})},\]
\[\mathscr{S}^{loc}_n(1)_{X|D}:={\mathscr{S}_n(q)}^{loc}_{X|D, (Z_n\times Z'_n,M_{Z_n\times Z'_n}), (\mathscr{D}_n\times\mathscr{D}'_n, M_{\mathscr{D}_n\times \mathscr{D}'_n})},\]
\[\mathscr{S}^{I, loc}_n(1)_{X|D}:={\mathscr{S}_n(q)}^{loc}_{X|D, (Z_n,M_{Z_n}),(\mathscr{D}_n,M_{\mathscr{D}_n})},\]
\[\mathscr{S}^{II, loc}_n(1)_{X|D}:={\mathscr{S}_n(q)}^{loc}_{X|D, (Z'_n,M_{Z'_n}),(\mathscr{D}'_n,M_{\mathscr{D}'_n})},\]
\[\mathfrak{C}:=\Cone\left(C_{n+1}(X_{n+1}, Z_{n+1}\times Z'_{n+1})\longrightarrow C_{n+1}(D_{n+1}, \mathscr{D}_{n+1}\times  \mathscr{D}'_{n+1}) \right)[-1],\]
\[C_{n+1}(X_{n+1}, Z_{n+1}\times Z'_{n+1}):=\left[1+J_{\mathscr{E}_{n+1,  Z_{n+1}\times Z'_{n+1}}} \rightarrow M^{gp}_{\mathscr{E}_{n+1,  Z_{n+1}\times Z'_{n+1}}}\right],\]
\[C_{n+1}(D_{n+1}, \mathscr{D}_{n+1}\times  \mathscr{D}'_{n+1}):=\left[1+J_{\mathscr{E}_{n+1, \mathscr{D}_{n+1}\times  \mathscr{D}'_{n+1}}} \rightarrow M^{gp}_{\mathscr{E}_{n+1, \mathscr{D}_{n+1}\times  \mathscr{D}'_{n+1}}}\right].\]
\begin{Lem} The following diagram is commutative:
{\tiny\[
\xymatrix@M=10pt{ & &\mathcal{H}^1(C^{II}_{n+1})\ar[d]^-{\cong}_{f''^{II}}\ar[rrd]&&\\
&(1+I_{D_{n+1}})^{\times}\ar@/_24pt/[ld]^-{\cong}_{f'^I}\ar@/^24pt/[ru]^-{\cong}_{f'^{II}}\ar[d]^-{\cong}_{f^I}\ar[r]^-{\cong}_{f^{II}} & \mathcal{H}^1(\tilde{C}^{II}_{n+1})\ar@{}[lu]|{(1)}\ar[d]\ar[rrd]&(2)& \mathcal{H}^1(s^{II}_n(1)_{X|D})\ar[d]^-{\cong}\\
\mathcal{H}^1(C^{I}_{n+1})\ar[r]^-{\cong}_{f''^I}\ar[rd]& \mathcal{H}^1(\tilde{C}^{I}_{n+1})\ar[rd]\ar[r] \ar@{}[lu]|{(6)}& \mathcal{H}^1(\mathfrak{C})\ar@{}[lu]|{(7)} \ar[rrd]&(3)&\mathcal{H}^1\left(\mathscr{S}^{II, loc}_n(1)_{X|D}\right)\ar[d]\\
 &\mathcal{H}^1(s^{I}_n(1)_{X|D})\ar[r]^-{\cong}\ar@{}[u]|{(5)}&\mathcal{H}^1\left(\mathscr{S}^{I, loc}_n(1)_{X|D}\right)\ar[rr]\ar@{}[u]|{(4)}&& \mathcal{H}^1\left(\mathscr{S}^{loc}_n(1)_{X|D}\right),
}\]} where we define $f^I:=f''^I\circ f'^I$ and  $f^{II}:=f''^{II}\circ f'^{II}$.\\
Here the isomorphisms
\[f'^I : (1+I_{D_{n+1}})^{\times} \xrightarrow{\cong}  \mathcal{H}^1(\tilde{C}^{I}_{n+1}),\quad f'^{II} : (1+I_{D_{n+1}})^{\times} \xrightarrow{\cong}  \mathcal{H}^1(\tilde{C}^{II}_{n+1})\]
are defined from an exact sequences \[0\rightarrow \big(1+J_{\mathscr{E}_{n+1}})\cap \big(1+\mathscr{O}_{Z_{n+1}}(-\mathscr{D}_{n+1})\big)^{\times}   \rightarrow \big(1+\mathscr{O}_{Z_{n+1}}(-\mathscr{D}_{n+1})\big)^{\times} \rightarrow (1+I_{D_{n+1}})^{\times} \rightarrow 0, \]
 \[0\rightarrow \big(1+J_{\mathscr{E}'_{n+1}})\cap \big(1+\mathscr{O}_{Z'_{n+1}}(-\mathscr{D}'_{n+1})\big)^{\times}   \rightarrow \big(1+\mathscr{O}_{Z'_{n+1}}(-\mathscr{D}'_{n+1})\big)^{\times} \rightarrow (1+I_{D_{n+1}})^{\times} \rightarrow 0.\]
For $f''^I$ is defined by the morphism  $(*)\;:\; C^I_{n+1} \rightarrow \tilde{C}^{I}_{n+1}$. In degree $0$, the morphism of complexes $(*)$ is 
\[\big(1+J_{\mathscr{E}_{n+1}})\cap \big(1+\mathscr{O}_{Z_{n+1}}(-\mathscr{D}_{n+1})\big)^{\times}  \rightarrow 1+J_{\mathscr{E}_{n+1}, Z_{n+1}}, \]
and in degree $1$, the morphism $(*)$ is
\[ \big(1+\mathscr{O}_{Z_{n+1}}(-\mathscr{D}_{n+1})\big)^{\times} \rightarrow M^{gp}_{\mathscr{E}_{n+1},Z_{n+1}} \oplus 1+J_{\mathscr{E}_{n+1},\mathscr{D}_{n+1}}.\]
The map $f''^{II}$ is similar.

\end{Lem}
\begin{proof} We show the commutativity of $(1)$--$(7)$ below:\\ The diagrams $(1)$ and $(6)$ are commutative by the definition of $f^I$ and $f^{II}$.\\
\textbf{The commutativity of $(3)$ and $(4)$}: The proofs of $(3)$ and $(4)$ are the same, we only show the case $(4)$.
It is enough to show that the commutativity of the following diagram of complexes:
\[
\xymatrix@M=10pt{
\tilde{C}^{I}_{n+1} \ar[r] \ar[d]& \mathscr{S}^{I, loc}_n(1)_{X|D} \ar[d]\\
 \mathfrak{C}\ar[r] & \mathscr{S}^{loc}_n(1)_{X|D}.
}\]
In degree $0$, the above diagram is 
\[
\xymatrix@M=10pt{
1+J_{\mathscr{E}_{n+1}, Z_{n+1}} \ar[r] \ar[d]^-{p^*_{Z_{n+1}}} & J_{\mathscr{E}_{n}, Z_{n}} \ar[d]^-{p^*_{Z_{n}}}\\
1+J_{\mathscr{E}_{n+1}, Z_{n+1}\times Z'_{n+1} } \ar[r] &J_{\mathscr{E}_{n}, Z_{n}\times Z'_{n}}.
}\]
Here the upper and lower horizontal morphisms are defined by $a \mapsto \log(a)\mod p^n$, $p_{Z_{n+1}} (\resp.\ p_{\mathscr{D}_{n+1}})$ is the projection $Z_{n+1}\times Z'_{n+1}\rightarrow Z_{n+1}\ (\resp.\ \mathscr{D}_{n+1}\times \mathscr{D}'_{n+1} \rightarrow \mathscr{D}_{n+1})$. 
In degree $1$, we have the diagram:
{\footnotesize\[
\xymatrix@M=10pt{
M^{gp}_{\mathscr{E}_{n+1}, Z_{n+1}}\oplus \left(1+J_{\mathscr{E}_{n+1}, \mathscr{D}_{n+1}}\right)  \ar[r] \ar[d]^-{p^*_{Z_{n+1}}\oplus p^*_{\mathscr{D}_{n+1}}} &\left((\mathscr{O}_{\mathscr{E}_n,Z_n}\otimes \omega^1_{Z_n})\oplus \mathscr{O}_{\mathscr{E}_n, Z_n} \right)\oplus J_{\mathscr{E}_{n}, \mathscr{D}_{n}} \ar[d]^-{p^*_{Z_{n}}\oplus p^*_{\mathscr{D}_{n+1}}} \\
M^{gp}_{\mathscr{E}_{n+1}, Z_{n+1}\times Z'_{n+1}}\oplus \left(1+J_{\mathscr{E}_{n+1}, \mathscr{D}_{n+1}\times  \mathscr{D}'_{n+1}}\right)\ar[r] & \left((\mathscr{O}_{\mathscr{E}_n, Z_{n}\times Z'_{n}}\otimes \omega^1_{Z_{n}\times Z'_{n}})\oplus \mathscr{O}_{\mathscr{E}_n,} \right)\oplus J_{\mathscr{E}_{n},\mathscr{D}_{n}\times  \mathscr{D}'_{n}}.
}\]}
Here the upper (\resp.\ lower) horizontal morphisms are defined by 
\[(a,b) \mapsto \left(\left(d\log a\mod p^n, p^{-1}\log(a^p\varphi_{\mathscr{E}_n, Z_{n}}(a)^{-1})\right),\ \log b \mod p^n\right),\]
 \[\left(\resp.\ (c,d) \mapsto \left(\left(d\log c\mod p^n, p^{-1}\log(c^p\varphi_{\mathscr{E}_n, Z_{n}\times Z'_{n}}(c)^{-1})\right),\ \log d \mod p^n\right)\right).\]
 In degree $2$, we have the diagram 
 {\footnotesize\[
\xymatrix@M=10pt{
M^{gp}_{\mathscr{E}_{n+1}, \mathscr{D}_{n+1}} \ar[r] \ar[d]&\mathcal{S}_n(1)^1_{D_n, \mathscr{D}_n}\ar[d] \\
M^{gp}_{\mathscr{E}_{n+1}, \mathscr{D}_{n+1}\times\mathscr{D}'_{n+1}}\ar[r] &\mathcal{S}_n(1)^1_{D_n, \mathscr{D}_n\times \mathscr{D}'_n}.
}\]}
Here the right and left vertical morphisms are projections $p^*_{Z_{n+1}}$. This commutativity is obvious.\\
Then we have the commutativity of the above diagrams $(4)$.\\
\textbf{The commutativity of $(2)$ and $(5)$}:
The proofs of $(2)$ and $(5)$ are the same, we only show the case $(5)$.
 It is enough to show the commutativity of the following diagram of complexes:
\[
\xymatrix@M=10pt{
C^{I}_{n+1} \ar[r] \ar[d]&s^I_n(1)_{X|D} \ar[d]\\
\tilde{C}^{I}_{n+1}\ar[r] & \mathscr{S}^{I, loc}_n(1)_{X|D}.
}\]
First we define the left vertical arrow $C^{I}_{n+1} \longrightarrow \tilde{C}^{I}_{n+1}$.
{\footnotesize\[
\xymatrix@M=10pt{C^{I}_{n+1}:\ar[d]& (1+J_{\mathscr{E}_{n+1}, Z_{n+1}})\cap \left(1+\mathscr{O}_{Z_{n+1}}(-\mathscr{D}_{n+1})\right)^{\times}\ar[d]^-{h^0} \ar[r]&\left(1+\mathscr{O}_{Z_{n+1}}(-\mathscr{D}_{n+1})\right)^{\times}\ar[d]^-{h^1} \\
\tilde{C}^{I}_{n+1}:& 1+J_{\mathscr{E}_{n+1}, Z_{n+1}} \ar[r] & M^{gp}_{\mathscr{E}_{n+1}, Z_{n+1}}\oplus \left(1+J_{\mathscr{E}_{n+1}, \mathscr{D}_{n+1}}\right).
}\]}
Here the upper horizontal arrows are the inclusion map and the lower horizontal arrow is $x \mapsto (-x, x|_{1+J_{\mathscr{E}_{n+1}, \mathscr{D}_{n+1}}})$. The morphism $h^0$ is the inclusion map and the morphism $h^1$ is $z \mapsto (-z, 1)$. Then the above diagram is commutative. Hence we obtain the morphism of complex $C^{I}_{n+1} \longrightarrow \tilde{C}^{I}_{n+1}$.\\
In degree $0$, 
{\footnotesize\[
\xymatrix@M=10pt{ (1+J_{\mathscr{E}_{n+1}, Z_{n+1}})\cap \left(1+\mathscr{O}_{Z_{n+1}}(-\mathscr{D}_{n+1})\right)^{\times} \ar[d] \ar[r]& J_{\mathscr{E}_{n}, Z_{n}}\otimes \mathscr{O}_{Z_{n}}(-\mathscr{D}_{n})\ar[d]\\
1+J_{\mathscr{E}_{n+1}, Z_{n+1}} \ar[r]& J_{\mathscr{E}_{n}, Z_{n}}.
}\]}
Here the upper and the lower horizontal arrow is $a \mapsto \log a \mod p^{n+1}$. The right and the left vertical arrows are the inclusion map. Then the diagram is commutative. \\
In degree $1$, the above diagram is
{\footnotesize\[
\xymatrix@M=10pt{
\left(1+\mathscr{O}_{Z_{n+1}}(-\mathscr{D}_{n+1})\right)^{\times} \ar[r] \ar[d]&\left((\mathscr{O}_{\mathscr{E}_n,Z_n}\otimes \omega^1_{Z_n|\mathscr{D}_n})\oplus \mathscr{O}_{\mathscr{E}_n, Z_n} \right)\oplus \left(\mathscr{O}_{\mathscr{E}_n,Z_n}\otimes \mathscr{O}_{Z_{n}}(-\mathscr{D}_{n})  \right) \ar[d] \ar[d]\\
M^{gp}_{\mathscr{E}_{n+1}, Z_{n+1}}\oplus \left(1+J_{\mathscr{E}_{n+1}, \mathscr{D}_{n+1}}\right) \ar[r] &\left((\mathscr{O}_{\mathscr{E}_n,Z_n}\otimes \omega^1_{Z_n})\oplus \mathscr{O}_{\mathscr{E}_n, Z_n} \right)\oplus J_{\mathscr{E}_n,\mathscr{D}_{n}}.
}\]}
The left vertical arrow is $b \mapsto (-b, 1)$, the upper horizontal arrow is \[b \mapsto (d\log b\mod p^n,\  p^{-1}\log(b^p\varphi_{\mathscr{E}_{n+1}}(b)^{-1})),\]
the lower horizontal arrow is $(x,y) \mapsto \left((d\log x\mod p^n,\  p^{-1}\log(x^p\varphi_{\mathscr{E}_{n+1}}(x)^{-1})),\ \log y \mod p^n \right)$, and the right vertical arrow is the inclusion to the first component $w \mapsto (w, 0)$. Then we obtain the commutativity of this diagram $(5)$.  \\
\textbf{The commutativity of $(7)$}: We consider the diagram
\[
\xymatrix@M=10pt{
 (1+I_{D_{n+1}})^{\times} \ar[r]^-{\cong} \ar[d]^-{\cong}&\mathcal{H}^1(\tilde{C}^{II}_{n+1}) \ar[d]^-{p^*_{Z'_{n+1}}\oplus p^*_{\mathscr{D}'_{n+1}}}\\
\mathcal{H}^1(\tilde{C}^{I}_{n+1})\ar[r]^-{p^*_{Z_{n+1}}\oplus p^*_{\mathscr{D}_{n+1}}} & \mathcal{H}^1(\mathfrak{C}).
}\]
By the isomorphisms  $p^*_{Z_{n+1}}(J_{\mathscr{E}_{n+1}, Z_{n+1}})\cong p^*_{Z'_{n+1}}(J_{\mathscr{E}_{n+1}, Z'_{n+1}})$, $p^*_{Z_{n+1}}(M_{\mathscr{E}_{n+1}, Z_{n+1}})\cong p^*_{Z'_{n+1}}(M_{\mathscr{E}_{n+1}, Z'_{n+1}})$ and $p^*_{\mathscr{D}_{n+1}}(J_{\mathscr{E}_{n+1}, \mathscr{D}_{n+1}})\cong p^*_{\mathscr{D}'_{n+1}}(J_{\mathscr{E}_{n+1}, \mathscr{D}'_{n+1}})$,  the above diagram $(7)$ is commutative. This completes the proof.
\end{proof}

Take $X^{\bullet}$ in the above construction, we have the morphism 
\[(1+I_{D^{\bullet}_{n+1}})^{\times}[-1] \longrightarrow \mathscr{S}^{loc}_n(1)_{X^{\bullet}|D^{\bullet}}.\]
Since $\theta^*(1+I_{D_{n+1}})^{\times}=(1+I_{D^{\bullet}_{n+1}})^{\times}|X^{\bullet}_n$, we obtain a morphism 
\[\theta^*(1+I_{D_{n+1}})^{\times}[-1] \longrightarrow \mathscr{S}^{loc}_n(1)_{X^{\bullet}|D^{\bullet}}.\]
Taking $R\theta_*$, we obtain a morphism 
\[(1+I_{D_{n+1}})^{\times}[-1] \longrightarrow \mathscr{S}_n(1)_{X|D}.\]
The local symbol map has functorial property for $X$ and $Z$, hence we get the symbol map 
\[(1+I_{D_{n+1}})^{\times} \otimes (M_{X_{n+1}}^{gp})^{\otimes{q-1}} \longrightarrow \mathcal{H}^q(\mathscr{S}_n(q)_{X|D}).\]


\section{\textbf{Main Rsults}}
In this and the next section, for $0 \leq q \leq p-2$ and $p \geq 3$, we calculate the cohomology sheaf
\begin{equation}\mathcal{H}^q(s_1(q)_{X|D})\quad(0 \leq q \leq p-2, p\geq 3).\end{equation}
We first define two filtrations on the sheaf $\mathcal{H}^q(s_1(q)_{X|D})$ using symbols and state our main results on the associated graded pieces.

\begin{Def}\label{fil} We define the filtrations $U^{\cdot}$ and $V^{\cdot}$ on $(1+I_{D_2})^{\times} \otimes (M_{X_2}^{gp})^{\otimes (q-1)}$ ($q \geq 1$) by
\begin{equation}
U^0((1+I_{D_2})^{\times} ):=(1+I_{D_2})^{\times},\quad V^0((1+I_{D_2})^{\times} ):=(1+\pi I_{D_2})^{\times} \cdot \pi^{\mathbb{N}} , 
\end{equation}
\begin{equation}
U^m((1+I_{D_2})^{\times} ):=(1+\pi^m I_{D_2})^{\times},\quad V^m((1+I_{D_2})^{\times} ):=U^{m+1}((1+I_{D_2})^{\times} )\;(m \geq 1)  , 
\end{equation}
if $q=1$, and
\begin{equation}
U^m\Big((1+I_{D_2})^{\times} \otimes (M_{X_2}^{gp})^{\otimes (q-1)} \Big):=({\rm the\;image\; of\;} U^m((1+I_{D_2})^{\times} ))\otimes (M_{X_2}^{gp})^{\otimes (q-1)},
\end{equation}
\begin{equation}
V^m\Big((1+I_{D_2})^{\times} \otimes (M_{X_2}^{gp})^{\otimes (q-1)} \Big):=({\rm the\;image\; of\;} U^m((1+I_{D_2})^{\times} ))\otimes (M_{X_2}^{gp})^{\otimes (q-2)}\otimes \pi^{\mathbb{N}} 
\end{equation}
\begin{equation*}
+\;U^{m+1}\Big((1+I_{D_2})^{\times} \otimes (M_{X_2}^{gp})^{\otimes (q-1)} \Big)
\end{equation*}
if $q \geq 2$. Here $(1+\pi^m I_{D_2})^{\times}:=(1+\pi^m I_{D_2}) \cap \mathscr{O}_{X_2}^{\times}$ for $m \geq 0$.
\end{Def}

Here we denote by the same notation $\pi$ the image of $\pi \in \Gamma(S, N)=\mathscr{O}_K\backslash \{0\}$ under the map $\Gamma(S, N)\longrightarrow \Gamma(X, M_X)$ and its images in $\Gamma(X_n, M_{X_n})\ (n \in \mathbb{N})$.
We define the filtration $U^{\cdot}$ and $V^{\cdot}$ on $\mathcal{H}^q(s_1(q)_{X|D})$($q \geq 0$) to be the images of these filtrations under the symbol map \ref{symb}. 
There are natural inclusions $U^{m+1} \subset V^m$ and $V^m \subset U^m$.
Put 
\begin{align}
\gr_0^m\mathcal{H}^q(s_1(q)_{X|D})&:=U^m\mathcal{H}^q(s_1(q)_{X|D})/V^m\mathcal{H}^q(s_1(q)_{X|D}),\\
\gr_1^m\mathcal{H}^q(s_1(q)_{X|D})&:=V^m\mathcal{H}^q(s_1(q)_{X|D})/U^{m+1}\mathcal{H}^q(s_1(q)_{X|D}).
\end{align}

To describe these graded pieces, we introduce some differential sheaves on $Y$.
We define \begin{equation}\omega_{Y|D_s}^q:=\omega_{Y}^q \otimes_{\mathscr{O}_Y}\mathscr{O}_Y(-D_s),\end{equation} where $s:=\Spec(k)$, $\omega_{Y}^q:= \Omega^q_{Y/s}\big(\log(M_Y/N_s)\big)$, $D_s:=D \otimes_{\mathscr{O}_K} k$ and $(s, N_s)$ denotes the log point over $s$. We define the subsheaves $Z_{Y|D_s}^q$ and $B_{Y|D_s}^q$ of $\omega_{Y|D_s}^q$ by 
\begin{equation}
Z_{Y|D_s}^q:=\Ker(d^q:\omega_{Y|D_s}^q \rightarrow \omega_{Y|D_s}^{q+1}),
\end{equation}
\begin{equation}
B_{Y|D_s}^q:=\Imm(d^{q-1}:\omega_{Y|D_s}^{q-1}\rightarrow \omega_{Y|D_s}^q).
\end{equation}

Let $\omega_{Y|D_s,\log}^q$ be the subsheaf of abelian groups of $\omega_{Y}^q$ generated by local sections of the form
\begin{equation*}
d\log(x)\land d\log(a_1)\land \cdots \land d\log(a_{q-1}),
\end{equation*}
where $x \in \Big(1+\mathscr{O}_Y\big(-D_s\big)\Big)^{\times}$ and $a_1, \dotsc , a_{q-1} \in M_Y$.

If $D=\sum_{\lambda \in \Lambda} m_{\lambda}D_{\lambda}$, we denote $D':=\sum_{\lambda \in \Lambda} m_{\lambda}'D_{\lambda}$. Here $m_{\lambda}':=\min\{l \in \mathbb{N}\;|\; p\cdot l \geq m_{\lambda} \}$. We put $D_s:=\sum_{\lambda \in \Lambda} m_{\lambda}{(D_s)}_{\lambda}$. 
 We define a map $d: \omega^{q}_{Y}\otimes_{\mathscr{O}_Y}  \mathscr{O}_Y(-D_s)\rightarrow \omega^{q+1}_{Y}\otimes_{\mathscr{O}_Y}  \mathscr{O}_Y(-D_s)$ by the local assignment 
\[\omega \otimes \prod_{\lambda \in \Lambda}\pi_{\lambda}^{m_\lambda} \mapsto \big(d\omega+\sum_{\lambda \in \Lambda}m_\lambda\cdot d\log(\pi_\lambda)\land \omega\big) \otimes \prod_{\lambda \in \Lambda}\pi_{\lambda}^{m_\lambda}\quad (\omega \in  \omega^{q}_{Y}),\]
where $\pi_\lambda \in \mathscr{O}_Y$ denotes a local uniformizer of ${(D_\lambda)}_s$, for each $\lambda \in \Lambda$. 
\begin{Lem}
 $(\omega_Y^{\cdot}\otimes_{\mathscr{O}_Y} \mathscr{O}_Y(-D_s), d)$ is a complex.
\end{Lem}
\begin{proof} It is enough to show that $d^2=0$. We have 
{\footnotesize\begin{align*}
d^2& \left(\omega \otimes \prod_{\lambda \in \Lambda}\pi_{\lambda}^{m_\lambda} \right)=d\left(\big(d\omega+\sum_{\lambda \in \Lambda}m_\lambda\cdot d\log(\pi_\lambda)\land \omega\big) \otimes \prod_{\lambda \in \Lambda}\pi_{\lambda}^{m_\lambda}\right)\\
&=\left(d\left(d\omega+\sum_{\lambda \in \Lambda}m_\lambda\cdot d\log(\pi_\lambda)\land \omega \right)+ \sum_{\lambda \in \Lambda}m_\lambda\cdot d\log(\pi_\lambda)\land \big(d\omega+\sum_{\lambda \in \Lambda}m_\lambda\cdot d\log(\pi_\lambda)\land \omega\big) \right)\otimes \prod_{\lambda \in \Lambda}\pi_{\lambda}^{m_\lambda}\\
&=\left(-\sum_{\lambda \in \Lambda}m_\lambda\cdot d\log(\pi_\lambda)\land d\omega+ \sum_{\lambda \in \Lambda}m_\lambda\cdot d\log(\pi_\lambda)\land d\omega \right)\otimes \prod_{\lambda \in \Lambda}\pi_{\lambda}^{m_\lambda}=0.
\end{align*} }This completes the proof.\end{proof}
We have the following Lemma:
\begin{Lem}\label{Omega}{\rm (cf. \cite[Theorem 3.2]{SS} )} For each integer $q \geq 0$, there exists an isomorphism
\begin{equation}\label{Carinv}
C^{-1}: \omega_{Y|D_s'}^q \stackrel{\cong}{\longrightarrow} \mathcal{H}^q(\omega_{Y|D_s}^{\text{\large$\cdot$}})
\end{equation}
\begin{equation}
a\;\dlog(b_1)\land\dlog(b_2)\land...\land\dlog(b_q) \mapsto \tcf\;a^p\;\dlog(b_1)\land\dlog(b_2)\land...\land\dlog(b_q),
\end{equation}
\noindent
where $a \in \mathscr{O}_Y(-D'_s)$ and $b_1,...,b_q \in M_Y$.
\end{Lem}
\begin{proof}
We use a similar argument as in \cite[Theorem 3.2]{SS}.
If $p$ divides $m_\lambda$ for any $\lambda \in \Lambda$, then the map $d: \omega_Y^{q}\otimes \mathscr{O}_Y(-D_s)\rightarrow \omega^{q+1}_Y\otimes \mathscr{O}_Y(-D_s)$ sends \[\omega \otimes  \prod_{\lambda \in \Lambda}\pi_{\lambda}^{m_\lambda} \mapsto d\omega \otimes  \prod_{\lambda \in \Lambda}\pi_{\lambda}^{m_\lambda}.\]
Then we have $\mathcal{H}^q(\omega_{Y|D_s}^{\text{\large$\cdot$}}) \cong \mathcal{H}^q(\omega_{Y}^{\cdot})\otimes  \mathscr{O}_Y(-D_s)$. By the fact that Theorem A3 in Appendix \cite{Tsu2}, we have an isomorphism:\; $C^{-1}:\; \omega^q_{Y} \xrightarrow{\cong} \mathcal{H}^q(\omega_{Y}^{\cdot})$. Thus $C^{-1}: \omega_{Y|D_s'}^q \rightarrow \mathcal{H}^q(\omega_{Y|D_s}^{\text{\large$\cdot$}})$ is injective. By the assumption $p$ divides $m_\lambda$ for any $\lambda \in \Lambda$, the surjectivity of the map in the statement of Proposition \ref{Omega} $C^{-1}: x \otimes \omega \mapsto x^p \otimes \omega$ is obvious.
Thus we have the isomorphism\[C^{-1}:\omega^{q}_Y\otimes \mathscr{O}_Y(-D_s')\xrightarrow{\cong} \mathcal{H}^{q}\big(\omega^{\cdot}_Y\otimes \mathscr{O}_Y(-D_s)\big).\]

We next show the general case. We see that the natural inclusion
\[\omega^{\cdot}_Y\otimes \mathscr{O}_Y(-p\cdot D_s') \hookrightarrow \omega^{\cdot}_Y\otimes \mathscr{O}_Y(-D_s)\]
is a quasi-isomorphism. We define $\omega_{\textbf{m}}^{\cdot}:=\omega^{\cdot}_Y\otimes \mathscr{O}_Y(-D_s)(= \omega^{\cdot}_{Y|D_s}),$ where $\textbf{m}=(m_\lambda)_{\lambda \in \Lambda}$. We can consider a filtration
\[\omega_{p\cdot\textbf{m}'}^{\cdot}= \omega_{\textbf{m}_t}^{\cdot}\subset \cdots \subset \omega_{\textbf{m}_1}^{\cdot}\subset \omega_{\textbf{m}_0}^{\cdot}=\omega_{\textbf{m}}^{\cdot}\]
such that \[\sum_{\lambda \in \Lambda} m_{\lambda, i+1} -\sum_{\lambda \in \Lambda} m_{\lambda,i}=1\quad for\quad 0 \leq i <t,\] 
where $\textbf{m}_i:=(m_{\lambda, i})_{\lambda}$, $\textbf{m}'_i:=(m'_{\lambda, i})_{\lambda}$ and $\omega_{\textbf{m}_i}^{\cdot}:=\omega^{\cdot}_Y\otimes_{\mathscr{O}_Y} \mathscr{O}_Y(-D_s^i)$, $D_s^i:=\sum_\lambda m_{\lambda, i}(D_\lambda)_s$.
We have an short exact sequence $0\longrightarrow \mathscr{O}_Y(-D_s)\longrightarrow \mathscr{O}_Y\longrightarrow \mathscr{O}_{D_s} \longrightarrow 0$ by definitions, which gives 
short exact sequence
\[0\longrightarrow \mathscr{O}_Y\left(-(m+1)D_s\right)\longrightarrow \mathscr{O}_Y(-mD_s)\longrightarrow \mathscr{O}_{D_s}\otimes_{\mathscr{O}_Y}\mathscr{O}_Y(-mD_s) \longrightarrow 0.\]
Then the graded pieces of the above filtration are of the form $\frac{\omega_{\textbf{m}_{i}}^{\cdot}}{\omega_{\textbf{m}_{i+1}}^{\cdot}}$ is isomorphic to $\omega_{\textbf{m}_i,\mu}^{\cdot}:=\omega^{\cdot}_{D_{\mu}}\otimes_{\mathscr{O}_Y} \mathscr{O}_Y(-D^i_s)$. Here $\mu=\mu(i) \in \Lambda$ is the unique element such that $m_{\mu,i+1}>m_{\mu,i}$ and $\omega^{\cdot}_{D_{\mu}}:=\omega^{\cdot}_{Y}\otimes_{\mathscr{O}_Y} \mathscr{O}_{D_{\mu}}$.
\begin{Lem}(cf. {\rm \cite[Lemma 3.4]{SS}})\label{graded}
If $(m_{\mu,i}, p)=1$, the complex of sheaves $\omega_{\textbf{m}_i,\mu}^{\cdot}(=\frac{\omega_{\textbf{m}_{i}}^{\cdot}}{\omega_{\textbf{m}_{i+1}}^{\cdot}})$  are acyclic for each $i$.
\end{Lem}
\begin{proof}The proof is the same as the proof of \rm \cite[Lemma 3.4]{SS}. 
It suffices to show that $\omega_{\textbf{m}_{i}, \mu}^{\cdot}$ is acyclic if $(m_{\mu,i},p )=1$. Note that $ \omega_{D_{\mu}}^q$ is generated by $\Omega^q_{D_{\mu}}$ and the form $d\log(\pi_{\lambda})\land \eta$ with $\lambda \in \Lambda$ and $\eta \in \Omega^{q-1}_{D_{\mu}}$ (cf. \cite[Corollary 1.9]{Tsu2}).
There is a residue homomorphism $\Res^q: \omega^q_{D_{\mu}} \longrightarrow \omega^{q-1}_{D_{\mu}}$ which is characterized by the following properties for $q \geq 0$ (cf. see \cite[Lemma 3.4]{SS}):
 \begin{align*}&(1)\ For\ \omega \in \Omega_{D_{\mu}}^q,\ \Res^q(\omega)=0.\\
&(2)\ For\ \eta \in \Omega_{D_{\mu}}^{q-1},\ we\  have\ \
 \Res^q(d\log(\pi_{\lambda})\land \eta)=\begin{cases} \eta & \ \ \ (\lambda = \eta)\\  
0& \ \ \ (\lambda \neq \eta).   \end{cases} 
\end{align*}
We define a residue homomorphism
\[\Res^q : \omega_{\textbf{m}_{i}, \mu}^q \longrightarrow \omega_{\textbf{m}_{i}, \mu}^{q-1}\]
by $\Res(\alpha \otimes \omega):= \alpha \otimes \Res^q(\omega)$ for $\alpha \in \mathscr{O}_Y(-D^i_ s)$ and $\omega \in \omega^q_{D_{\mu}}$.
We have  \[ d\Res^q(x)+\Res^{q+1}(dx)=m_{\mu,i}\cdot x\ \ \ \ \ {\rm for\ any\ }x \in \omega_{\textbf{m}_{i}, \mu}^q\]
by the same computation as the proof of \cite[Lemma 3.4]{SS}.
This implies that $\omega_{\textbf{m}_{i}, \mu}^{\cdot}$ is acyclic if $(m_{\mu,i}, p)=1$. 
\end{proof}

Using above Lemma \ref{graded}, we obtain the isomorphism \[C^{-1}: \omega^{q}_Y\otimes \mathscr{O}_Y(-D_s')\xrightarrow{\cong}\mathcal{H}^{q}\big(\omega^{\cdot}_Y\otimes \mathscr{O}_Y(-p\cdot D_s') \big)\xrightarrow{\cong} \mathcal{H}^{q}\big(\omega^{\cdot}_Y\otimes \mathscr{O}_Y(-D_s)\big).\] 
Here the first isomorphism comes from the first case. This completes the proof of Lemma \ref{Omega}.
\end{proof}

For each integer $q \geq 0$, we have the following morphism which restricts a morphism (\ref{Carinv}) to $ \omega_{Y|D}^q$:
\begin{equation}
C^{-1}: \omega_{Y|D_s}^q \longrightarrow \mathcal{H}^q(\omega_{Y|D_s}^{\text{\large$\cdot$}})
\end{equation}
\begin{equation}
a\;\dlog(b_1)\land\dlog(b_2)\land...\land\dlog(b_q) \mapsto \tcf\;a^p\;\dlog(b_1)\land\dlog(b_2)\land...\land\dlog(b_q),
\end{equation}
\noindent
where $a \in \mathscr{O}_Y(-D_s)$, and $b_1,...,b_q \in M_Y$.

\begin{Lem}\label{omega log}{\rm(cf. \cite[Theorem 1.2.1, Proposition 1.2.3]{JSZ})}
We keep the notations and the assumptions as in \S3.
Then, for each integer $q \geq 0$, we have the following exact sequence.
\begin{equation}\label{diffmodulus}
0 \longrightarrow \omega_{Y|D_s,\log}^q \longrightarrow Z^q_{Y|D_s} \overset{1-C^{-1}}{\longrightarrow} \mathcal{H}^q(\omega_{Y|D_s}^{\text{\large$\cdot$}}) \longrightarrow 0.
\end{equation}
\end{Lem}
\begin{proof}\textbf{The surjectivity of $1-C^{-1}$} : It suffices to show  the surjectivity of $1-C^{-1}$ on sections over 
the strict henselisation of a local ring of $Y$. This follows form the following fact:

\noindent
\underline{Fact} (\cite[Lemma 1.2.2]{JSZ}.): {\it Let $A$ be a strictly henselian regular local ring of equi-characteristic $p>0$ and $\mathfrak{m} \subset A$ be the maximal ideal. Let $\pi \in \mathfrak{m}$ and $a \in A$. If $a \in \pi A$, then there exists $b \in A$, such that $b \in \pi A$ and $b^p-b=a$. }\\
\textbf{The exactness of the middle term} : It suffices to show that the exactness of the sequence
\[0 \longrightarrow \omega_{Y|D_s,\log}^q \longrightarrow \omega^q_{Y|D_s} \overset{1-C^{-1}}{\longrightarrow} \omega_{Y|D_s}^q/d\omega^{q-1}_{Y|D_s} \longrightarrow 0.\]
If we have this short exact sequence, we have \eqref{diffmodulus} because $ \omega_{Y|D_s,\log}^q \subset Z^q_{Y|D_s}$. We have the following commutative diagram:
\[
\xymatrix@M=10pt{ 0 \ar[r]& \omega_{Y,\log}^q \ar[r]& \omega^q_Y \ar[r]^-{1-C^{-1}} & \omega^q_Y/d\omega^{q-1}_Y \ar[r] &0\\
0 \ar[r]&\Ker(1-C^{-1}) \ar@<-0.3ex>@{^{(}->}[u] \ar[r]&\omega^q_{Y|D_s}  \ar@<-0.3ex>@{^{(}->}[u] \ar[r]^-{1-C^{-1}} &  \omega_{Y|D_s}^q/d\omega^{q-1}_{Y|D_s} \ar@<-0.3ex>@{^{(}->}[u] \ar[r] &0
}\] Here the upper horizontal row is proved by \cite[Theorem 6.1.1]{Tsu0} and \cite[Theorem A4]{Tsu2}.
Then it suffices to show that $ \omega_{Y|D_s,\log}^q=\omega_{Y, \log} ^q \cap \omega_{Y|D_s}^q\ (=\Ker(1-C^{-1}))$. This is an \'etale local problem and a consequence of Proposition \ref{diflog} below.  Let $R$ be the henselization of a local ring of $Y$.
We put $Y:=\Spec(R)$ and
\[\mathfrak{G}^q:=\Ker\left((T_1^{n_1}\cdots T_e^{n_e})\cdot \omega^q_Y\xrightarrow{1-C^{-1}} \omega^q_Y/ d\omega^{q-1}_Y\right),\]
where we choose a system of regular parameters $\{T_1,\dotsc,T_d\}$ of $R$ such that \[\Supp(D_s)=\Spec(R/(T^{n_1}_1\cdots T^{n_e}_e))\] for some $e \leq d=\dim R$.
\begin{Prop}\label{diflog}(cf. \cite[Proposition 1.2.3]{JSZ})
$\mathfrak{G^q}$ is generated by elements of the form
\[d\log(a)\land d\log(x_1)\land\cdots\land d\log(x_q)\ \ \ \ \ a\in 1+(T_1^{n_1}\cdots T_e^{n_e}),\ \ \ \ x_i \in R\left[\frac{1}{T_1^{n_1}\cdots T_e^{n_e}}\right]^{\times}\ \ (1\leq i \leq q).\]
\end{Prop}
The proof of Proposition \ref{diflog} is the same computations as the proof of \cite[Proposition 1.2.3]{JSZ}. This completes the proof.
\end{proof}

We have the following main results. 
\begin{Thm}\label{main result}Let $n \geq 1$ be an integer. If $0 \leq r \leq p-2$ and $p \geq 3$, 
the cokernel of the symbol map \[ {\rm Symb}_{X|D}: (1+I_{D_{n+1}})^{\times} \otimes (M_{X_{n+1}}^{gp})^{\otimes{r-1}} \longrightarrow \mathcal{H}^r(\mathscr{S}_n(r)_{X|D})\] is Mittag-Leffler zero with respect to the multiplicities of the prime components of $D$.
\end{Thm}

\begin{Thm}\label{Main result}
We assume that $p \geq 3$. Let $e$ be the absolute ramification index of $K$. Then the sheaf $\mathcal{H}^q\big(s_1(q)_{X|D}\big)$ has the folllowing structure:
\begin{enumerate}
\item\;For $m=0$, we have short exact sequences:\\
 \begin{equation*} 
0\longrightarrow\frac{\mathcal{R}}{\mathcal{R}\cap\gr_1^0\mathcal{H}^q\big(s_1(q)_{X|D}\big)}\longrightarrow\gr_0^0\mathcal{H}^q\big(s_1(q)_{X|D}\big) \longrightarrow \omega_{Y|D_s,\log}^q\longrightarrow 0,\end{equation*}
\begin{equation*}
\quad\quad\quad\quad\quad\quad\quad\quad\quad\quad\quad\quad\quad\quad\quad\quad\quad\quad\quad\quad\{x, a_1,\dotsc, a_{q-1}\}\mapsto d\log \overline{x} \land d\log \overline{ a_1} \land \cdots \land d\log \overline{a_q}
\end{equation*}
Here $x \in (1+I_{D_2})^{\times}$, $a_1, \dotsc, a_{q-1} \in M_{X_2}^{gp}$ and $y \in \mathscr{O}_{X_2}(-D_2)$. We denote by $\overline{x}$ {\rm (resp. }$\overline{a_i}${\rm )} the image of $x$ {\rm (resp.} $a_i${\rm)} in $M_Y^{gp}$, and we denote by $\overline{y}$ the image of $y$ in $\mathscr{O}_Y(-D_s)$.

\begin{equation*} 
0\longrightarrow\mathcal{R}\cap\gr_1^0\mathcal{H}^q\big(s_1(q)_{X|D}\big)\longrightarrow\gr_1^0\mathcal{H}^q\big(s_1(q)_{X|D}\big)\longrightarrow \omega_{Y|D_s,\log}^{q-1}
\longrightarrow 0,\end{equation*}
\begin{equation*}
\quad\quad\quad\quad\quad\quad\quad\quad\quad\quad\quad\quad\quad\quad\quad\quad\quad\quad\quad\quad\{x, a_1,\dotsc, a_{q-2},\pi\}\mapsto d\log \overline{x} \land  d\log \overline{ a_1} \land \cdots \land d\log \overline{a_{q-2}}
\end{equation*}
where 
\[\mathcal{R}:=\Ker\Big(\gr_U^0\mathcal{H}^q(S_D^{\cdot})\rightarrow  \Ker\big(Z^q\big( \mathscr{O}_Y \otimes_{\mathscr{O}_{Z_1}} \omega_{Z_1|\mathscr{D}_1}^{\cdot}\big)\xrightarrow{1-\varphi\otimes \land^{q}d\varphi/p} \mathcal{H}^q\big( \mathscr{O}_Y \otimes_{\mathscr{O}_{Z_1}} \omega_{Z_1|\mathscr{D}_1}^{\cdot}\big)\big)\Big).\]

\item\;If $0 <m <pe/(p-1)$ and $p \nmid  m$, then we have
\begin{equation*} 
\gr_0^m\mathcal{H}^q\big(s_1(q)_{X|D}\big) \cong \frac{\omega_{Y|D_s}^{q-1}}{B_{Y|D_s}^{q-1}}\end{equation*}
\begin{equation*}
\{1+\pi^my, a_1,\dotsc, a_{q-1}\}\mapsto \overline{y}d\log \overline{ a_1} \land \cdots \land d\log \overline{a_q}
\end{equation*}
\begin{equation*} 
\gr_1^m\mathcal{H}^q\big(s_1(q)_{X|D}\big) \cong \frac{\omega_{Y|D_s}^{q-2}}{Z_{Y|D_s}^{q-2}}\end{equation*}
\begin{equation*}
 \{1+\pi^my, a_1,\dotsc, a_{q-2}, \pi\}\mapsto \overline{y}d\log \overline{ a_1} \land \cdots \land d\log \overline{a_{q-2}}
\end{equation*}
\item\;If $0 <m <pe/(p-1)$ and $p | m$, then we have short exact sequences
\begin{equation*} 
0\longrightarrow\frac{\mathcal{L}^m}{\mathcal{L}^m\cap \mathcal{H}^q\big(s_1(q)_{X|D}\big)}\longrightarrow\gr_0^m\mathcal{H}^q\big(s_1(q)_{X|D}\big) \longrightarrow \frac{\omega_{Y|D_s}^{q-1}}{Z_{Y|D_s}^{q-1}}\rightarrow 0,\end{equation*}
\begin{equation*}
\quad\quad\quad\quad\quad\quad\quad\quad\quad\quad\quad\quad\quad\quad\quad\quad\quad\quad\quad\quad\{1+\pi^my, a_1,\dotsc, a_{q-1}\}\mapsto \overline{y}d\log \overline{ a_1} \land \cdots \land d\log \overline{a_q}
\end{equation*}

\begin{equation*} 
0\longrightarrow  \mathcal{L}^m\cap \mathcal{H}^q\big(s_1(q)_{X|D}\big)\longrightarrow\gr_1^m\mathcal{H}^q\big(s_1(q)_{X|D}\big) \longrightarrow \frac{\omega_{Y|D}^{q-2}}{Z_{Y|D}^{q-2}}\longrightarrow 0,\end{equation*}
\begin{equation*}
 \quad\quad\quad\quad\quad\quad\quad\quad\quad\quad\quad\quad\quad\quad\quad\quad\quad\quad\quad\quad\{1+\pi^my, a_1,\dotsc, a_{q-2}, \pi\}\mapsto \overline{y}d\log \overline{ a_1} \land \cdots \land d\log \overline{a_{q-2}}
\end{equation*}
where  $\mathcal{L}^m$ is a certain subsheaf of $\gr_{U}^m\mathcal{H}^q\big(s_1(q)_{X|D}\big)$ which is given more explicitly in a sufficientlly local situation (see \eqref{L^m} in Lemma \ref{Lem8} below).\item\;If $m \geq pe/(p-1)$, then $U^m\mathcal{H}^q\big(s_1(q)_{X|D}\big)=0$.
\end{enumerate}
\end{Thm}

\section{\textbf{ Proof of Main Results}}

\subsection{Proof of Theorem \ref{main result}}We put $\Theta_n:=(1+I_{D_{n+1}})^{\times} \otimes (M_{X_{n+1}}^{gp})^{\otimes{r-1}}$.
We consider the diagram
\[
\xymatrix@M=10pt{
& \Theta_n \ar[d]^-{\Symb_{X|D}} \ar[r]^-{\times p} & \Theta_{n+1} \ar[d]^-{\Symb_{X|D}} \ar[r]&\Theta_1 \ar[r] \ar[d]^-{\Symb_{X|D} }&0\\
 \cdots \ar[r]&\mathcal{H}^r(\mathscr{S}_n(r)_{X|D})\ar[r]^-{\times p}& \mathcal{H}^r(\mathscr{S}_{n+1}(r)_{X|D}) \ar[r]& \mathcal{H}^r(\mathscr{S}_1(r)_{X|D}) \ar[r]&\cdots,   
}\]
where the lower horizontal line is the long exact sequence which is obtained by Lemma \ref{dist}. By this diagram, the assertion is reduced to the case $n=1$. 
Then we show the claim in the case $n=1$.
By Lemma \ref{CokernelSymb} and Lemma \ref{compU} below, the cokernel of the morphism \[\gr_{U}^m\big((1+I_{D_2})^{\times} \otimes (M_{X_2}^{gp})^{\otimes{q-1}}\big)\longrightarrow \gr_{U}^m\mathcal{H}^q(s_1(q)_{X|D})\] will be Mittag-Leffler zero with respect to the multiplicities of the prime components of $D$. Then we will obtain that $\Coker({\rm Symb}_{X|D})$ is Mittag-Leffler zero by the finiteness of the filtration $\{U^m\}_{m \in \mathbb{N}}$ in Theorem \ref{Main result} $(4)$. $\square$
\subsection{Proof of Theorem \ref{Main result}}
If $m=0$, by \eqref{K} in Lemma \ref{Lem2} $(3)$ and Lemma \ref{Lem8} $(1)$ below, we have the following diagram of the short exact sequences:
\[
\xymatrix@M=10pt{
0 \ar[r] & \gr_1^0\mathcal{H}^q(s_1(q)_{X|D}) \ar@{->>}[d] \ar[r]&\gr_U^0\mathcal{H}^q(s_1(q)_{X|D}) \ar@{->>}[d] \ar[r]& \gr_0^0\mathcal{H}^q(s_1(q)_{X|D})\ar[d] \ar[r]&0\\
0 \ar[r]&\omega_{Y|D_s, \log}^{q-1} \ar[r]& \mathcal{K} \ar[r] & \omega_{Y|D_s, \log}^{q} \ar[r] & 0,
 }\]
where the surjectivity of the left and middle vertical arrows are form Lemma \ref{Lem8} $(1)$ and Lemma \ref{Lem2} (3) below. Here we put
\[\mathcal{K}:=\Ker\left(Z^q\big( \mathscr{O}_Y \otimes_{\mathscr{O}_{Z_1}} \omega_{Z_1|\mathscr{D}_1}^{\cdot}\big)\xrightarrow{1-\varphi\otimes \land^{q}d\varphi/p} \mathcal{H}^q\big( \mathscr{O}_Y \otimes_{\mathscr{O}_{Z_1}} \omega_{Z_1|\mathscr{D}_1}^{\cdot}\big)\right).\]
By the snake lemma, we have two short exact sequences in the assertion $(1)$. 
If $0< m < pe/(p-1)$ and $p\nmid m$, we consider the following diagram
{\tiny \[
\xymatrix@M=10pt{
0 \ar[r] & \gr_1^m\mathcal{H}^q(s_1(q)_{X|D}) \ar@{->>}[d] \ar[r]&\gr_U^m\mathcal{H}^q(s_1(q)_{X|D}) \ar[d]^-{\cong} \ar[r]& \gr_0^m\mathcal{H}^q(s_1(q)_{X|D})\ar[d] \ar[r]&0\\
0 \ar[r]& \frac{\omega_{Y|D_s}^{q-2}}{Z_{Y|D_s}^{q-2}} \ar[r]&B^q \left(\left(\frac{T^m\mathscr{O}_{Z_1}}{T^{m+1}\mathscr{O}_{Z_1}}\right)\otimes \omega_{Z_1|\mathscr{D}_1}^{\cdot}\right)\ar[r] & \frac{\omega_{Y|D_s}^{q-1}}{B_{Y|D_s}^{q-1}} \ar[r] & 0,
 }\]}
 where the surjectivity of the left vertical arrow and the isomorphism of the middle vertical arrow are form Lemma \ref{Lem8} $(2)$ and Lemma \ref{Lem2} (1) below. By the snake lemma,  we obtain the isomorphisms in the assertion $(2)$.
If $0< m < pe/(p-1)$ and $p | m$, by \eqref{K} in Lemma \ref{Lem2} $(1)$ and Lemma \ref{Lem8} $(2)$ below, we have the following diagram
{\tiny \[
\xymatrix@M=10pt{
0 \ar[r] & \gr_1^m\mathcal{H}^q(s_1(q)_{X|D}) \ar@{->>}[d] \ar[r]&\gr_U^m\mathcal{H}^q(s_1(q)_{X|D}) \ar@{->>}[d] \ar[r]& \gr_0^m\mathcal{H}^q(s_1(q)_{X|D})\ar[d] \ar[r]&0\\
0 \ar[r]& \frac{\omega_{Y|D_s}^{q-2}}{Z_{Y|D_s}^{q-2}} \ar[r]&B^q \left(\left(\frac{T^m\mathscr{O}_{Z_1}}{T^{m+1}\mathscr{O}_{Z_1}}\right)\otimes \omega_{Z_1|\mathscr{D}_1}^{\cdot}\right)\ar[r] & \frac{\omega_{Y|D_s}^{q-1}}{Z_{Y|D_s}^{q-1}} \ar[r] & 0.
 }\]}
Here the surjectivity of the left and middle vertical arrows are form Lemma \ref{Lem8} $(3)$ and Lemma \ref{Lem2} (2) below. By the snake lemma, we obtain the assertion $(3)$. 
From Lemma \ref{Lem5} ($3$) below, we will obtain $\gr_{U}^m\mathcal{H}^q(S_D^{\cdot})=0$ for $pe/(p-1) \leq m <pe$.  Since $U^{pe}\mathcal{H}^q(S_D^{\cdot})=0$ by Lemma \ref{Lem3} and Corollary \ref{Lem6} below, this implies ($4$). This completes the proof of Theorem \ref{Main result}.  $\square$\\

In the rest of this section we prove the lemmas that have been mentioned in the above proof of Theorem \ref{main result} and Theorem \ref{Main result}.
We will work with the following local situation. We keep the \textbf{Assumption \ref{Asum*}} in the following sections.
\subsection{Local computation}\label{Loc} We denote by $(S,N)$ the scheme $\Spec(\mathscr{O}_K)$ with log sturcture $N$ defined by the closed point.
Let $(V,M_V)$ be the scheme $\Spec(W[T])$ with the log structure defined by the divisor $\{T=0\}$, and let $i_V: (S,N) \rightarrow (V,M_V)$ be the exact closed immersion defined by $T \mapsto \pi$.
We assume that there exists a factorization $(Z,M_Z)\rightarrow(V,M_V)\rightarrow \Spec(W)$ such that $(Z,M_Z)\rightarrow(V,M_V)$ is smooth and compatible with the liftings of frobenii, and such  that the following diagram is cartesian (the left cartesian diagram is mentioned in \textbf{Assumption} \ref{Asum*}): 
\[
\xymatrix@M=10pt{   
(D, M_D)\ar@{}[dr]|{\square}\ar[d]^{\beta_D}\ar[r]&(X,M_X)\ar@{}[dr]|{\square}\ar[d]^{\beta} \ar[r]&(S,N) \ar[d]^{i_V}&\\
(\mathscr{D}, M_{\mathscr{D}})\ar[r]&(Z,M_Z)\ar[r]&(V,M_V).
}
\]
We define the liftings of Frobenius $(V, M_V) \longrightarrow (V, M_V)$ by the Frobenius of $W$ and $T \mapsto T^p$.
These assumptions are \'etale locally: We have the following diagram
\[
\xymatrix@M=10pt{   
X=\Spec \mathscr{O}_K[t_1,\dotsc, t_d]/(t_1\cdots t_d-\pi)\ar[r]\ar[d]^-{\beta}&S=\Spec\mathscr{O}_K\ar[d]^-{i_V}\\
Z=\Spec W[T_0,T_1,\dotsc,T_d,T_{\infty}]\ar[r]&V=\Spec W[T].
}\]
Here $\beta$ is defined by $T_0\mapsto \pi,\ T_i\mapsto t_i\ (1\leq i \leq d),\ T_{\infty}\mapsto 0$. The lower horizontal map is defined by $T\mapsto T_0$.
Then this diagram is commutative and the morphism $(Z,M_Z)\rightarrow(V,M_V)$ is smooth and compatible with the liftings of frobenii.

\begin{Lem}\label{Lem1} Let $n$ be a non-negative integer.
\begin{enumerate}
\item\;From the reduction$\mod T $ of the short exact sequence
 \begin{equation}\label{(1)}
0 \longrightarrow \omega_{Z_1/V_1}^{q-1}\otimes_{\mathscr{O}_{Z_1}} \mathscr{O}_{Z_1}(-\mathscr{D}_1)\xrightarrow{\land \dlog T} \omega_{Z_1|\mathscr{D}_1}^{q}\longrightarrow \omega_{Z_1/V_1}^{q}\otimes_{\mathscr{O}_{Z_1}} \mathscr{O}_{Z_1}(-\mathscr{D}_1)\longrightarrow0,
\end{equation}
and the $\mathscr{O}_{Z_1}$-linear isomorphism
\begin{equation}\label{(2)}
\mathscr{O}_{Y} \otimes_{\mathscr{O}_{Z_1}} \omega_{Z_1|\mathscr{D}_1}^{q} \stackrel{\cong}{\longrightarrow}
\big(T^{m}\mathscr{O}_{Z_1}/T^{m+1}\mathscr{O}_{Z_1}\big)\otimes \omega_{Z_1|\mathscr{D}_1}^q
\end{equation}
induced by the multiplication by $T^m$ on $ \omega_{Z_1|\mathscr{D}_1}^{q}$ for each integer $q \geq 0$, we obtain a short exact sequence of complexes:
\begin{equation}\label{omega}
0 \longrightarrow \omega_{Y|D_s}^{\text{\large$\cdot$}}[-1] \longrightarrow \big(T^{m}\mathscr{O}_{Z_1}/T^{m+1}\mathscr{O}_{Z_1}\big)\otimes \omega_{Z_1|\mathscr{D}_1}^{\cdot} \longrightarrow\omega_{Y|D_s}^{\text{\large$\cdot$}} \longrightarrow0.
\end{equation}
$(*)$\; Furthermore, for each integer $q \geq 0$, the connecting homomorphism \[\mathcal{H}^q(\omega_{Y|D_s}^{\text{\large$\cdot$}})\rightarrow\mathcal{H}^q(\omega_{Y|D_s}^{\text{\large$\cdot$}})\] of the long exact sequence associated to {\rm (\ref{omega})} is the multiplication by $(-1)^qm$. In particular, it is the zero map if $p | m$.\\
\item\;If $p\nmid m$, $\mathcal{H}^q(\big(T^{m}\mathscr{O}_{Z_1}/T^{m+1}\mathscr{O}_{Z_1}\big)\otimes \omega_{Z_1|\mathscr{D}_1}^{\cdot})=0.$\\
\item\;If $p|m$, there is an isomorphism:
\begin{equation}
T^m\cdot\Big(\varphi \otimes \land^q \frac{d\varphi}{p} \Big): \mathscr{O}_{Y} \otimes_{\mathscr{O}_{Z_1}} \omega_{Z_1|\mathscr{D'}_1}^{q} \stackrel{\cong}{\longrightarrow}
\mathcal{H}^q\Big(\big(T^{m}\mathscr{O}_{Z_1}/T^{m+1}\mathscr{O}_{Z_1}\big)\otimes \omega_{Z_1|\mathscr{D}_1}^{\cdot}\Big).
\end{equation}
\end{enumerate}
\end{Lem}
\begin{proof}
The assertions \ref{(1)}, \ref{(2)} and \ref{omega} are easily follows from \cite[Lemma 2.4.2]{Tsu1} and $(2)$ follows from $(*)$. 
We prove $(*)$ and ($3$). There is a commutative diagram of complexes with exact rows which comes from \eqref{omega}:\\
$$
\begin{CD}   
0 @>>>\omega_{Y|D_s'}^{q-1}[-q]@>\land d\log(T)>>\mathscr{O}_Y \otimes \omega_{Z_1|\mathscr{D'_1}}^{q}[-q]@>>>\omega_{Y|D_s'}^{q}[-q]@>>> 0\\
@.@VV\land^{q-1}\frac{d\varphi}{p}V@VVT^m\cdot\big(\varphi \otimes \land^{q}\frac{d\varphi}{p}\big)V@VV\land^{q}\frac{d\varphi}{p}V \\
0 @>>>\omega_{Y|D_s}^{\cdot}[-1]@ >T^m\cdot(\land d\log(T))>>\Big(\frac{T^{m}\mathscr{O}_{Z_1}}{T^{m+1}\mathscr{O}_{Z_1}}\Big)\otimes \omega_{Z_1|\mathscr{D}_1}^{\cdot}@>>>\omega_{Y|D_s}^{\cdot}@>>> 0
\end{CD}
$$\\
and taking cohomology, we get the following commutative diagram:

{\tiny$$\begin{CD}   
0@>>>\omega_{Y|D_s'}^{q-1}@>\land d\log(T)>>\mathscr{O}_Y \otimes \omega_{Z_1|\mathscr{D'_1}}^{q}@>>>\omega_{Y|D_s'}^{q}@>>>0\\
@VVV@VVC^{-1}V@VVT^m\cdot\big(\varphi \otimes \land^{q}\frac{d\varphi}{p}\big)V@VVC^{-1}V@VVV \\
\mathcal{H}^{q-1}(\omega_{Y|D_s}^{\cdot})@>>>\mathcal{H}^{q-1}(\omega_{Y|D_s}^{\cdot})@ >T^m\cdot(\land d\log(T))>>\mathcal{H}^q\Big(\big(\frac{T^{m}\mathscr{O}_{Z_1}}{T^{m+1}\mathscr{O}_{Z_1}}\big)\otimes \omega_{Z_1|\mathscr{D}_1}^{\cdot}\Big)@>>>\mathcal{H}^{q}(\omega_{Y|D_s}^{\cdot})@>>> \mathcal{H}^{q}(\omega_{Y|D_s}^{\cdot}),
\end{CD}
$$}\\
where $C^{-1}$ is the inverse Cartier morphism. We have the following lemma:
\begin{Lem}(cf.\cite[Lemma 7.1.4]{Tsu0})
For the map $\frac{d\varphi}{p}$ of $\omega_{Z_1|\mathscr{D}_1}^1$, we have 
\[\land^q \frac{d\varphi}{p} (a\cdot d\log b_1 \land \cdots \land d\log b_q) \equiv a^p\cdot d\log b_1 \land \cdots \land d\log b_q \mod d(\omega_{Z_1|\mathscr{D}_1}^{q-1})\]
for $a \in \mathscr{O}_{Z_1}(-\mathscr{D}_1), b_1, \dotsc b_q \in M_{Z_1}$.
\end{Lem}
\begin{proof} This proof is the same argument as \cite[Lemma 7.1.4]{Tsu0}. It suffices to show that $\frac{d\varphi}{p}(a\cdot d\log b) \equiv a^p \cdot d\log b \mod d(\mathscr{O}_{Z_1}(-\mathscr{D}_1))$. Let $b \in M_{Z_1}$ be the image of $b' \in M_{Z_2}$ and $a \in \mathscr{O}_{Z_1}(-\mathscr{D}_1)$ be the image of $\alpha \in \mathscr{O}_{Z_2}(-\mathscr{D}_2)$. If we have $\varphi_{Z_1}(b)=b^p$, there exists $x \in \mathscr{O}_{Z_2}$ such that
\[\varphi_{Z_2}(b')= b'^p\cdot(1+p\cdot x).\]
Then we have \[d\varphi_{Z_2}(\alpha \cdot d\log b')=\alpha ^p\cdot \{p\cdot d\log b' + p\cdot (1+p\cdot x)^{-1}\cdot dx\}.  \]
Thus we obtain $\frac{d\varphi}{p}(a\cdot d\log b) \equiv a^p\cdot d\log b \mod d(\mathscr{O}_{Z_1}(-\mathscr{D}_1))$.
\end{proof}
The commutativity of the above two diagram follows from this Lemma, the fact $\frac{d\varphi}{p}(d\log T)=d\log T$ and the characterization of Cartier isomorphism.
Then we have (3) from Lemma \ref{Omega}.
We prove the claim $(*)$ by the same argument as in the proof of \cite[Lemma 7.4.3 (2)]{Tsu0}. We note that $T^m\omega' \in \left(T^m\mathscr{O}_{Z_1}/T^{m+1}\mathscr{O}_{Z_1}\right)\otimes \omega^q_{Z_1|\mathscr{D}_1}$ is a lifting of $\omega \in Z^q_{Y|D_s}$. Thus we have 
\[d(T^m\omega' )=mT^md\log T \land \omega' + T^m d\omega'.\]
The image of $T^md\omega'$ in $\omega^{q+1}_{Y|D_s}$ is $d\omega'  \mod d\log T =d\omega =0.$ Then there exists $\eta \in \omega^q_{Y|D_s}$ such that $T^md\omega'=T^m(\eta \land d\log T)$.  Hence we have 
\begin{align*}
d(T^m\omega' )&=mT^md\log T \land \omega' + T^m(\eta \land d\log T)\\
&=T^m\{ \left((-1)^m\cdot m \omega + \eta \right) \land d\log T \}.
\end{align*}
Then $\delta : \mathcal{H}^q(\omega^{\cdot}_{Y|D_s}) \longrightarrow \mathcal{H}^q(\omega^{\cdot}_{Y|D_s})$ maps the class of $\omega$ to the class of $(-1)^m\cdot m \omega+\eta$. If $m=0$, we have $\delta=0$ by the above proof of $(4)$ so that the class of $\eta$ is $0$. We obtain the claim $(2)$. This completes the proof.
\end{proof}

\begin{Lem}\label{Lem2} Let $m$ be a non-negative integer.
\begin{enumerate}
 \item\;If $p\nmid m$, there is a short exact sequence
\begin{equation}
0 \longrightarrow \frac{\omega_{Y|D_s}^{q-2}}{Z_{Y|D_s}^{q-2}}\longrightarrow B^q\Big(\big(T^{m}\mathscr{O}_{Z_1}/T^{m+1}\mathscr{O}_{Z_1}\big)\otimes \omega_{Z_1|\mathscr{D}_1}^{\cdot}\Big)\longrightarrow  \frac{\omega_{Y|D_s}^{q-1}}{B_{Y|D_s}^{q-1}}\longrightarrow 0
\end{equation}
which is characterized by the following properties. For $x \in \mathscr{O}_Y(-\mathscr{D}_1)$ and $a_1,\dotsc,  a_{q-1} \in M_{Z_1}^{gp}$, the image of 
\begin{equation}
d\big(T^m x \otimes d\log a_1 \land \cdots \land d\log a_{q-1}\big) \in B^q\Big(\big(T^{m}\mathscr{O}_{Z_1}/T^{m+1}\mathscr{O}_{Z_1}\big)\otimes \omega_{Z_1|\mathscr{D}}^{\cdot}\Big) 
\end{equation}
in $ \frac{\omega_{Y|D_s}^{q-1}}{B_{Y|D_s}^{q-1}}$ is $x d\log\overline{a_1}\land \cdots \land d\log\overline{a_{q-1}}$, and 
\begin{equation}
d\big(T^m x \otimes d\log a_1 \land \cdots \land d\log a_{q-2} \land d\log T\big) \in B^q\Big(\big(T^{m}\mathscr{O}_{Z_1}/T^{m+1}\mathscr{O}_{Z_1}\big)\otimes \omega_{Z_1|\mathscr{D}_1}^{\cdot}\Big) 
\end{equation}
is the image of $x d\log\overline{a_1}\land \cdots \land d\log\overline{a_{q-2}} \in \frac{\omega_{Y|D_s}^{q-2}}{Z_{Y|D_s}^{q-2}}$, where $\overline{a_i}$ denote the images of 
$a_i$ in $M_Y^{gp}$.

\item\;If $p|m$, there is a short exact sequence
\begin{equation}
0 \longrightarrow \frac{\omega_{Y|D_s}^{q-2}}{Z_{Y|D_s}^{q-2}}\longrightarrow  B^q\Big(\big(T^{m}\mathscr{O}_{Z_1}/T^{m+1}\mathscr{O}_{Z_1}\big)\otimes \omega_{Z_1|\mathscr{D}_1}^{\cdot}\Big)\longrightarrow  \frac{\omega_{Y|D_s}^{q-1}}{Z_{Y|D_s}^{q-1}}\longrightarrow 0
\end{equation}
which is characterized in the same way as {\rm (1)}.

\item\;The  homomorphism 
\begin{equation}
1-\varphi\otimes\land^{q}\frac{d\varphi}{p}:Z^q\big(\mathscr{O}_Y \otimes \omega_{Z_1|\mathscr{D}_1}^{\cdot}\big)\rightarrow \mathcal{H}^q\big(\mathscr{O}_Y \otimes \omega_{Z_1|\mathscr{D}_1}^{\cdot}  \big) 
\end{equation}
is surjective. Its kernel $\mathcal{K}$ is the subsheaf of abelian groups of $Z^q\big(\mathscr{O}_Y \otimes \omega_{Z_1|\mathscr{D}_1}^{\cdot}\big)$ generated by local sections of the form 
\begin{equation}
1\otimes d\log(a_1)\land d\log(a_2)\land\cdots\land d\log(a_q),  \quad( a_1 \in 1+\mathscr{O}_{Z_1}(-\mathscr{D}_1), a_2\dotsc, a_q \in M_{Z_1}^{gp}.\big)
\end{equation}
and there is a short exact sequence
\begin{equation}\label{K}
0 \rightarrow \omega_{Y|D_s,\log}^{q-1}  \rightarrow \mathcal{K} \rightarrow  \omega_{Y|D_s,\log}^{q} \rightarrow 0
\end{equation}
which is characterized by the following properties: \\
For $a_1 \in 1+\mathscr{O}_{Z_1}(-\mathscr{D}_1), a_2, \dotsc, a_q \in M_{Z_1}^{gp}$, the image of 
\begin{equation}
1\otimes d\log(a_1)\land d\log(a_2) \land \cdots \land d\log(a_q)\in \mathcal{K} 
\end{equation}
in $\omega_{Y|D,\log}^{q}$ is $d\log(\overline{a_1})\land d\log(\overline{a_2}) \land \cdots \land d\log(\overline{a_q})$, and 
\begin{equation}
1\otimes d\log(a_1)\land d\log(a_2) \land\cdots\land d\log(a_{q-1})\land d\log T\in \mathcal{K} 
\end{equation}
is the image of $d\log(\overline{a_1})\land d\log(\overline{a_2}) \land\cdots\land d\log(\overline{a_{q-1}}) \in  \omega_{Y|D,\log}^{q-1}$, where $\overline{a_i}$ denote the images of $a_i$ in $M_Y^{gp}$.
\end{enumerate}
\end{Lem}
\begin{proof} 
If $p \nmid m$, $Z^{q-1}\Big(\big(T^{m}\mathscr{O}_{Z_1}/T^{m+1}\mathscr{O}_{Z_1}\big)\otimes \omega_{Z_1|\mathscr{D}_1}^{\cdot}\Big)=B^{q-1}\Big(\big(T^{m}\mathscr{O}_{Z_1}/T^{m+1}\mathscr{O}_{Z_1}\big)\otimes \omega_{Z_1|\mathscr{D}_1}^{\cdot}\Big)$ by Lemma \ref{Lem1} ($2$). Then we have from \eqref{omega} the following exact sequence:
\begin{equation}
0\rightarrow Z_{Y|D_s}^{q-2}\rightarrow Z^{q-1}\Big(\big(T^{m}\mathscr{O}_{Z_1}/T^{m+1}\mathscr{O}_{Z_1}\big)\otimes \omega_{Z_1|\mathscr{D}_1}^{\cdot}\Big) \rightarrow B_{Y|D_s}^{q-1}\rightarrow 0.
\end{equation}
If $p |m$, the homomorphism  
\begin{equation}
\mathcal{H}^{q-1}\Big(\big(T^{m}\mathscr{O}_{Z_1}/T^{m+1}\mathscr{O}_{Z_1}\big)\otimes \omega_{Z_1|\mathscr{D}_1}^{\cdot}\Big)\rightarrow \mathcal{H}^{q-1}(\omega_{Y|D_s}^{\cdot})
\end{equation}
is surjective by Lemma \ref{Lem1} ($1$). Hence the homomorphism 
\begin{equation}
Z^{q-1}\Big(\big(T^{m}\mathscr{O}_{Z_1}/T^{m+1}\mathscr{O}_{Z_1}\big)\otimes \omega_{Z_1|\mathscr{D}_1}^{\cdot}\Big)\rightarrow Z^{q-1}(\omega_{Y|D_s}^{\cdot})
\end{equation}
is surjective  and  \ref{omega} induces a short exact sequence:
\begin{equation}\label{T^mseq}
0\rightarrow Z_{Y|D_s}^{q-2}\rightarrow Z^{q-1}\Big(\big(T^{m}\mathscr{O}_{Z_1}/T^{m+1}\mathscr{O}_{Z_1}\big)\otimes \omega_{Z_1|\mathscr{D}_1}^{\cdot}\Big) \rightarrow Z_{Y|D_s}^{q-1}\rightarrow 0.
\end{equation}
(1) and (2) follows from these two short exact sequences and \eqref{omega}. 
We will prove $(3)$. 
We have a commutative diagram
{\tiny \[
\xymatrix@M=10pt{
0\ar[d]\ar[r] &Z_{Y|D_s}^{q-1}\ar[r]^-{\land d\log T}\ar[d]^-{1-C^{-1}} &Z^{q-1}\Big(\mathscr{O}_Y\otimes \omega_{Z_1|\mathscr{D}_1}^{\cdot}\Big)\ar[d]^-{1-\varphi\otimes\land^{q}\frac{d\varphi}{p}} \ar[r]& Z_{Y|D_s}^{q}\ar[r]\ar[d]^-{1-C^{-1}} &0\ar[d]\\
 \mathcal{H}^{q-1}(\omega_{Y|D_s}^{\text{\large$\cdot$}})\ar[r]& \mathcal{H}^{q-1}(\omega_{Y|D_s}^{\text{\large$\cdot$}})\ar[r]^-{\land d\log T} &  \mathcal{H}^{q-1}\Big(\mathscr{O}_Y\otimes \omega_{Z_1|\mathscr{D}_1}^{\cdot}\Big)  \ar[r] &  \mathcal{H}^q(\omega_{Y|D_s}^{\text{\large$\cdot$}})\ar[r]& \mathcal{H}^q(\omega_{Y|D_s}^{\text{\large$\cdot$}}), }\]}
where the upper horizontal short exact sequence is the case $m=0$ of \eqref{T^mseq}. The second and fourth morphism $1-C^{-1}$ is surjective by Lemma \ref{omega log}. Then the middle morphism is surjective. By the snake lemma in the above commutative diagram, we have the short exact sequence \eqref{K}.
Hence we obtain the claim $(3)$.  This completes the proof.
\end{proof}

Let $A^{\cdot}$(resp. $B^{\cdot}$) be the subcomplex of $J_{\mathscr{E}_1}^{[q-\cdot]} \otimes \omega_{Z_1|\mathscr{D}_1}^{\cdot}$ (resp. $\mathscr{O}_{\mathscr{E}_1}\otimes \omega_{Z_1|\mathscr{D}_1}^{\cdot}$) which coincides with $J_{\mathscr{E}_1}^{[q-\cdot]} \otimes \omega_{Z_1|\mathscr{D}_1}^{\cdot}$ (resp. $\mathscr{O}_{\mathscr{E}_1}\otimes \omega_{Z_1|\mathscr{D}_1}^{\cdot}$) in degree $q-1$, $q$, and $q+1$ (resp. degree $q-2$, $q-1$, and $q$), and is $0$ in other degree.
\begin{Lem}
 The inclusion map (resp. $\varphi_1\otimes \land^{q-1}d\varphi/p$) $J_{\mathscr{E}_1}^{[q-\cdot]} \otimes \omega_{Z_1|\mathscr{D}_1}^{q-1} \rightarrow \mathscr{O}_{\mathscr{E}_1}\otimes \omega_{Z_1|\mathscr{D}_1}^{q-1}$ and the identity map (resp. $\varphi \otimes \land^{q}d\varphi/p$) $\mathscr{O}_{\mathscr{E}_1}\otimes \omega_{Z_1|\mathscr{D}_1}^{q} \rightarrow \mathscr{O}_{\mathscr{E}_1}\otimes \omega_{Z_1|\mathscr{D}_1}^{q} $ give a morphism of complexes $1$ (resp. $\varphi_q$) : $A^{\cdot}\rightarrow B^{\cdot}$.
 \end{Lem}
 \begin{proof} It is obvious  that the morphism $1: A^{\cdot}\longrightarrow B^{\cdot}$ is a morphism of complex. We consider the case $\varphi_q : A^{\cdot}\longrightarrow B^{\cdot}$. It suffices to show that the following diagram commutative
 \[
\xymatrix@M=10pt{
J_{\mathscr{E}_1}^{[1]} \otimes \omega_{Z_1|\mathscr{D}_1}^{q-1}\ar[d]^-{\varphi_1\otimes \land^{q-1}d\varphi/p}\ar[r]^-{d} &\mathscr{O}_{\mathscr{E}_1}\otimes \omega_{Z_1|\mathscr{D}_1}^{q}\ar[d]^-{\varphi \otimes \land^{q}d\varphi/p}\\
\mathscr{O}_{\mathscr{E}_1}\otimes \omega_{Z_1|\mathscr{D}_1}^{q-1} \ar[r]^-{d} & \mathscr{O}_{\mathscr{E}_1}\otimes \omega_{Z_1|\mathscr{D}_1}^{q}\\
}\]
By \eqref{p^r} and by the definition of $\varphi_1$, we have the commutativity of the above diagram. This completes the proof.
 \end{proof}
We put ${S_D}^{\cdot}$ the mapping fiber of the morphism $1-\varphi_q:A^{\cdot}\rightarrow B^{\cdot}$. Then we have $\mathcal{H}^q(s_1(q)_{X|D})=\mathcal{H}^q({S_D}^{\cdot})$.

We define the descending filtration ${\Tilde{U}}^m$\;($0 \leq m \leq pe$) on $A^{\cdot}$(resp. $B^{\cdot}$) as follows: 
\begin{equation}
\cdots \rightarrow0 \rightarrow \big(T^{m'}\mathscr{O}_{\mathscr{E}_1}+J_{\mathscr{E}_1}^{[p]}\big)\otimes \omega_{Z_1|\mathscr{D}_1}^{q-1}\rightarrow \big(T^{m}\mathscr{O}_{\mathscr{E}_1}+J_{\mathscr{E}_1}^{[p]}\big)\otimes \omega_{Z_1|\mathscr{D}_1}^{q}
\end{equation}
\begin{equation*}
\rightarrow \big(T^{m}\mathscr{O}_{\mathscr{E}_1}+J_{\mathscr{E}_1}^{[p]}\big)\otimes \omega_{Z_1|\mathscr{D}_1}^{q+1} \rightarrow 0 \rightarrow \cdots
\end{equation*}
 \begin{equation}\Big({\rm resp.}
\cdots \rightarrow0 \rightarrow \big(T^{m}\mathscr{O}_{\mathscr{E}_1}+J_{\mathscr{E}_1}^{[p]}\big)\otimes \omega_{Z_1|\mathscr{D}_1}^{q-2}\rightarrow \big(T^{m}\mathscr{O}_{\mathscr{E}_1}+J_{\mathscr{E}_1}^{[p]}\big)\otimes \omega_{Z_1|\mathscr{D}_1}^{q-1}
\end{equation}
\begin{equation*}
\rightarrow \big(T^{m}\mathscr{O}_{\mathscr{E}_1}+J_{\mathscr{E}_1}^{[p]}\big)\otimes \omega_{Z_1|\mathscr{D}_1}^{q} \rightarrow 0 \rightarrow \cdots\Big),
\end{equation*}
where $m'$ denotes the smallest integer which is $\geq \max(e+m/p, m)$. The morphism $1: A^{\cdot}\rightarrow B^{\cdot}$ is compatible with the filtrations ${\Tilde{U}}^{\cdot}$.
By the assumpution $p \geq 3$, we have $\varphi_1(J_{\mathscr{E}_1}^{[p]})=p\cdot \varphi_2(J_{\mathscr{E}_1}^{[p]}) =0$. Then the morphism $\varphi_q:A^{\cdot}\rightarrow B^{\cdot}$ is also compatible with the filtrations ${\Tilde{U}}^{\cdot}$.

We define the filtration ${\Tilde{U}}^{m}$($0 \leq m \leq pe $) on ${S_D}^{\cdot}$ to be the mapping fiber of  $1-\varphi_q: {\Tilde{U}}^{m}A^{\cdot}\rightarrow {\Tilde{U}}^{m}B^{\cdot}$
and define the filtration ${\Tilde{U}}^{m}$ on $\mathcal{H}^q({S_D}^{\cdot})$ to be the image of $\mathcal{H}^q({\Tilde{U}}^{m}{S_D}^{\cdot})$.  We will show that ${\Tilde{U}}^{m}\mathcal{H}^q({S_D}^{\cdot})=\mathcal{H}^q({\Tilde{U}}^{m}{S_D}^{\cdot})$ ($0 \leq m \leq pe$) (see Corollary \ref{Cor1}).

Next we calculate the image of $(1+I_{D_2})^{\times} \otimes (M_{X_2}^{gp})^{\otimes (q-1)}$ under the symbol map \ref{symb}. 
\begin{Lem}\label{symb map 2}
For $x \in \big(1+\mathscr{O}_{Z_2}(-\mathscr{D}_2)\big)^{\times}$, $a_1,\dotsc,a_{q-1} \in M_{\mathscr{E}_{2}}^{gp}$, the image of $\overline{x}\otimes \overline{a_1}\otimes \cdots \otimes\overline{a_{q-1}}$ in $\mathcal{H}^q({S_D}^{\cdot})$ under the symbol map {\rm \ref{symb}},  is the class of the cocycle
{\footnotesize \begin{equation}
\Big(d\log x \land d\log a_1 \land \cdots\land d\log a_{q-1},\; p^{-1}\log(x^p\varphi_{\mathscr{E}_2}(x)^{-1})\cdot d\varphi/p (d\log a_1) \land\cdots\land d\varphi/p (d\log a_{q-1})\end{equation}
\begin{equation*}+\Sigma _{i=1}^{q-1}(-1)^{i-1}p^{-1}\log(a_i^p\varphi_{\mathscr{E}_{2}}(a_i)^{-1})\cdot d\log x \land d\log a_1 \land\cdots\end{equation*}
\begin{equation*}
 \land d\log a_{i-1}\land \frac{d\varphi}{p}(d\log a_{i+1})\land \cdots
\land \frac{d\varphi}{p}(d\log a_{q-1})   \Big)\in \big(\mathscr{O}_{\mathscr{E}_1}\otimes \omega_{Z_1|\mathscr{D}_1}^{q} \big)\oplus\big(\mathscr{O}_{\mathscr{E}_1}\otimes \omega_{Z_1|\mathscr{D}_1}^{q-1} \big),
\end{equation*}}
where $\overline{x}$ denote the image of $x$ in $(1+I_{D_2})^{\times}$ and $\overline{a_i}$ denote the images of $a_i$ in $M_{X_2}^{gp}$.
\end{Lem}
\begin{proof}
This is a straightforward calculation by  (\ref{symb}). We only show that the case $q=2$ for simplicity. 
By the construction of the symbol map, we consider the image of the class of cocycle under product structure
{\footnotesize\[\left(d\log x, \frac{1}{p}\log(x^p\varphi_{\mathscr{E}_2}(x)^{-1})\right) \otimes \left(d\log a_1, \frac{1}{p}\log(a_1^p\varphi_{\mathscr{E}_2}(a_1)^{-1})\right).\]} Its image in $ \mathcal{H}^2(S^{\cdot}_D)$ under product structure (see Lemma \ref{prod s_n} ) is the class of cocycle
\[\left(d\log x \land d\log a_1,  -p^{-1}\log(a_1^p\varphi_{\mathscr{E}_{2}}(a_1)^{-1})\cdot d\log x \land d\log a_1 +  \frac{1}{p}\log(x^p\varphi_{\mathscr{E}_2}(x)^{-1})\cdot \frac{d\varphi}{p}(d\log a_1)\right). \]
This completes the proof of the case $q=2$. \end{proof}

\begin{Lem}\label{Lem3} 
For $0 \leq m \leq pe$, we have ${U}^{m}\mathcal{H}^q({S_D}^{\cdot}) \subset {\Tilde{U}}^{m}\mathcal{H}^q({S_D}^{\cdot})$.
\end{Lem}
\begin{proof} We use a similar argument as in \cite[Lemma 2.5.2]{Tsu1}.
By Lemma \ref{symb map 2}, by $Z^q({\Tilde{U}}^{m}{S_D}^{\cdot})=Z^q({S_D}^{\cdot}) \cap {\Tilde{U}}^{m}{S_D}^{\cdot}$, and by the definition of $ {\Tilde{U}}^{m}{S_D}^{\cdot}$, it suffices to show that the following two assertions:
\begin{equation}\label{eq1}
d\log(1+T^m x) \in \big(T^{m}\mathscr{O}_{\mathscr{E}_1}+J_{\mathscr{E}_1}^{[p]}\big)\otimes \omega_{Z_1|\mathscr{D}_1}^1,
\end{equation}
\begin{equation}\label{eq2}
p^{-1}\log\big((1+T^m x)^p\varphi_{\mathscr{E}_2}(1+T^m x)^{-1}\big) \in \big(T^{m}\mathscr{O}_{\mathscr{E}_1}+J_{\mathscr{E}_1}^{[p]}\big)\otimes_{\mathscr{O}_{{Z}_1}} \mathscr{O}_{Z_1}(-\mathscr{D}_1),
\end{equation}
for $x \in \mathscr{O}_{\mathscr{E}_2}\otimes_{\mathscr{O}_{Z_2}} \mathscr{O}_{Z_2}(-\mathscr{D}_2)$ and $1 \leq m \leq pe$. 
Here we denote by $T \in \Gamma(\mathscr{E}_2, \mathscr{O}_{\mathscr{E}_2})$ a lifting of $\pi \in \Gamma(X_2, \mathscr{O}_{X_2})$.
We have 
\begin{align*}
d\log(1+T^mx)&=(1+T^mx)^{-1}d(1+T^mx)\\
&=(1+T^mx)^{-1}T^mx\cdot d\log x + (1+T^mx)^{-1}mT^mx\cdot d\log T.
\end{align*}
Then we obtain \eqref{eq1}. We show that \eqref{eq2}. There exists $y \in  \mathscr{O}_{\mathscr{E}_2}\otimes_{\mathscr{O}_{Z_2}} \mathscr{O}_{Z_2}(-\mathscr{D}_2)$ such that $\varphi_{\mathscr{E}_2}(x)=x^p+py$. We put $z:=y(1+T^{pm}x^p)^{-1}$. Then we have 
\begin{align*}\varphi_{\mathscr{E}_2}(1+T^mx)&=1+T^{pm}\cdot(x^p+py)\\
&=1+T^{pm}x^p + p\cdot T^{pm}\cdot z(1+T^{pm}x^p)\\
&=(1+T^{pm}x^p)(1+pT^{pm}z). 
\end{align*}

We can write $(1+T^mx)^p=1+T^{pm}x^p+pT^m w$ for some $w \in \mathscr{O}_{\mathscr{E}_2}\otimes_{\mathscr{O}_{Z_2}} \mathscr{O}_{Z_2}(-\mathscr{D}_2)$. Thus we obtain 
\begin{align*}
(1+T^m x)^p&\varphi_{\mathscr{E}_2}(1+T^m x)^{-1} =\frac{1+T^{pm}x^p+pT^m}{(1+T^{pm}x^p)(1+pT^{pm}z) }\\
&= \frac{\big(1+pT^m w(1+T^{pm}x^p)^{-1} \big)}{(1+pT^{pm}z)}.
\end{align*}
This completes the proof.
\end{proof}
\mbox{}

Next we calculate $\mathcal{H}^q(\gr_{\Tilde{U}}^m{S_D}^{\cdot})$ for $0 \leq m < pe$. By definition, we have a long exact sequence:
\begin{equation}\label{longexact}
0\rightarrow Z^{q-2}(\gr_{\Tilde{U}}^m{B}^{\cdot})\rightarrow \mathcal{H}^{q-1}(\gr_{\Tilde{U}}^m{S_D}^{\cdot})\rightarrow Z^{q-1}(\gr_{\Tilde{U}}^m{A}^{\cdot})\xrightarrow{1-\varphi_q}\mathcal{H}^{q-1}(\gr_{\Tilde{U}}^m{B}^{\cdot})
\end{equation}
\begin{equation*}
\rightarrow \mathcal{H}^{q}(\gr_{\Tilde{U}}^m{S_D}^{\cdot}) \rightarrow \mathcal{H}^{q}(\gr_{\Tilde{U}}^m{A}^{\cdot})\xrightarrow{1-\varphi_q} \frac{\gr_{\Tilde{U}}^m{B}^q}{B^q(\gr_{\Tilde{U}}^m{B}^{\cdot})}
\end{equation*}
Since $m \geq e+m/p$\;(resp. $m \leq e+m/p$)$\Longleftrightarrow$$m \geq pe/(p-1)$\;(resp. $m \leq pe/(p-1)$) and the differential \[d^{q-1}: \gr^m_{\tilde{U}}A^{q-1} \longrightarrow \gr^m_{\tilde{U}}A^{q}\] vanishes when $0 \leq m < pe/(p-1)$, we have the following:
\begin{equation}\label{(424)}
Z^{q-1}(\gr_{\Tilde{U}}^m{A}^{\cdot})=\begin{cases} 0,&(0 \leq m < pe/(p-1), p \not | m), \\
\big(T^{e+m/p}\mathscr{O}_{Z_1}/T^{e+m/p+1}\mathscr{O}_{Z_1}\big)\otimes \omega_{Z_1|\mathscr{D}_1}^{q-1}, &(0 \leq m < pe/(p-1), p | m),\\
Z^{q-1}\Big(\big(T^{m}\mathscr{O}_{Z_1}/T^{m+1}\mathscr{O}_{Z_1}\big)\otimes \omega_{Z_1|\mathscr{D}_1}^{\cdot}\Big)&(pe/(p-1) \leq m < pe), \end{cases}
\end{equation}
\begin{equation}\label{(425)}
\mathcal{H}^{q-1}(\gr_{\Tilde{U}}^m{B}^{\cdot})=\mathcal{H}^{q-1}\big( T^{m}\mathscr{O}_{Z_1}/T^{m+1}\mathscr{O}_{Z_1} \otimes \omega_{Z_1|\mathscr{D}_1}^{\cdot}\big),
\end{equation}
\begin{equation}\label{H^qA}
\mathcal{H}^{q}(\gr_{\Tilde{U}}^m {A}^{\cdot})=\begin{cases} 
Z^{q}\Big(\big(T^{m}\mathscr{O}_{Z_1} / T^{m+1}\mathscr{O}_{Z_1}\big)\otimes \omega_{Z_1|\mathscr{D}_1}^{\cdot}\Big)&(0 \leq m < pe/(p-1)),\\

\mathcal{H}^{q}\Big(\big(T^{m}\mathscr{O}_{Z_1}/T^{m+1}\mathscr{O}_{Z_1}\big)\otimes \omega_{Z_1|\mathscr{D}_1}^{\cdot}\Big)&(pe/(p-1) \leq m < pe), \end{cases}
\end{equation}
\begin{equation}\label{H^qB}
\frac{\gr_{\Tilde{U}}^m{B}^q}{B^q(\gr_{\Tilde{U}}^m{B}^{\cdot})}=\frac{\big(T^{m}\mathscr{O}_{Z_1}/T^{m+1}\mathscr{O}_{Z_1}\big)\otimes \omega_{Z_1|\mathscr{D}_1}^{q}}{B^q\Big(\big(T^{m}\mathscr{O}_{Z_1}/T^{m+1}\mathscr{O}_{Z_1}\big)\otimes \omega_{Z_1|\mathscr{D}_1}^{\cdot}\Big)}.
\end{equation}

\begin{Lem}\label{Lem5}  Let $m$ be an integer such that $0 \leq m < pe$. Then :
 \begin{enumerate}
   \renewcommand{\labelenumii}{\roman{enumii}).}
\item\;If $m=0$, we have a short exact sequence
\begin{multline}
0 \longrightarrow \frac{\mathscr{O}_Y\otimes_{\mathscr{O}_{Z_1}}\omega_{Z_1|\mathscr{D'}_1}^{q-1}}{\mathscr{O}_Y\otimes_{\mathscr{O}_{Z_1}}\omega_{Z_1|\mathscr{D}_1}^{q-1}}\longrightarrow  \mathcal{H}^{q}(\gr_{\Tilde{U}}^0{S_D}^{\cdot})\\ \longrightarrow
 \Ker\Big(Z^q\big( \mathscr{O}_Y \otimes_{\mathscr{O}_{Z_1}} \omega_{Z_1|\mathscr{D}_1}^{\cdot}\big)\xrightarrow{1-\varphi\otimes \land^{q}\frac{d\varphi}{p}} \mathcal{H}^q\big( \mathscr{O}_Y \otimes_{\mathscr{O}_{Z_1}} \omega_{Z_1|\mathscr{D}_1}^{\cdot} \big)\Big)\longrightarrow 0.
\end{multline}
\item\;We have an isomorphism
 \begin{equation} 
 Z^{q-2}\big(\gr_{\Tilde{U}}^mB^{\cdot}\big)\xrightarrow{\cong} \mathcal{H}^{q-1}(\gr_{\Tilde{U}}^m{S_D}^{\cdot}).
 \end{equation} 
\item\;If $0<m<pe/(p-1)$ and $p|m$, we have a map
 \begin{equation} \mathcal{H}^{q}(\gr_{\Tilde{U}}^m{S_D}^{\cdot}) \twoheadrightarrow B^q\Big(\big(T^{m}\mathscr{O}_{Z_1} / T^{m+1}\mathscr{O}_{Z_1}\big)\otimes \omega_{Z_1|\mathscr{D}_1}^{\cdot} \Big)\quad\cdots (+), \end{equation}
 its kernel is $ \frac{\mathscr{O}_Y\otimes_{\mathscr{O}_{Z_1}}\omega_{Z_1|\mathscr{D'}_1}^{q-1}}{\mathscr{O}_Y\otimes_{\mathscr{O}_{Z_1}}\omega_{Z_1|\mathscr{D}_1}^{q-1}}$.\\ If $0<m<pe/(p-1)$ and $p \not |m$, the map $(+)$ is an isomorphism.
 \item\;If $pe/(p-1) \leq m<pe$,
 \begin{equation} \mathcal{H}^{q}(\gr_{\Tilde{U}}^m{S_D}^{\cdot})=0. \end{equation}
 \end{enumerate}
\end{Lem}
\begin{proof}
We describe the homomorphism $Z^{q-1}(\gr_{\Tilde{U}}^m{A}^{\cdot})\xrightarrow{1-\varphi_q}\mathcal{H}^{q-1}(\gr_{\Tilde{U}}^m{B}^{\cdot})$ as follows:

\noindent
{\rm (i)}{\footnotesize\begin{equation*}
 0 \rightarrow \mathcal{H}^{q-1}\Big(\big(T^{m}\mathscr{O}_{Z_1}/T^{m+1}\mathscr{O}_{Z_1}\big)\otimes \omega_{Z_1|\mathscr{D}_1}^{\cdot}\Big),\;(0 \leq m < pe/(p-1), p \not | m),
\end{equation*}}

\noindent
{\rm (ii)}{\footnotesize \begin{equation*}
\big(T^{e+m/p}\mathscr{O}_{Z_1}/T^{e+m/p+1}\mathscr{O}_{Z_1}\big)\otimes \omega_{Z_1|\mathscr{D}_1}^{q-1}
\xrightarrow{\varphi_1 \otimes \land^{q-1}d\varphi/p} \mathcal{H}^{q-1}\Big(\big(T^{m}\mathscr{O}_{Z_1}/T^{m+1}\mathscr{O}_{Z_1}\big)\otimes \omega_{Z_1|\mathscr{D}_1}^{\cdot}\Big),
 \end{equation*}}
\begin{equation*}
(0 \leq m < pe/(p-1), p  | m),
\end{equation*}

\noindent
{\rm (iii)}
{\footnotesize\begin{equation*}
 Z^{q-1}\Big(\big(T^{m}\mathscr{O}_{Z_1} / T^{m+1}\mathscr{O}_{Z_1}\big)\otimes \omega_{Z_1|\mathscr{D}_1}^{\cdot}\Big)
\xrightarrow{1-\varphi_1 \otimes \land^{q-1}d\varphi/p} \mathcal{H}^{q-1}\Big(\big(T^{m}\mathscr{O}_{Z_1}/T^{m+1}\mathscr{O}_{Z_1}\big)\otimes \omega_{Z_1|\mathscr{D}_1}^{\cdot}\Big),
\end{equation*}}
\begin{equation*}
(m=pe/(p-1)),
\end{equation*}

\noindent
{\rm (iv)}
{\footnotesize \begin{equation*}
Z^{q-1}\Big(\big(T^{m}\mathscr{O}_{Z_1} / T^{m+1}\mathscr{O}_{Z_1}\big)\otimes \omega_{Z_1|\mathscr{D}_1}^{\cdot}\Big)
\xrightarrow{1} \mathcal{H}^{q-1}\Big(\big(T^{m}\mathscr{O}_{Z_1}/T^{m+1}\mathscr{O}_{Z_1}\big)\otimes \omega_{Z_1|\mathscr{D}_1}^{\cdot}\Big),
\end{equation*}}
\begin{equation*}
(pe/(p-1)<m<pe).
\end{equation*}

The first homomorphism (i) is an isomorphism by Lemma \ref{Lem1} (2). We consider the following commutative diagram
\[
\xymatrix@M=10pt{
\big(T^{e+m/p}\mathscr{O}_{Z_1}/T^{e+m/p+1}\mathscr{O}_{Z_1}\big)\otimes \omega_{Z_1|\mathscr{D}_1}^{q-1} \ar[r]^-{\varphi_1 \otimes \land^{q-1}\frac{d\varphi}{p}}&\mathcal{H}^{q-1}\Big(\big(T^{m}\mathscr{O}_{Z_1}/T^{m+1}\mathscr{O}_{Z_1}\big)\otimes \omega_{Z_1|\mathscr{D}_1}^{\cdot}\Big)\\
\mathscr{O}_Y\otimes_{Z_1}\omega^{q-1}_{Z_1|\mathscr{D}_1} \ar[u]^-{\cong}_{ T^{e+m/p}\quad\quad\quad\quad\quad\quad(\blacktriangle)}\ar@{^{(}->}[r]&\mathscr{O}_Y\otimes_{Z_1}\omega^{q-1}_{Z_1|\mathscr{D}'_1}\ar[u]^-{\cong}_{T^m\cdot \left(\varphi \otimes \land^{q-1} \frac{d \varphi}{p}\right)}
 }\]
This commutativity is from Lemma 7.4.2 (4), p. 120 in \cite{Tsu0}. 
Here we use \eqref{(2)} for the left vertical isomorphism and Lemma \ref{Lem1} (3) for the right vertcal isomorphism. Hence the second homomorphism is injective.

 The third homomorphism (iii) is surjective by Lemma \ref{Lem2} (3). In the case (iv), we have $\varphi_1(T^m)=T^{pm}/p$. Then this map is $1$ because $pm> m+1$. It is trivial that (iv) is surjective. If $m=0$, by the long exact sequence \eqref{longexact}, we have the short exact sequence
\begin{multline*}0\longrightarrow \Coker\left(Z^{q-1}(\gr_{\Tilde{U}}^0{A}^{\cdot})\xrightarrow{1-\varphi_q}\mathcal{H}^{q-1}(\gr_{\Tilde{U}}^0{B}^{\cdot}) \right) \longrightarrow \mathcal{H}^q(\gr^0_{\tilde{U}}S_{D}^{\cdot})\\ \longrightarrow  \Ker\left( \mathcal{H}^{q}(\gr_{\Tilde{U}}^0{A}^{\cdot})\xrightarrow{1-\varphi_q} \frac{\gr_{\Tilde{U}}^0{B}^q}{B^q(\gr_{\Tilde{U}}^0{B}^{\cdot})}\right)
 \longrightarrow 0.\end{multline*}
 By the above commutative diagram $(\blacktriangle)$, we take $m=0$, we have 
 \begin{align*}\Coker&\left(Z^{q-1}(\gr_{\Tilde{U}}^0{A}^{\cdot})\xrightarrow{1-\varphi_q}\mathcal{H}^{q-1}(\gr_{\Tilde{U}}^0{B}^{\cdot}) \right)\\&\cong  \Coker\left(\big(T^{e}\mathscr{O}_{Z_1}/T^{e+1}\mathscr{O}_{Z_1}\big)\otimes \omega_{Z_1|\mathscr{D}_1}^{q-1} \xrightarrow{\varphi_1 \otimes \land^{q-1}\frac{d\varphi}{p}} \mathcal{H}^{q-1}\Big(\big(\mathscr{O}_{Z_1}/T\mathscr{O}_{Z_1}\big)\otimes \omega_{Z_1|\mathscr{D}_1}^{\cdot}\Big)  \right)\\
 &\cong \frac{\mathscr{O}_Y\otimes_{\mathscr{O}_{Z_1}}\omega_{Z_1|\mathscr{D'}_1}^{q-1}}{\mathscr{O}_Y\otimes_{\mathscr{O}_{Z_1}}\omega_{Z_1|\mathscr{D}_1}^{q-1}}.
 \end{align*}
 By the isomorphisms (\ref{H^qB}), (\ref{H^qA}) and (\ref{(2)}), we have the following isomorphism:
 \begin{align*}\Ker&\left( \mathcal{H}^{q}(\gr_{\Tilde{U}}^0{A}^{\cdot})\xrightarrow{1-\varphi_q} \frac{\gr_{\Tilde{U}}^0{B}^q}{B^q(\gr_{\Tilde{U}}^0{B}^{\cdot})}\right)\\
 & \cong \Ker\left(Z^q((\mathscr{O}_{Z_1}/T\mathscr{O}_{Z_1})\otimes \omega^{\cdot}_{Z_1|\mathscr{D}_1})\xrightarrow{1-\varphi\otimes \land^{q}\frac{d\varphi}{p}}\mathcal{H}^q((\mathscr{O}_{Z_1}/T\mathscr{O}_{Z_1})\otimes \omega^{\cdot}_{Z_1|\mathscr{D}_1})\right)\\
 &\cong  \Ker\Big(Z^q\big( \mathscr{O}_Y \otimes_{\mathscr{O}_{Z_1}} \omega_{Z_1|\mathscr{D}_1}^{\cdot}\big)\xrightarrow{1-\varphi\otimes \land^{q}\frac{d\varphi}{p}} \mathcal{H}^q\big( \mathscr{O}_Y \otimes_{\mathscr{O}_{Z_1}} \omega_{Z_1|\mathscr{D}_1}^{\cdot} \big)\Big). \end{align*}
 Then we obtain the short exact sequence in $(1)$.
 By the long exact sequence \eqref{longexact}, we have 
 \[0 \longrightarrow Z^{q-2}(\gr_{\Tilde{U}}^m{B}^{\cdot})\rightarrow \mathcal{H}^{q-1}(\gr_{\Tilde{U}}^m{S_D}^{\cdot}) \longrightarrow \Ker\left(Z^{q-1}(\gr_{\Tilde{U}}^m{A}^{\cdot})\xrightarrow{1-\varphi_q}\mathcal{H}^{q-1}(\gr_{\Tilde{U}}^m{B}^{\cdot})\right)\longrightarrow 0.\]
Since all homomorphisms (i)--(iv) are injective, we have $\Ker\left(Z^{q-1}(\gr_{\Tilde{U}}^m{A}^{\cdot})\xrightarrow{1-\varphi_q}\mathcal{H}^{q-1}(\gr_{\Tilde{U}}^m{B}^{\cdot})\right)=0.$
 Thus we obtain $(2)$.  We have the surjective homomorphism
\begin{equation}\label{surjectiveK}
 \mathcal{H}^{q}(\gr_{\Tilde{U}}^m{S_D}^{\cdot}) \twoheadrightarrow \Ker\left( \mathcal{H}^{q}(\gr_{\Tilde{U}}^m{A}^{\cdot})\xrightarrow{1-\varphi_q} \frac{\gr_{\Tilde{U}}^m{B}^q}{B^q(\gr_{\Tilde{U}}^m{B}^{\cdot})}\right)(:=\mathscr{K})
\end{equation}
 by the long exact sequence \eqref{longexact}, \eqref{(424)} and \eqref{(425)}. If $0 < m <pe/(p-1)$, we have \[\mathscr{K} \cong  B^q\left(\big(T^{m}\mathscr{O}_{Z_1} / T^{m+1}\mathscr{O}_{Z_1}\big)\otimes \omega_{Z_1|\mathscr{D}_1}^{\cdot} \right).\] If $0 < m <pe/(p-1)$ and $p | m$, this homomorphism is surjective as it is. 
   If $0 < m <pe/(p-1)$ and $p \nmid m$, this homomorphism is an isomorphism by the surjectivity of (i), (iii) and (iv). Therefore we obtain the claim $(3)$.
   Finally, if $pe/(p-1) < m < pe$, the above homomorphism \eqref{surjectiveK} is an isomorphism and $\mathscr{K}=0$ by \eqref{H^qA} and \eqref{H^qB}. Thus we obtain the claim $(4)$. This completes the proof.
\end{proof}

\begin{Lem}\label{Lem6} 
We have $ \mathcal{H}^{q}(\Tilde{U}^m{S_D}^{\cdot})=0$ for $pe/(p-1) < m \leq pe$
\end{Lem}
\begin{proof}
If $pe/(p-1) < m \leq pe$, we have 
\begin{equation}
\Tilde{U}^mA^{q-1}=\Tilde{U}^mB^{q-1}=\big(T^{m}\mathscr{O}_{\mathscr{E}_1}+J_{\mathscr{E}_1}^{[p]}\big)\otimes \omega_{Z_1|\mathscr{D}_1}^{q-1},
\end{equation}
\begin{equation}\Big({\rm resp}.\; \Tilde{U}^mA^{q}=\Tilde{U}^mB^{q}=\big(T^{m}\mathscr{O}_{\mathscr{E}_1}+J_{\mathscr{E}_1}^{[p]}\big)\otimes \omega_{Z_1|\mathscr{D}_1}^q\Big),
\end{equation}
since $m > e+m/p$ ($\Leftrightarrow pe/(p-1) < m$ ). We have
\[\varphi_1\left(T^{m}\mathscr{O}_{\mathscr{E}_1}+J_{\mathscr{E}_1}^{[p]}\right) \subset T^{m+1}\mathscr{O}_{\mathscr{E}_1}+J_{\mathscr{E}_1}^{[p]},\ \ \ \ \varphi_1\left(T^{pe}\mathscr{O}_{\mathscr{E}_1}+J_{\mathscr{E}_1}^{[p]}\right)=0,\]
for $m > pe/(p-1)$. Here we use $\varphi_1(J_{\mathscr{E}_1}^{[p]})=0$ ($\because p \geq 3$). 

Hence $\varphi_1 \otimes \land^{q-1}d\varphi/p$ (\resp.\;$\varphi \otimes \land^{q}d\varphi/p$) is nilpotent on $\big(T^{m}\mathscr{O}_{\mathscr{E}_1}+J_{\mathscr{E}_1}^{[p]}\big)\otimes \omega_{Z_1|\mathscr{D}_1}^{q-1}$ (\resp.\ $\big(T^{m}\mathscr{O}_{\mathscr{E}_1}+J_{\mathscr{E}_1}^{[p]}\big)\otimes \omega_{Z_1|\mathscr{D}_1}^q\Big)$).
Then $1-\varphi_q: {\Tilde{U}}^{m}A^{\cdot}\rightarrow {\Tilde{U}}^{m}B^{\cdot}$ are bijective in degree $q-1$ and degree $q$. 
\end{proof}

\begin{Cor}\label{Cor vanish}
We have  $ \Tilde{U}^m\mathcal{H}^{q}({S_D}^{\cdot})=0$ for $pe/(p-1) < m \leq pe$.
\end{Cor}
\begin{Lem}\label{Lem7} 
The homomorphism $ \mathcal{H}^{q}(\Tilde{U}^{m+1}{S_D}^{\cdot})\rightarrow \mathcal{H}^{q}( \Tilde{U}^m{S_D}^{\cdot})$ is injective for $0 \leq m <pe$.
\end{Lem}
\begin{proof}
By Lemma \ref{Lem6}, we may assume that $m+1 \leq pe/(p-1)$. It is enough to show that 
\begin{equation}
 \mathcal{H}^{q-1}(\Tilde{U}^{m}{S_D}^{\cdot})\rightarrow \mathcal{H}^{q-1}(\gr_{\Tilde{U}}^m{S_D}^{\cdot})
\end{equation}
is surjective. From the argument before Lemma \ref{Lem5} (i), we obtain an isomorphism
\begin{equation}
Z^{q-2}\big(\gr_{\Tilde{U}}^mB^{\cdot}\big)=Z^{q-2}\Big(\big(T^{m}\mathscr{O}_{Z_1}/T^{m+1}\mathscr{O}_{Z_1}\big)\otimes \omega_{Z_1|\mathscr{D}_1}^{\cdot}\Big)\xrightarrow{\cong} \mathcal{H}^{q-1}(\gr_{\Tilde{U}}^m{S_D}^{\cdot}).
\end{equation}
Then it suffices to prove that the natural homomorphism 
\begin{equation}
Z^{q-2}\big(\Tilde{U}^mB^{\cdot}\big)=Z^{q-2}\Big(\big(T^{m}\mathscr{O}_{Z_1}+J_{\mathscr{E}_1}^{[p]}\big)\otimes \omega_{Z_1|\mathscr{D}_1}^{\cdot}\Big)
\end{equation}
\begin{equation*}
\rightarrow Z^{q-2}\big(\gr_{\Tilde{U}}^mB^{\cdot}\big)=Z^{q-2}\Big(\big(T^{m}\mathscr{O}_{Z_1}/T^{m+1}\mathscr{O}_{Z_1}\big)\otimes \omega_{Z_1|\mathscr{D}_1}^{\cdot}\Big)
\end{equation*}
is surjective or equivalently that the homomorphism 
\begin{equation}
\mathcal{H}^{q-2}\Big( \big(T^{m}\mathscr{O}_{Z_1}+J_{\mathscr{E}_1}^{[p]}\big)\otimes \omega_{Z_1|\mathscr{D}_1}^{\cdot} \Big) \rightarrow \mathcal{H}^{q-2}\Big( \big(T^{m}\mathscr{O}_{Z_1}/T^{m+1}\mathscr{O}_{Z_1}\big)\otimes \omega_{Z_1|\mathscr{D}_1}^{\cdot} \Big) 
\end{equation}
is surjective.  When $p \nmid m$, this is obvious by Lemma \ref{Lem1} (2). In the case of $p|m$, this follows from the following commutative diagram in which the lower horizontal arrow is surjective and the right vertical arrow is an isomorphism by Lemma \ref{Lem1} (3).

$$
\begin{CD}   
\mathcal{H}^{q-2}\Big( \big(T^{m}\mathscr{O}_{Z_1}+J_{\mathscr{E}_1}^{[p]}\big)\otimes \omega_{Z_1|\mathscr{D}_1}^{\cdot} \Big)@>>>\mathcal{H}^{q-2}\Big( \big(T^{m}\mathscr{O}_{Z_1}/T^{m+1}\mathscr{O}_{Z_1}\big)\otimes \omega_{Z_1|\mathscr{D}_1}^{\cdot} \Big)  \\
@AAT^m\cdot \varphi \otimes \land^{q-2}\frac{d\varphi}{p}A @AAT^m\cdot \varphi \otimes \land^{q-2}\frac{d\varphi}{p}A \\
\omega_{Z_1|\mathscr{D'}_1}^{q-2}@>>>\mathscr{O}_Y \otimes \omega_{Z_1|\mathscr{D'}_1}^{q-2}.
\end{CD}
$$
\end{proof}

\begin{Cor}\label{Cor1}
$ \mathcal{H}^{q}(\Tilde{U}^m{S_D}^{\cdot}) = \Tilde{U}^m\mathcal{H}^{q}({S_D}^{\cdot})$ for $0 \leq m \leq pe$.
\end{Cor}
From lemma \ref{Lem3} and \ref{Cor1}, we have homomorphisms 
\begin{equation}
\alpha_{m,D}: \gr_{U}^m\mathcal{H}^{q}({S_D}^{\cdot}) \rightarrow \gr_{\Tilde{U}}^m\mathcal{H}^{q}({S_D}^{\cdot})
\end{equation}
and injective homomorphisms
\begin{equation}
\beta_{m,D}: \gr_{\Tilde{U}}^m\mathcal{H}^{q}({S_D}^{\cdot}) \rightarrow \mathcal{H}^{q}(\gr_{\Tilde{U}}^m{S_D}^{\cdot})
\end{equation}
for $0 \leq m <pe$.
\begin{Lem}\label{Lem8}
Let $m$ be a non-negative integer. Let $x \in (1+I_{\mathscr{D}_2})^{\times}$, let $a_1,\dotsc  a_{q-1} \in M_{Z_2}^{gp}$ and let $y \in \mathscr{O}_{Z_2}(-\mathscr{D}_2)$.
Let $\overline{x}$ denote the image of $x$ in $ (1+I_{D_2})^{\times}$, let $\overline{a_i}$ denote the image of $a_i$ in $M_{X_2}^{gp}$ and let $\overline{y}$ denote the image of $y$ in $\mathscr{O}_{X_2}(-{D_2})$. Then we have:
\begin{enumerate}
\item\;If $m=0$, the image of \[\overline{x}\otimes \overline{a_1}\otimes\cdots \otimes \overline{a_{q-1}} \in (1+I_{D_2})^{\times}\otimes (M_{X_2}^{gp})^{\otimes (q-1)}\] under the composite 
 \begin{align*}
(\P1)\quad &U^0\left((1+I_{D_2})^{\times}\otimes (M_{X_2}^{gp})^{\otimes (q-1)}\right)=(1+I_{D_2})^{\times}\otimes (M_{X_2}^{gp})^{\otimes (q-1)} \rightarrow \gr_{U}^0\mathcal{H}^{q}({S_D}^{\cdot})\\ &\xrightarrow{\alpha_{0,D}}\gr_{\Tilde{U}}^0\mathcal{H}^{q}({S_D}^{\cdot}) 
\xhookrightarrow{\beta_{0,D}} \mathcal{H}^{q}(\gr_{\Tilde{U}}^0{S_D}^{\cdot})\overset{(a)}{\twoheadrightarrow} \Ker\left(Z^q\big( \mathscr{O}_Y \otimes_{\mathscr{O}_{Z_1}} \omega_{Z_1|\mathscr{D}_1}^{\cdot}\big)\xrightarrow{1-\varphi\otimes \land^{q}d\varphi/p} \mathcal{H}^q\big( \mathscr{O}_Y \otimes_{\mathscr{O}_{Z_1}} \omega_{Z_1|\mathscr{D}_1}^{\cdot} \big)\right)
\end{align*}
is $1 \otimes d\log x \land d\log a_1\land \cdots \land d\log a_{q-1}$.\\
 The image of \[\overline{x}\otimes \overline{a_1}\otimes\cdots \otimes \overline{a_{q-2}}\otimes \pi \in V^0\left((1+I_{D_2})^{\times}\otimes (M_{X_2}^{gp})^{\otimes (q-1)}\right)\] under the map $(\P1)$ is $1 \otimes d\log x \land d\log a_1\land \cdots \land d\log a_{q-2}$.\\

\item\; The image of \[(1+\pi^m \overline{y}) \otimes \overline{a_1}\otimes\cdots \otimes \overline{a_{q-1}} \in U^m\Big( (1+I_{D_2})^{\times}\otimes (M_{X_2}^{gp})^{\otimes (q-1)}\Big)\] under the composite 
\begin{equation}
 (\P2)\quad U^m\Big( (1+I_{D_2})^{\times}\otimes (M_{X_2}^{gp})^{\otimes (q-1)}\Big) \rightarrow  \gr_{U}^m\mathcal{H}^{q}({S_D}^{\cdot}) \xrightarrow{\beta_{m,D}\alpha_{m,D}} \mathcal{H}^{q}(\gr_{\Tilde{U}}^m{S_D}^{\cdot})  \end{equation}
 \begin{equation*}
 \xrightarrow[Lemma\;\ref{Lem5}\;(3)]{(b)} B^q\Big(\big(T^{m}\mathscr{O}_{Z_1} / T^{m+1}\mathscr{O}_{Z_1}\big)\otimes \omega_{Z_1|\mathscr{D}_1}^{\cdot} \Big)
 \end{equation*}
 is $d\big(T^m y \otimes d\log a_1\land\cdots \land d\log a_{q-1}\big)$. If $0< m < pe/(p-1)$ and $p \nmid m$, the map $(b)$ is isomorphism. If $0< m < pe/(p-1)$ and $p | m$, the map $(b)$ is surjective.\\
  The image of \[(1+\pi^m \overline{y}) \otimes \overline{a_1}\otimes\cdots \otimes \overline{a_{q-2}}\otimes \pi \in V^m\Big( (1+I_{D_2})^{\times}\otimes (M_{X_2}^{gp})^{\otimes (q-1)}\Big)\]  under the map $(\P2)$ is $d\big(T^m y \otimes d\log a_1\land\cdots \land d\log a_{q-2}\land d\log T\big)$. 
 \end{enumerate}
 We put 
 \begin{equation}\label{L^m}
 \mathcal{L}^m:=\ker\left(\gr^m_{U}\mathcal{H}^q(S^{\cdot}_D) \twoheadrightarrow B^q\left((T^m\mathscr{O}_{Z_1} / T^{m+1}\mathscr{O}_{Z_1}) \otimes \omega^{\cdot}_{Z_1|\mathscr{D}_1}\right)\right).
 \end{equation}
\end{Lem}

\begin{proof}
We note that $T \in \Gamma(Z_2, \mathscr{O}_{Z_2})$ is a lifting of $\pi \in \Gamma(X_2, \mathscr{O}_{X_2})$. 
If $m=0$, the image of $\overline{x}\otimes \overline{a_1}\otimes \cdots \otimes \overline{a_{q-1}} \in (1+I_{D_2})^{\times}\otimes (M_{X_2}^{gp})^{\otimes (q-1)}(\mbox{resp.}\;\; \overline{x}\otimes \overline{a_1}\otimes\cdots \otimes \overline{a_{q-2}}\otimes \pi)$  by the symbol map \eqref{symb} is the class of a cocycle of the form
 \[ \Big(d\log x \land d\log a_1 \land \cdots \land d\log a_{q-1},\ \cdots\Big)\quad\left(\mbox{resp.}\;\;\Big(d\log x \land d\log a_1 \land \cdots \land d\log a_{q-2}\land d\log T,\ \cdots\Big)\right) \]
 by  Lemma \ref{symb map 2}.
 Its image in $\Ker\left(Z^q\big( \mathscr{O}_Y \otimes_{\mathscr{O}_{Z_1}} \omega_{Z_1|\mathscr{D}_1}^{\cdot}\big)\xrightarrow{1-\varphi\otimes \land^{q}d\varphi/p} \mathcal{H}^q\big( \mathscr{O}_Y \otimes_{\mathscr{O}_{Z_1}} \omega_{Z_1|\mathscr{D}_1}^{\cdot} \big)\right)
$ is \[1 \otimes d\log x \land d\log a_1 \land \cdots \land d\log a_{q-1}\quad\Big(\mbox{resp.}\;\;1 \otimes d\log x \land d\log a_1 \land \cdots \land d\log a_{q-2} \land d\log T\Big) \] by the construction of the homomorphism $(a)$. Thus we obtain the claim $(1)$.
If $0 < m < pe/(p-1)$, the image of $(1+\pi^m \overline{y}) \otimes \overline{a_1}\otimes\cdots \otimes \overline{a_{q-1}} \in U^m\Big( (1+I_{D_2})^{\times}\otimes (M_{X_2}^{gp})^{\otimes (q-1)}\Big)\;(\mbox{resp.}\;\;(1+\pi^m \overline{y}) \otimes \overline{a_1}\otimes\cdots \otimes \overline{a_{q-2}}\otimes \pi \in V^m)$ by the symbol map \eqref{symb} is the class of a cocycle of the form
 \begin{align*}(*)\  \Big(d\log (1+ T^m y) \land d\log a_1 \land \cdots \land& d\log a_{q-1},\ \cdots\Big)\\& \left(\mbox{resp.}\;\; \Big(d\log (1+ T^m y) \land d\log a_1 \land \cdots \land d\log a_{q-2}\land d\log T,\ \cdots\Big)\right) \end{align*}
We have 
\begin{align*}
d\log &(1+ T^m y) \land d\log a_1 \land \cdots \land d\log a_{q-1}= (1+T^my)^{-1}\cdot d(1+T^my) \land d\log a_1 \land \cdots \land d\log a_{q-1}\\
&=(1-T^my)\cdot d(1+T^my) \land d\log a_1 \land \cdots \land d\log a_{q-1}\ \ \ \left(\ (1+T^my)^{-1}=1-T^my\ \because\ T^{2m}=0\right)\\
&= d(1+T^my) \land d\log a_1 \land \cdots \land d\log a_{q-1}-T^my d(T^my) \land d\log a_1 \land \cdots \land d\log a_{q-1}\\
&=d(T^my) \land d\log a_1 \land \cdots \land d\log a_{q-1} - T^my\{T^mdy + y d(T^m)\} \land d\log a_1 \land \cdots \land d\log a_{q-1}\\
&=d(T^my \otimes d\log a_1 \land \cdots \land d\log a_{q-1})- T^{2m}y\{dy + y d\log(T^m)\} \land d\log a_1 \land \cdots \land d\log a_{q-1}.
\end{align*} 
Then the image of $(*)$ in {\footnotesize $B^q\Big(\big(T^{m}\mathscr{O}_{Z_1} / T^{m+1}\mathscr{O}_{Z_1}\big)\otimes \omega_{Z_1|\mathscr{D}_1}^{\cdot} \Big)$} is 
\[d\big(T^m y \otimes d\log a_1\land\cdots \land d\log a_{q-1}\big)\quad\left(\mbox{resp.}\;\;d\big(T^m y \otimes d\log a_1\land\cdots \land d\log a_{q-2}\land d\log T\big)\right)\]
by the construction of the homomorphism $(b)$. Then we obtain the claim $(2)$.
\end{proof}
\begin{Rem} By the above Lemma $\ref{Lem8}$ (2), the map $\beta_m$ and $\alpha_m$ are isomorphism in the case $0< m < pe/(p-1),\; p \nmid m$. Then we obtain an isomorphism $ \gr_{\Tilde{U}}^m\mathcal{H}^{q}({S_D}^{\cdot}) \xrightarrow{\cong} \mathcal{H}^{q}(\gr_{\Tilde{U}}^m{S_D}^{\cdot})$ in this case.
\end{Rem}

\begin{Lem}\label{CokernelSymb}
The cokernel of the morphism \[\gr^m{\rm Symb}_{X|D}:\gr_{U}^m\big((1+I_{D_2})^{\times}\otimes(M_{X_2}^{gp})^{\otimes(q-1)}\big)\longrightarrow \gr_{\Tilde{U}}^{m}\mathcal{H}^{q}(S_{D}^{\cdot})\]
is Mittag-Leffler zero with respect to the multiplicities of the prime components of $D$.
\end{Lem}
\begin{proof}
We have the following commutative diagram:\\
{\footnotesize  \[
\xymatrix{
&&0 \ar[d]&& \\
&&\mathcal{L}'^m\ar@{->>}[ld]_{(*)} \ar[d]&&\\
0&\Coker(\gr^m{\rm Symb}_{X|D})\ar[l]&\gr_{\Tilde{U}}^m\mathcal{H}^q(S_{D}^{\cdot})\ar[l]\ar[d]&\gr_{U}^m\big((1+I_{D_2})^{\times}\otimes(M_{X_2}^{gp})^{\otimes(q-1)}\big)\ar@{->>}[ld]^{(**)} \ar[l]&\\
&&\mathfrak{D} \ar[d]&&\\
&&0&&
}\]
}
where the vertical and horizontal sequences is exact.  Here we put 
 \[ \mathfrak{D}:=\begin{cases}  \Ker\Big(Z^q\big( \mathscr{O}_Y \otimes_{\mathscr{O}_{Z_1}} \omega_{Z_1|\mathscr{D}_1}^{\cdot}\big)\xrightarrow{1-\varphi\otimes \land^{q}\frac{d\varphi}{p}} \mathcal{H}^q\big( \mathscr{O}_Y \otimes_{\mathscr{O}_{Z_1}} \omega_{Z_1|\mathscr{D}_1}^{\cdot} \big)\Big)\ {\rm if}\ m=0,&\\
 B^q\left((T^m\mathscr{O}_{Z_1}/T^{m+1}\mathscr{O}_{Z_1})\otimes \omega^{\cdot}_{Z_1|\mathscr{D}_1}\right)\ \ \ \ \ \ {\rm if}\ 0 <m<pe/(p-1),&\\
 
 \end{cases}\]

{\footnotesize \[ \mathcal{L}_0^{'m}:=\begin{cases}
\Ker\left(\mathcal{H}^q(\gr_{\tilde{U}}^0S_{D}^{\cdot}) \twoheadrightarrow \Ker\Big(Z^q\big( \mathscr{O}_Y \otimes_{\mathscr{O}_{Z_1}} \omega_{Z_1|\mathscr{D}_1}^{\cdot}\big)\xrightarrow{1-\varphi\otimes \land^{q}\frac{d\varphi}{p}} \mathcal{H}^q\big( \mathscr{O}_Y \otimes_{\mathscr{O}_{Z_1}} \omega_{Z_1|\mathscr{D}_1}^{\cdot} \big)\Big)\right)\ {\rm if}\ m=0,& \\
\Ker\left(\mathcal{H}^q(\gr_{\tilde{U}}^mS_{D}^{\cdot}) \twoheadrightarrow  B^q\left((T^m\mathscr{O}_{Z_1}/T^{m+1}\mathscr{O}_{Z_1})\otimes \omega^{\cdot}_{Z_1|\mathscr{D}_1}\right)\right)\ \ \ \ \ \ {\rm if}\ 0 <m<pe/(p-1),\  p|m,&\\
0\ \ \ \ \ \  {\rm if}\ 0 <m<pe/p-1,\  p \nmid m,&
\end{cases}\]}
\[\mathcal{L}'^m:= \mathcal{L}_0^{'m} \cap \gr_{\Tilde{U}}^m\mathcal{H}^q(S_{D}^{\cdot}).\]

The morphism $(**)$ is constructed in Lemma \ref{Lem8}. They are surjective by the explicit assignments in Lemma \ref{Lem8}. 
We have 
\[\mathcal{L}_0^{'m} \cong \frac{\mathscr{O}_Y\otimes_{\mathscr{O}_{Z_1}}\omega^{q-1}_{Z_1|\mathscr{D}'_1}}{\mathscr{O}_Y\otimes_{\mathscr{O}_{Z_1}}\omega^{q-1}_{Z_1|\mathscr{D}_1}}\ \ ({\rm if}\ 0 \leq m \leq pe/(p-1),\ p|m)\]
by Lemma \ref{Lem5} $(1)$ and $(3)$.

Thus $\mathcal{L}_0^{'m}$ is Mittag-Leffler zero with respect to the multiplicities of the prime components of $D$. Since $(**)$ is surjective, $(*)$ is also surjective. Then $\mathcal{L}'^m$ is Mittag-Leffler zero. Hence {\footnotesize$\Coker(\gr^m\Symb_{X|D})$} is also Mittag-Leffler zero. This completes the proof.
\end{proof}

\begin{Lem}\label{compU} The kernel and the cokernel of the morphisms \[\alpha_{m,D}: \gr_{U}^m\mathcal{H}^{q}(S_{D}^{\cdot}) \rightarrow \gr_{\Tilde{U}}^m\mathcal{H}^{q}(S_{D}^{\cdot})\] and the cokernel of 
\[ U^{m}\mathcal{H}^{q}(S_{D}^{\cdot}) \hookrightarrow  \tilde{U}^{m}\mathcal{H}^{q}(S_{D}^{\cdot})\]
are Mittag-Leffler zero  with respect to the multiplicities of the prime components of $D$.
\end{Lem}
\begin{proof}We consider the following commutative diagram:\\
$$
\begin{CD}   
0 @>>> U^{m+1}\mathcal{H}^{q}(S_{D}^{\cdot})@>>> U^{m}\mathcal{H}^{q}(S_{D}^{\cdot})@>>>\gr_{U}^{m}\mathcal{H}^{q}(S_{D}^{\cdot})@>>>0\\
@.@VVV @VVV@VVV\\
0@>>>\Tilde{U}^{m+1}\mathcal{H}^{q}(S_{D}^{\cdot})@>>> \Tilde{U}^{m}\mathcal{H}^{q}(S_{D}^{\cdot})@>>> \gr_{\Tilde{U}}^{m}\mathcal{H}^{q}(S_{D}^{\cdot})@>>>0
\end{CD}
$$\\
The left and central vertical morphism is injective by Lemma \ref{Lem3}. If $m > pe/(p-1)$, the claim is trivial. We assume that $0 \leq m \leq pe/(p-1)$. If $m=pe/(p-1)$, the right vertical morphism is injective by Corollaly \ref{Cor vanish} and the cokernel of $\alpha_{m,D}$ is Mittag-Leffler zero from Lemma \ref{CokernelSymb}. We can easily show the assertion by induction on $m$.
\end{proof}


\section{\textbf{Calculation of $\mathcal{H}^q(s_1(r)_{X|D})$ for $0 \leq r < q \leq p-2$}}
In this section, for $0 \leq q< r \leq p-2$, we will calculate the cohomology sheaf $\mathcal{H}^q(s_1(r)_{X|D})$ by a similar computations as in \cite[Appendix]{Tsu2}. 
The setting remains as in \S \ref{Loc}. We keep the assumption $p \geq 3$ in the following sections.

We define a descending filtration $\Tilde{\mathfrak{U}}^m, m\in\mathbb{N}$ on $s_1(r)_{X|D}$ for an integer $0 \leq r \leq p-2$ as follows:
we define the filtration $\Tilde{\mathfrak{U}}^m, m\in\mathbb{N}$ on $\mathscr{O}_{\mathscr{E}_1}$ (resp. $J_{\mathscr{E}_1}^{[r]}$ ($r \leq p-2$)) by 
\[T^m\mathscr{O}_{\mathscr{E}_1}+J_{\mathscr{E}_1}^{[p]}\quad( {\rm resp.}\;T^{ \max\{er+\lceil \frac{m}{p}\rceil,m\}}\mathscr{O}_{\mathscr{E}_1}+J_{\mathscr{E}_1}^{[p]} ).\]
Here $\lceil x \rceil$ for $x \in \mathbb{R}$ denotes the smallest integer $\geq x$. We can easy to see that the morphism $1, \varphi_r : J_{\mathscr{E}_1}^{[r-\cdot]}\otimes \omega_{Z_1|\mathscr{D}_1}^{\cdot} \rightarrow \mathscr{O}_{\mathscr{E}_1}\otimes \omega_{Z_1|\mathscr{D}_1}^{\cdot}$ are compatible with $\Tilde{\mathfrak{U}}^m$.
We define the filtration ${\Tilde{\mathfrak{U}}}^{m}$ on $s_1(r)_{X|D}$ to be the mapping fiber of  $1-\varphi_q: {\Tilde{\mathfrak{U}}}^{m}(J_{\mathscr{E}_1}^{[r-\cdot]})\otimes \omega_{Z_1|\mathscr{D}_1}^{\cdot} \rightarrow {\Tilde{\mathfrak{U}}}^{m}(\mathscr{O}_{\mathscr{E}_1})\otimes \omega_{Z_1|\mathscr{D}_1}^{\cdot}$.

{\begin{Lem}\label{generalLem1}(cf. {\rm \cite[Lemma A.8]{Tsu2}}) Let $m$ be a non-negative integer.
For $a\in k^{*}$, the  homomorphism 
\begin{equation}
1-a^pC^{-1}:Z^q\big(\mathscr{O}_Y \otimes \omega_{Z_1|\mathscr{D}_1}^{\cdot}\big)\rightarrow \mathcal{H}^q\big(\mathscr{O}_Y \otimes \omega_{Z_1|\mathscr{D}_1}^{\cdot}  \big) 
\end{equation}
is surjective. Its kernel $\mathcal{K}$ is the subsheaf of abelian groups of $Z^q\big(\mathscr{O}_Y \otimes \omega_{Z_1|\mathscr{D}_1}^{\cdot}\big)$ generated by local sections of the form 
 \begin{equation}
x \otimes d\log(a_1)\land d\log(a_2)\land\cdots\land d\log(a_q), \end{equation}
\begin{equation*}(x \in \Ker(1-a^p\varphi:\mathscr{O}_Y\rightarrow \mathscr{O}_Y),  a_1 \in 1+\mathscr{O}_{Z_1}(-\mathscr{D}_1), a_2\dotsc, a_q \in M_{Z_1}^{gp}\big).
\end{equation*}
and there is a short exact sequence
\begin{equation}\label{a^p-ses}
0 \rightarrow \Ker\big(1-a^pC^{-1}:Z_{Y|D}^{q-1} \longrightarrow \mathcal{H}^{q-1}(\omega_{Y|D}^{\cdot})\big)  \rightarrow \mathcal{K} \rightarrow \Ker\big(1-a^pC^{-1}:Z_{Y|D}^{q} \longrightarrow \mathcal{H}^{q}(\omega_{Y|D}^{\cdot})\big) \rightarrow 0
\end{equation}
which is characterized by the following properties: \\
For $a_1 \in 1+\mathscr{O}_{Z_1}(-\mathscr{D}_1), a_2, \dotsc, a_q \in M_{Z_1}^{gp}$ and $x \in  \Ker(1-a^p\varphi:\mathscr{O}_Y\rightarrow \mathscr{O}_Y)$, the image of 
\begin{equation}
x\otimes d\log(a_1)\land d\log(a_2) \land \cdots \land d\log(a_q)\in \mathcal{K} 
\end{equation}
in the right term is $d\log(\overline{a_1})\land d\log(\overline{a_2}) \land \cdots \land d\log(\overline{a_q})$, and 
\begin{equation}
x\otimes d\log(a_1)\land d\log(a_2) \land\cdots\land d\log(a_{q-1})\land d\log T\in \mathcal{K} 
\end{equation}
is the image of $d\log(\overline{a_1})\land d\log(\overline{a_2}) \land\cdots\land d\log(\overline{a_{q-1}})$ in the left term, where $\overline{a_i}$ denote the images of $a_i$ in $M_Y^{gp}$.
\end{Lem}}
\begin{proof}We can obtain this Lemma in the same way as Lemma \ref{Lem2} $(3)$. 
We can reduce to the case $a=1$ by the commutative diagrams (cf. \cite[Lemma A8]{Tsu2}): 
{\tiny\[
\xymatrix@M=3pt{
Z^q_{Y|D_s}\ar[r]^-{1-a^pC^{-1}}& \mathcal{H}^q(\omega^{\cdot}_{Y|D_s}) &\\
Z^q_{Y|D_s}\ar[r]^-{1-C^{-1}}\ar[u]^-{b^{-p}}_{\cong}& \mathcal{H}^q(\omega^{\cdot}_{Y|D_s})\ar[u]^-{b^{-p}}_{\cong}, &  
}
\xymatrix@M=3pt{
Z^q((\mathscr{O}_{Z_1}/T\mathscr{O}_{Z_1})\otimes \omega^{\cdot}_{Z_1|\mathscr{D}_1})\ar[r]^-{1-a^pC^{-1}}& \mathcal{H}^q((\mathscr{O}_{Z_1}/T\mathscr{O}_{Z_1})\otimes \omega^{\cdot}_{Z_1|\mathscr{D}_1}) &\\
Z^q((\mathscr{O}_{Z_1}/T\mathscr{O}_{Z_1})\otimes \omega^{\cdot}_{Z_1|\mathscr{D}_1})\ar[r]^-{1-C^{-1}}\ar[u]^-{b^{-p}}_{\cong}& \mathcal{H}^q((\mathscr{O}_{Z_1}/T\mathscr{O}_{Z_1})\otimes \omega^{\cdot}_{Z_1|\mathscr{D}_1})\ar[u]^-{b^{-p}}_{\cong},&  
}
\]}
where we put $b:=a^{\frac{1}{p-1}}$ which exists \'etale locally on $\Spec(k)$. The short exact sequence \eqref{a^p-ses} is obtained by the same argument as the proof of the exactness of \eqref{K} in Lemma \ref{Lem2} $(3)$.\end{proof}

\begin{Lem}\label{Lemq<r}(cf. {\rm Lemma \ref{Lem5},\   \cite[Lemma A9]{Tsu2}})
Let $q$ and $r$ be integers such that $0 \leq q \leq r \leq p-2$. 
The map \[\mathcal{H}^q(\gr^m_{\tilde{\mathfrak{U}}}(1-\varphi_r)):\;\mathcal{H}^q\big(\gr_{\Tilde{\mathfrak{U}}}^m( J_{\mathscr{E}_1}^{[r-\cdot]}\otimes \omega_{Z_1|\mathscr{D}_1}^{\cdot} )\big)\longrightarrow \mathcal{H}^q\big(\gr_{\Tilde{\mathfrak{U}}}^m( \mathscr{O}_{\mathscr{E}_1} \otimes \omega_{Z_1|\mathscr{D}_1}^{\cdot})\big)\quad\cdots(\bigstar).\] is 
surjective without the case $ep(r-q)/(p-1)<m<ep(r-q+1)/(p-1), p|m$  and its kernels are follows:
\begin{enumerate}
\item If $m < ep(r-q)/(p-1)$ or $m \geq ep(r-q+1)/(p-1)$, then $(\bigstar)$ is an isomorphism.
\item If $m= ep(r-q)/(p-1)$, then the kernel of $(\bigstar)$ is isomorphic to the kernel of \[1-a_0^{p(r-q)}\cdot C^{-1}: Z^q\big( (\mathscr{O}_{Z_1}/T\mathscr{O}_{Z_1})\otimes_{\mathscr{O}_{Z_1}} \omega_{Z_1|\mathscr{D}_1}^{\cdot}\big)\longrightarrow\mathcal{H}^q\big( (\mathscr{O}_{Z_1}/T\mathscr{O}_{Z_1}) \otimes_{\mathscr{O}_{Z_1}} \omega_{Z_1|\mathscr{D}_1}^{\cdot} \big), \] where $a_0:=N_{\mathscr{O}_{\hat{K}/W}}(-\pi)\cdot p^{-1}\mod p\in k^*$.
\item Suppose $ep(r-q)/(p-1)<m<ep(r-q+1)/(p-1)$, then the kernel of $(\bigstar)$ is isomorphic to $B^q\big((\frac{T^m\mathscr{O}_{Z_1}}{T^{m+1}\mathscr{O}_{Z_1}})\otimes \omega_{Z_1|\mathscr{D}_1}\big)$. 
\end{enumerate}
\end{Lem}
\begin{proof}The following argument is the similar computations as the proof of Lemma \ref{Lem5}.
We note that \[m \lesseqqgtr e(r-q)+m/p \Leftrightarrow m \lesseqqgtr ep(r-q)/(p-1). \]
We have the following facts:
 \[(\dagger)\ \ \gr_{\tilde{\mathfrak{U}}}^m(\mathscr{O}_{\mathscr{E}_1}\otimes_{\mathscr{O}_{Z_1}}\omega_{Z_1}^{\cdot}) \xleftarrow{\cong} (T^m\mathscr{O}_{Z_1}/T^{m+1}\mathscr{O}_{Z_1})\otimes_{\mathscr{O}_{Z_1}}\omega_{Z_1|\mathscr{D}_1}^{\cdot},\]
\[(\dagger\dagger)\ \ \varphi_{r-q}\big(T^{e(r-q)}\big)=(a_0^p+a_1^pT^p+\cdots+a_{e-1}^pT^{(e-1)p}+(T^e)^{[p]}\big)^{r-q}.\]
Here $T^e+p(a_{e-1}T^{e-1}+\cdots+a_1T+a_0)\;(a_i \in W)$ denotes the Eisenstein polynomial of $\pi$ over $W$.

\textbf{The proof of $(1)$}: If $m \geq ep(r-q+1)/(p-1)$, we obtain \[\tilde{\mathfrak{U}}^mJ_{\mathscr{E}_1}^{[r-j]}\otimes \omega_{Z_1|\mathscr{D}_1}^{j}=\tilde{\mathfrak{U}}^m\mathscr{O}_{\mathscr{E}_1}\otimes \omega_{Z_1|\mathscr{D}_1}^{j}\ (j \geq q-1),\ \ \ \ \   \varphi_{r-q}\Big( \tilde{\mathfrak{U}}^mJ_{\mathscr{E}_1}^{[r-q]} \Big)\subset \tilde{\mathfrak{U}}^{m+1}\mathscr{O}_{\mathscr{E}_1}.\]
Then the morphism $(\bigstar)$ is the identity. Next if $m<ep(r-q)/(p-1)$, we have the following two cases:
\begin{itemize}
\item\ $p \not |m$ case : By Lemma \ref{Lem1} $(2)$, we have \[\mathcal{H}^q\big(\gr_{\Tilde{\mathfrak{U}}}^m( J_{\mathscr{E}_1}^{[r-\cdot]}\otimes \omega_{Z_1|\mathscr{D}_1}^{\cdot} )\big)=\mathcal{H}^q\big(\gr_{\Tilde{\mathfrak{U}}}^m( \mathscr{O}_{\mathscr{E}_1} \otimes \omega_{Z_1|\mathscr{D}_1}^{\cdot})\big)=0.\]
\item\ $p |m$ case : We have $\varphi_{r-q}(T^{e(r-q)+m/p})=T^m\cdot z,\ z \in \mathscr{O}^*_{\mathscr{E}_1}$ by the fact $(\dagger\dagger)$.
By using $(\dagger)$ and Lemma \ref{Lem1} $(3)$, the kernel of the morphism  $(\bigstar)$ vanishes.
\end{itemize}

\textbf{The proof of $(3)$}: If $ep(r-q)/(p-1)< m< ep(r-q+1)/(p-1)$,  we have \[\tilde{\mathfrak{U}}^mJ_{\mathscr{E}_1}^{[r-j]}\otimes \omega_{Z_1|\mathscr{D}_1}^{j}=\tilde{\mathfrak{U}}^m\mathscr{O}_{\mathscr{E}_1}\otimes \omega_{Z_1|\mathscr{D}_1}^{j}\ (j \geq q),\ \ \ \ \ d^{q-1}\Big(\tilde{\mathfrak{U}}^mJ_{\mathscr{E}_1}^{[r-q+1]}\otimes \omega_{Z_1|\mathscr{D}_1}^{q-1}\Big)\subset \tilde{\mathfrak{U}}^mJ_{\mathscr{E}_1}^{[r-q]}\otimes \omega_{Z_1|\mathscr{D}_1}^{q}.\] 
If $m > ep(r-q)/(p-1)$, we have \[\varphi_{r-q}\Big( \tilde{\mathfrak{U}}^mJ_{\mathscr{E}_1}^{[r-q]} \Big)\subset \tilde{\mathfrak{U}}^{m+1}\mathscr{O}_{\mathscr{E}_1}.\]

Then the kernel of the morphism $(\bigstar)$ is $B^q\big((T^m\mathscr{O}_{Z_1}/T^{m+1}\mathscr{O}_{Z_1})\otimes \omega_{Z_1|\mathscr{D}_1}\big)$.

\textbf{The proof of $(2)$}:
If $m=ep(r-q)/(p-1)$, we have $(2)$ from Lemma \ref{Lem1} and Lemma \ref{generalLem1} and the fact $(\dagger\dagger)$.
This completes the proof. \end{proof}

If $K$ contains a primitive $p$-th root of unity, then we have $a_0 \in (k^*)^{p-1}$(See \cite[the proof of Proposition A17]{Tsu2}).
Choose a $(p-1)$-th root $b_0 \in k$ of $a_0$. Then, by Lemma \ref{Omega}, for integers $q \geq 0, \theta\geq0$, we have 
\[(*1)\quad \omega_{Y|D,\log}^{q} \xrightarrow{\cong}  \Ker\big(1-a_0^{p\theta}C^{-1}:Z_{Y|D}^{q} \longrightarrow \mathcal{H}^{q}(\omega_{Y|D}^{\cdot})\big),\quad \omega \mapsto b_0^{-p\theta}\cdot \omega.\]

\begin{Prop}\label{general}
Let the notation and assumption be as above. Let $q$ and $r$ be an integers such that $0 \leq q \leq r \leq p-2$. Then, for every integer $m \geq 0$, we have the structure of $\mathcal{H}^q\big(\gr_{\Tilde{\mathfrak{U}}}^m(s_1(r)_{X|D})\big)$ as follows:
\begin{enumerate}
\item If $m < ep(r-q)/(p-1)$ or $m \geq ep(r-q+1)/(p-1)$, then \[\mathcal{H}^q\big(\gr_{\Tilde{\mathfrak{U}}}^m(s_1(r)_{X|D})\big)=0.\]
\item If $m= ep(r-q)/(p-1)$, then there exists an exact sequence  \[0 \longrightarrow \omega_{Y|D,\log}^{q-1} \longrightarrow \frac{\mathcal{H}^q\big(\gr_{\Tilde{\mathfrak{U}}}^m(s_1(r)_{X|D})\big)}{\mathfrak{K}^q} \longrightarrow  \omega_{Y|D,\log}^{q} \longrightarrow 0,\]
where $\mathfrak{K}^q:=\frac{\mathscr{O}_Y\otimes_{\mathscr{O}_{Z_1}}\omega_{Z_1|\mathscr{D'}_1}^{q-1}}{\mathscr{O}_Y\otimes_{\mathscr{O}_{Z_1}}\omega_{Z_1|\mathscr{D}_1}^{q-1}}$.
\item Suppose $ep(r-q)/(p-1)<m<ep(r-q+1)/(p-1)$. Then 
\begin{enumerate}
\item If $p \nmid m$, there exists an exact sequence 
\[0 \longrightarrow \frac{\omega_{Y|D}^{q-2}}{Z_{Y|D}^{q-2}}\longrightarrow \mathcal{H}^q\big(\gr_{\Tilde{\mathfrak{U}}}^m(s_1(r)_{X|D})\big)\longrightarrow  \frac{\omega_{Y|D}^{q-1}}{B_{Y|D}^{q-1}}\longrightarrow 0.\]
\item If $p | m$, there exists an exact sequence 
\[0 \longrightarrow \frac{\omega_{Y|D}^{q-2}}{Z_{Y|D}^{q-2}}\longrightarrow  \frac{\mathcal{H}^q\big(\gr_{\Tilde{\mathfrak{U}}}^m(s_1(r)_{X|D}\big)}{\mathfrak{K}^q} \longrightarrow  \frac{\omega_{Y|D}^{q-1}}{Z_{Y|D}^{q-1}}\longrightarrow 0.
\]\end{enumerate}
\end{enumerate}
\end{Prop}
\begin{proof}
We have the long exact sequence 
{\footnotesize \begin{multline*}(\natural)\ \ \cdots \longrightarrow \mathcal{H}^{q-2}\big(\gr_{\Tilde{\mathfrak{U}}}^m( \mathscr{O}_{\mathscr{E}_1} \otimes \omega_{Z_1|\mathscr{D}_1}^{\cdot})\big)  \longrightarrow  \mathcal{H}^{q-1}\big(\gr_{\Tilde{\mathfrak{U}}}^m(s_1(r)_{X|D})\big) \longrightarrow \mathcal{H}^{q-1}\big(\gr_{\Tilde{\mathfrak{U}}}^m( J_{\mathscr{E}_1}^{[r-\cdot]}\otimes \omega_{Z_1|\mathscr{D}_1}^{\cdot} )\big)\\  \xrightarrow{\mathcal{H}^{q-1}(\gr^m_{\tilde{\mathfrak{U}}}(1-\varphi_q))} \mathcal{H}^{q-1}\big(\gr_{\Tilde{\mathfrak{U}}}^m( \mathscr{O}_{\mathscr{E}_1} \otimes \omega_{Z_1|\mathscr{D}_1}^{\cdot})\big) \longrightarrow  \mathcal{H}^{q}\big(\gr_{\Tilde{\mathfrak{U}}}^m(s_1(r)_{X|D})\big)\\ \longrightarrow  \mathcal{H}^{q}\big(\gr_{\Tilde{\mathfrak{U}}}^m( J_{\mathscr{E}_1}^{[r-\cdot]}\otimes \omega_{Z_1|\mathscr{D}_1}^{\cdot} )\big)  \xrightarrow{\mathcal{H}^{q}(\gr^m_{\tilde{\mathfrak{U}}}(1-\varphi_q))} \mathcal{H}^{q}\big(\gr_{\Tilde{\mathfrak{U}}}^m( \mathscr{O}_{\mathscr{E}_1} \otimes \omega_{Z_1|\mathscr{D}_1}^{\cdot})\big) \longrightarrow \cdots. \end{multline*}}
If $m < ep(r-q)/(p-1)$ or $m \geq ep(r-q+1)/(p-1)$, the morphism $\mathcal{H}^{q}(\gr^m_{\tilde{\mathfrak{U}}}(1-\varphi_q))\ (p-2\geq r\geq q \geq 0)$ is an isomorphism by Lemma \ref{Lemq<r} $(1)$.  Then we have $\mathcal{H}^q\big(\gr_{\Tilde{\mathfrak{U}}}^m(s_1(r)_{X|D})\big)=0$ by the above long exact sequence $(\natural)$.  
If $m= ep(r-q)/(p-1)$, the kernel of $\mathcal{H}^{q-1}(\gr^m_{\tilde{\mathfrak{U}}}(1-\varphi_q))$ is isomorphic to 
\[\mathcal{K}:=\Ker\left(1-a_0^{p(r-q)}\cdot C^{-1}: Z^q\big( (\mathscr{O}_{Z_1}/T\mathscr{O}_{Z_1})\otimes_{\mathscr{O}_{Z_1}} \omega_{Z_1|\mathscr{D}_1}^{\cdot}\big)\longrightarrow\mathcal{H}^q\big( (\mathscr{O}_{Z_1}/T\mathscr{O}_{Z_1}) \otimes_{\mathscr{O}_{Z_1}} \omega_{Z_1|\mathscr{D}_1}^{\cdot} \big)\right)\]
by the Lemma  \ref{Lemq<r} $(2)$. By the same argument as Lemma \ref{Lem5}, the cokernel of  $\mathcal{H}^{q}(\gr^m_{\tilde{\mathfrak{U}}}(1-\varphi_q))$ is isomorphic to $\mathfrak{K}^q=\frac{\mathscr{O}_Y\otimes_{\mathscr{O}_{Z_1}}\omega_{Z_1|\mathscr{D'}_1}^{q-1}}{\mathscr{O}_Y\otimes_{\mathscr{O}_{Z_1}}\omega_{Z_1|\mathscr{D}_1}^{q-1}}$. 
Then we have $(\mathfrak{Z}):\;\;\mathcal{K} \cong  \frac{\mathcal{H}^q\big(\gr_{\Tilde{\mathfrak{U}}}^m(s_1(r)_{X|D})\big)}{\mathfrak{K}^q}$ by $(\natural)$. Hence we obtain the short exact sequence in the claim by the isomorphism $(*1)$ and  \eqref{a^p-ses} of Lemma \ref{generalLem1}. 

Finally, we prove the case $ep(r-q)/(p-1)<m<ep(r-q+1)/(p-1)$. If $p \nmid m$, there is a short exact sequence 
\[0 \longrightarrow \frac{\omega_{Y|D_s}^{q-2}}{Z_{Y|D_s}^{q-2}}\longrightarrow  B^q\Big(\big(T^{m}\mathscr{O}_{Z_1}/T^{m+1}\mathscr{O}_{Z_1}\big)\otimes \omega_{Z_1|\mathscr{D}_1}^{\cdot}\Big)\longrightarrow  \frac{\omega_{Y|D_s}^{q-1}}{B_{Y|D_s}^{q-1}}\longrightarrow 0\]
by Lemma \ref{Lem2} $(2)$. Then we have the claim by Lemma \ref{Lemq<r} $(3)$ and $(\natural)$. If $p | m$, we can obtain the claim by Lemma \ref{Lemq<r} $(3)$ and $(\natural)$.
This completes the proof.
\end{proof}

For integers $0 \leq q \leq r<p-2$, we define the filtration $\Tilde{\mathfrak{U}}^m$ on $\mathcal{H}^q(s_1(r)_{X|D})$ to be the image of $\mathcal{H}^q\big(\gr_{\Tilde{\mathfrak{U}}}^m(s_1(r)_{X|D}\big)$.
By the same argument as in \cite[Proposition A6]{Tsu2}, we have 
\begin{align*}&(*2)\quad\gr_{\Tilde{\mathfrak{U}}}^m\big(\mathcal{H}^q(s_1(r)_{X|D})\big) \cong \mathcal{H}^q\big(\gr_{\Tilde{\mathfrak{U}}}^m(s_1(r)_{X|D}\big),\\
  &(*3)\quad\Tilde{\mathfrak{U}}^{ep}\mathcal{H}^q(s_1(r)_{X|D})=0.\end{align*}

For $x \in \Tilde{\mathfrak{U}}^m\mathcal{H}^q(s_1(r)_{X|D})$ and $x' \in  \Tilde{\mathfrak{U}}^{m'}\mathcal{H}^{q'}(s_1(r')_{X|D})$, where $m, m', q, q' \geq 0$ and $0 \leq r, r', r+r' \leq p-2$, the product $x\cdot x'$ is contained in $\Tilde{\mathfrak{U}}^{m+m'}\mathcal{H}^{q+q'}(s_1(r+r')_{X|D})$.
By Proposition \ref{general} and $(*2)$, for each integer $0 \leq r \leq p-2$, we have an isomorphism
\[(*4)\quad\mathcal{H}^0(s_1(r)_{X|D}) \xleftarrow{\cong} \tilde{\mathfrak{U}}^{epr/(p-1)}\mathcal{H}^0(s_1(r)_{X|D}) \xrightarrow{\cong} \gr_{\tilde{\mathfrak{U}}}^{epr/(p-1)}\mathcal{H}^0(s_1(r)_{X|D}) \xrightarrow{\cong} \mathbb{Z}/p\mathbb{Z}.\]
Here to obtain the last isomorphism, we use $(*1)$ with $q=0, \theta=r$.

\begin{Def} We define a filtrations on $\mathcal{H}^q\big(s_1(r)_{X|D}\big)$ as follows:
\[\mathcal{U}^m\mathcal{H}^q\big(s_1(r)_{X|D}\big):= the\;image\;of\;\mathbb{Z}/p\mathbb{Z}\otimes U^m\mathcal{H}^q\big(s_1(q)_{X|D}\big)\; under\; the\; product\; morphism\;\]
\[\mathbb{Z}/p\mathbb{Z}\otimes \mathcal{H}^q\big(s_1(q)_{X|D}\big)\longrightarrow \mathcal{H}^q\big(s_1(r)_{X|D}\big),\]
\[\mathcal{V}^m\mathcal{H}^q\big(s_1(r)_{X|D}\big):= the\;image\;of\;\mathbb{Z}/p\mathbb{Z}\otimes V^m\mathcal{H}^q\big(s_1(q)_{X|D}\big)\; under\; the\; product\; morphism\;\]
\[\mathbb{Z}/p\mathbb{Z}\otimes \mathcal{H}^q\big(s_1(q)_{X|D}\big)\longrightarrow \mathcal{H}^q\big(s_1(r)_{X|D}\big).\]
\end{Def}
As in Lemma \ref{Lem3}, we see that the image of $U^m( (1+I_{D_2})^{\times}\otimes (M_{X_2}^{gp})^{\otimes (q-1)})\;(m \in \mathbb{N})$ under the symbol map is contained in $\tilde{\mathfrak{U}}^m\mathcal{H}^q(s_1(r)_{X|D})$, i.e. 
\begin{align*}(\heartsuit)\ \ U^m\left( (1+I_{D_2})^{\times}\otimes (M_{X_2}^{gp})^{\otimes (q-1)}\right)&\longrightarrow \mathbb{Z}/p\mathbb{Z} \otimes U^m\left( (1+I_{D_2})^{\times}\otimes (M_{X_2}^{gp})^{\otimes (q-1)}\right)\\
&\longrightarrow \tilde{\mathfrak{U}}^{ep(r-q)/(p-1)}\mathcal{H}^0(s_1(r-q)_{X|D})\otimes \tilde{\mathfrak{U}}^m\mathcal{H}^q(s_1(q)_{X|D})\\
&\longrightarrow  \tilde{\mathfrak{U}}^{ep(r-q)/(p-1)+m}\mathcal{H}^q(s_1(r)_{X|D}).
\end{align*}

Hence we have a homomorphism \[
\gr_{\mathcal{U}}^m\big(\mathcal{H}^q(s_1(r)_{X|D})\big)\longrightarrow \gr_{\tilde{\mathfrak{U}}}^{ep(r-q)/(p-1)+m}\big(\mathcal{H}^q(s_1(r)_{X|D})\big)
\]
by using $(*4)$.
Put 
\begin{align}
\gr_0^m\mathcal{H}^q(s_1(r)_{X|D})&:=\mathcal{U}^m\mathcal{H}^q(s_1(r)_{X|D})/\mathcal{V}^m\mathcal{H}^q(s_1(r)_{X|D}),\\
\gr_1^m\mathcal{H}^q(s_1(r)_{X|D})&:=\mathcal{V}^m\mathcal{H}^q(s_1(r)_{X|D})/\mathcal{U}^{m+1}\mathcal{H}^q(s_1(r)_{X|D}).
\end{align}

\begin{Prop}\label{Prop r not q}(cf. {\rm  Lemma \ref{Lem8}})
Let $m$ be a non-negative integer.  Let $x \in (1+I_{\mathscr{D}_2})^{\times}$, let $a_1,\dotsc  a_{q-1} \in M_{Z_2}^{gp}$ and let $y \in \mathscr{O}_{Z_2}(-\mathscr{D}_2)$.
Let $\overline{x}$ denote the image of $x$ in $ (1+I_{D_2})^{\times}$, let $\overline{a_i}$ denote the image of $a_i$ in $M_{X_2}^{gp}$ and let $\overline{y}$ denote the image of $y$ in $\mathscr{O}_{X_2}(-{D_2})$. Then we have:
\begin{enumerate}
\item\;If $m=0$, the image of \[\overline{x}\otimes \overline{a_1}\otimes\dotsc \overline{a_{q-1}} \in (1+I_{D_2})^{\times}\otimes (M_{X_2}^{gp})^{\otimes (q-1)}\] under the composite 
 \begin{align*}
 (1+I_{D_2})^{\times}\otimes (M_{X_2}^{gp})^{\otimes (q-1)}&\longrightarrow\gr_{\mathcal{U}}^0\big(\mathcal{H}^q(s_1(r)_{X|D})\big)\\
 &\longrightarrow \gr_{\tilde{\mathfrak{U}}}^{ep(r-q)/(p-1)}\big(\mathcal{H}^q(s_1(r)_{X|D})\big)\\
\overset{(\tilde{a})}{\twoheadrightarrow} \Ker\Big(Z^q\big( \left(\frac{\mathscr{O}_{Z_1}}{T\mathscr{O}_{Z_1}}\right)&\otimes_{\mathscr{O}_{Z_1}} \omega_{Z_1|\mathscr{D}_1}^{\cdot}\big)\xrightarrow{1-a_0^{p(r-q)}\cdot C^{-1}}\mathcal{H}^q\big( \left(\frac{\mathscr{O}_{Z_1}}{T\mathscr{O}_{Z_1}}\right) \otimes_{\mathscr{O}_{Z_1}} \omega_{Z_1|\mathscr{D}_1}^{\cdot} \big)\Big)
 \end{align*}
 is $b_0^{-p(r-q)}d\log(x)\land d\log(a_1) \land\cdots\land d\log(a_{q-1})$.
 
 By Proposition \ref{general}, $(*1)$ and $(*2)$, we get an exact sequence :
 \[0 \longrightarrow \omega_{Y|D,\log}^{q-1} \longrightarrow \frac{\gr_{\Tilde{\mathfrak{U}}}^{ep(r-q)/(p-1)}\mathcal{H}^q\big(s_1(r)_{X|D}\big)}{\mathfrak{K}^q} \longrightarrow  \omega_{Y|D,\log}^{q} \longrightarrow 0,\]
 \item\;Suppose $1 \leq m < ep/(p-1)$. If $p \nmid m$, the image of \[1+\pi^m\overline{y}\otimes \overline{a_1}\otimes\dotsc \overline{a_{q-1}} \in U^m\Big((1+I_{D_2})^{\times}\otimes (M_{X_2}^{gp})^{\otimes (q-1)}\Big)\] under the composite
  \begin{align*}
U^m\Big((1+I_{D_2})^{\times}\otimes (M_{X_2}^{gp})^{\otimes (q-1)}\Big)&\longrightarrow \gr_{\mathcal{U}}^m\big(\mathcal{H}^q(s_1(r)_{X|D})\big)\\
 \longrightarrow \gr_{\tilde{\mathfrak{U}}}^{ep(r-q)/(p-1)+m}&(\mathcal{H}^q(s_1(r)_{X|D}))\xrightarrow{\cong\;(\tilde{b})} B^q\left(\frac{T^{ep(r-q)/(p-1)+m}\mathscr{O}_{Z_1}}{T^{ep(r-q)/(p-1)+m+1}\mathscr{O}_{Z_1}}\otimes \omega_{Z_1|\mathscr{D}_1}\right)
\end{align*}
is  $d\Big(T^{ep(r-q)/(p-1)+m}b_0^{-p(r-q)}y\cdot d\log(a_1) \land\cdots\land d\log(a_{q-1})\Big)$.

If $p \nmid m$ (resp. $p|m$), by Proposition \ref{general} and $(*2)$, we get an exact sequence:
\[0 \longrightarrow \frac{\omega_{Y|D}^{q-2}}{Z_{Y|D}^{q-2}}\longrightarrow \gr_{\Tilde{\mathfrak{U}}}^{ep(r-q)/(p-1)+m}\mathcal{H}^q\big(s_1(r)_{X|D}\big)\longrightarrow  \frac{\omega_{Y|D}^{q-1}}{B_{Y|D}^{q-1}}\longrightarrow 0\]
\[(\;resp.\quad0 \longrightarrow \frac{\omega_{Y|D}^{q-2}}{Z_{Y|D}^{q-2}}\longrightarrow  \frac{\gr_{\Tilde{\mathfrak{U}}}^{ep(r-q)/(p-1)+m}\mathcal{H}^q\big(s_1(r)_{X|D}\big)}{\mathfrak{K}^q} \longrightarrow  \frac{\omega_{Y|D}^{q-1}}{Z_{Y|D}^{q-1}}\longrightarrow 0
\quad).\]

\end{enumerate}
\end{Prop}
\begin{proof}We prove by the same argument as the proof of Lemma \ref{Lem8}. First, we explain the maps $(\tilde{a})$ and $(\tilde{b})$.
By the isomorphism $(*2)$, we have \[\gr_{\Tilde{\mathfrak{U}}}^{ep(r-q)/(p-1)}\big(\mathcal{H}^q(s_1(r)_{X|D})\big) \cong \mathcal{H}^q\big(\gr_{\Tilde{\mathfrak{U}}}^{ep(r-q)/(p-1)}(s_1(r)_{X|D})\big).\] From the long exact sequence $(\natural)$, we obtain the surjective map
\[\gr_{\Tilde{\mathfrak{U}}}^{ep(r-q)/(p-1)}\big(\mathcal{H}^q(s_1(r)_{X|D})\big) \twoheadrightarrow \Ker\left(  \mathcal{H}^q\big(\gr_{\Tilde{\mathfrak{U}}}^{ep(r-q)/(p-1)}(1-\varphi_q)\big) \right).\]
The right hand side is isomorphic to \[\Ker\Big(Z^q\big( \left(\frac{\mathscr{O}_{Z_1}}{T\mathscr{O}_{Z_1}}\right)\otimes_{\mathscr{O}_{Z_1}} \omega_{Z_1|\mathscr{D}_1}^{\cdot}\big)\xrightarrow{1-a_0^{p(r-q)}\cdot C^{-1}}\mathcal{H}^q\big( \left(\frac{\mathscr{O}_{Z_1}}{T\mathscr{O}_{Z_1}}\right) \otimes_{\mathscr{O}_{Z_1}} \omega_{Z_1|\mathscr{D}_1}^{\cdot} \big)\Big)\] by Lemma 5.2 (2). Then we have the map $(\tilde{a})$. The map $(\tilde{b})$ is the same argument and by using Lemma 5.2 (3).

We put $m_0:=ep(r-q)/(p-1)$ and denote by $c$ the image of $1 \in \mathbb{Z}/p\mathbb{Z}$ in 
\[\Tilde{\mathfrak{U}}^{m_0}\mathcal{H}^0(s_1(r-q)_{X|D})=\big(\Tilde{\mathfrak{U}}^{m_0}J_{\mathscr{E}_1}^{[r-q]}\otimes \mathscr{O}_{Z_1}(-\mathscr{D}_1)\big)^{\varphi_{r-q}=1, \nabla=0}\]
under $(*4)$. Then we have $c \equiv T^{m_0}b_0^{-p(r-q)}(\mod \Tilde{\mathfrak{U}}^{m_0+1}J_{\mathscr{E}_1}^{[r-q]}\otimes \mathscr{O}_{Z_1}(-\mathscr{D}_1) )$.
The image of $\overline{x}\otimes \overline{a_1}\otimes\dotsc \overline{a_{q-1}}$ under the map $(\heartsuit)$ is the class of a cocycle of the form
 \[ \Big(c\cdot d\log x \land d\log a_1 \land \cdots \land d\log a_{q-1},\ \cdots\Big) \]
 by using Lemma \ref{symb map 2}. Its image in\[\Ker\left(Z^q\big( \left(\frac{\mathscr{O}_{Z_1}}{T\mathscr{O}_{Z_1}}\right)\otimes_{\mathscr{O}_{Z_1}} \omega_{Z_1|\mathscr{D}_1}^{\cdot}\big)\xrightarrow{1-a_0^{p(r-q)}\cdot C^{-1}}\mathcal{H}^q\big( \left(\frac{\mathscr{O}_{Z_1}}{T\mathscr{O}_{Z_1}}\right) \otimes_{\mathscr{O}_{Z_1}} \omega_{Z_1|\mathscr{D}_1}^{\cdot} \big)\right)\]
 is \[b_0^{-p(r-q)}\cdot d\log x \land d\log a_1 \land \cdots \land d\log a_{q-1}\mod T.\]
 
 If $1 \leq m < pe/(p-1)$ and $p \nmid m$, the image of $(1+\pi^m \overline{y}) \otimes \overline{a_1}\otimes\cdots \otimes \overline{a_{q-1}} \in U^m\Big( (1+I_{D_2})^{\times}\otimes (M_{X_2}^{gp})^{\otimes (q-1)}\Big)$ by the map $(\heartsuit)$ is the class of a cocycle of the form
 \[ \Big(c\cdot d\log (1+ T^m y) \land d\log a_1 \land \cdots \land d\log a_{q-1},\ \cdots\Big). \]
Then its image in {\footnotesize $B^q\Big(\big(T^{m_0+m}\mathscr{O}_{Z_1} / T^{m_0+m+1}\mathscr{O}_{Z_1}\big)\otimes \omega_{Z_1|\mathscr{D}_1}^{\cdot} \Big)$} is 
\[d\big(T^{m_0+m}b_0^{-p(r-q)} y \cdot d\log a_1\land\cdots \land d\log a_{q-1} \mod T^{m_0+m+1}\big).\]
This completes the proof.
\end{proof}

\begin{Cor}If $K$ contains a primitive $p$-th roots of unity, for any integer $q$ and $r$ such that $0 \leq q \leq r \leq p-2$, the homomorphism 
\[(\circledast)\ \ \ \mathcal{H}^0(s_1(r-q)_{X|D})\otimes \mathcal{H}^q(s_1(q)_{X|D}) \longrightarrow \mathcal{H}^q(s_1(r)_{X|D})\]
induced by the product structure is an isomorphism.
\end{Cor}
\begin{proof}
 We will prove that the morphism $(\circledast)$ is an isomorphism. The morphism $(\circledast)$ induces a morphism
 \begin{align*}
 \mathcal{H}^0(s_1(r-q)_{X|D})&\otimes \gr_{\Tilde{\mathfrak{U}}}^{m}\mathcal{H}^q(s_1(q)_{X|D})\\
 & \xrightarrow{\cong\  (*4)} \mathbb{Z}/p\mathbb{Z} \otimes \gr_{\Tilde{\mathfrak{U}}}^{m}\mathcal{H}^q(s_1(q)_{X|D})\\
 &\xrightarrow{(\circledast')} \gr_{\Tilde{\mathfrak{U}}}^{ep(r-q)/(p-1)+m}\mathcal{H}^q(s_1(r)_{X|D})
 \end{align*}
for every non-negative integer $m$. It suffices to show that the morphism $(\circledast')$ is an isomorphism. We consider the following diagrams of exact sequences:\\
$(1)$\ If $m=0$ :
{\tiny\[
\xymatrix@M=10pt{
0\ar[r]&\mathbb{Z}/p\mathbb{Z} \otimes \omega^{q-1}_{Y|D_s, \log}\ar[r]\ar[d]^-{\cong}& \mathbb{Z}/p\mathbb{Z} \otimes \frac{\gr_{\Tilde{\mathfrak{U}}}^{0}\mathcal{H}^q(s_1(q)_{X|D})}{\mathfrak{K}^q}\ar[r]\ar[d]^-{(\circledast')} &\mathbb{Z}/p\mathbb{Z} \otimes \omega^{q}_{Y|D_s, \log}\ar[d]^-{\cong}\ar[r] & 0 \\
0\ar[r]& \omega^{q-1}_{Y|D_s, \log} \ar[r]&\frac{\gr_{\Tilde{\mathfrak{U}}}^{ep(r-q)/(p-1)}\mathcal{H}^q(s_1(r)_{X|D})}{\mathfrak{K}^q}\ar[r]&\omega^{q}_{Y|D_s, \log}\ar[r] & 0.
}\]}

$(2)$\ If $1 \leq m < ep/(p-1)$ and $p \nmid m$:
{\tiny\[
\xymatrix@M=10pt{
0\ar[r]&\mathbb{Z}/p\mathbb{Z} \otimes \frac{\omega^{q-2}_{Y|D_s}}{B^{q-2}_{Y|D_s}}\ar[r]\ar[d]^-{\cong}& \mathbb{Z}/p\mathbb{Z} \otimes \gr_{\Tilde{\mathfrak{U}}}^{m}\mathcal{H}^q(s_1(q)_{X|D})\ar[r]\ar[d]^-{(\circledast')} &\mathbb{Z}/p\mathbb{Z} \otimes \frac{\omega^{q-1}_{Y|D_s}}{B^{q-1}_{Y|D_s}}\ar[d]^-{\cong}\ar[r] & 0 \\
0\ar[r]&\frac{\omega^{q-2}_{Y|D_s}}{B^{q-2}_{Y|D_s}} \ar[r]&\gr_{\Tilde{\mathfrak{U}}}^{ep(r-q)/(p-1)+m}\mathcal{H}^q(s_1(r)_{X|D})\ar[r]& \frac{\omega^{q-1}_{Y|D_s}}{B^{q-1}_{Y|D_s}}\ar[r] & 0.
}\]}
$(3)$\ If $1 \leq m < ep/(p-1)$ and $p |m$:
{\tiny\[
\xymatrix@M=10pt{
0\ar[r]&\mathbb{Z}/p\mathbb{Z} \otimes  \frac{\omega^{q-2}_{Y|D_s}}{Z^{q-2}_{Y|D_s}}\ar[r]\ar[d]^-{\cong}& \mathbb{Z}/p\mathbb{Z} \otimes \frac{\gr_{\Tilde{\mathfrak{U}}}^{m}\mathcal{H}^q(s_1(q)_{X|D})}{\mathfrak{K}^q}\ar[r]\ar[d]^-{(\circledast')} &\mathbb{Z}/p\mathbb{Z} \otimes  \frac{\omega^{q-1}_{Y|D_s}}{Z^{q-1}_{Y|D_s}}\ar[d]^-{\cong}\ar[r] & 0 \\
0\ar[r]& \frac{\omega^{q-2}_{Y|D_s}}{Z^{q-2}_{Y|D_s}} \ar[r]&\frac{\gr_{\Tilde{\mathfrak{U}}}^{ep(r-q)/(p-1)+m}\mathcal{H}^q(s_1(r)_{X|D})}{\mathfrak{K}^q}\ar[r]& \frac{\omega^{q-1}_{Y|D_s}}{Z^{q-1}_{Y|D_s}}\ar[r] & 0.
}\]}
Here the upper horizontal exact rows are obtained by Proposition \ref{general} and the lower horizontal exact rows are obtained by Proposition \ref{Prop r not q}.
By the snake lemma, we have the isomorphisms
\[ \mathbb{Z}/p\mathbb{Z} \otimes \gr_{\Tilde{\mathfrak{U}}}^{m}\mathcal{H}^q(s_1(q)_{X|D})\xrightarrow{\cong} \gr_{\Tilde{\mathfrak{U}}}^{ep(r-q)/(p-1)+m}\mathcal{H}^q(s_1(r)_{X|D})\ \ \ {\rm for}\  1 \leq m < ep/(p-1),\  p \nmid m,\]
\[ \mathbb{Z}/p\mathbb{Z} \otimes \frac{\gr_{\Tilde{\mathfrak{U}}}^{m}\mathcal{H}^q(s_1(q)_{X|D})}{\mathfrak{K}^q}\xrightarrow{\cong} \frac{\gr_{\Tilde{\mathfrak{U}}}^{ep(r-q)/(p-1)+m}\mathcal{H}^q(s_1(r)_{X|D})}{\mathfrak{K}^q}\ \ \ {\rm for}\  0 \leq m < ep/(p-1),\  p | m.\]
By the second isomorphism and the commutative diagram
{\tiny\[
\xymatrix@M=10pt{
0\ar[r]&\mathbb{Z}/p\mathbb{Z} \otimes\mathfrak{K}^q\ar[r]\ar[d]^-{\cong}&\mathbb{Z}/p\mathbb{Z} \otimes\gr_{\Tilde{\mathfrak{U}}}^{m}\mathcal{H}^q(s_1(q)_{X|D})\ar[d]^-{(\circledast')}\ar[r]&\mathbb{Z}/p\mathbb{Z} \otimes \frac{\gr_{\Tilde{\mathfrak{U}}}^{m}\mathcal{H}^q(s_1(q)_{X|D})}{\mathfrak{K}^q}\ar[r]\ar[d]^-{\cong}&0\\
0\ar[r]&\mathfrak{K}^q\ar[r]& \gr_{\Tilde{\mathfrak{U}}}^{\frac{ep(r-q)}{(p-1)}+m}\mathcal{H}^q(s_1(r)_{X|D})\ar[r]& \frac{\gr_{\Tilde{\mathfrak{U}}}^{\frac{ep(r-q)}{(p-1)}+m}\mathcal{H}^q(s_1(r)_{X|D})}{\mathfrak{K}^q}\ar[r]&0, 
}\]}
we obtain the isomorphism $(\circledast')$ for $0 \leq m < ep/(p-1),\  p | m$. \\
$(4)$ If the case $m \geq ep/(p-1)$, the claim is trivial by Proposition  \ref{general} (1). This completes the proof.
\end{proof}

\begin{Cor}\label{r not q}(cf. Theorem \ref{Main result})
 Let $e$ be the absolute ramification index of $K$. Then the sheaf $\mathcal{H}^q\big(s_1(r)_{X|D}\big)$ has the folllowing structure:
\begin{enumerate}
\item\;For $m=0$, we have short exact sequences:\\
 \begin{equation*} 
0\longrightarrow\frac{\mathfrak{R}}{\mathfrak{R}\cap\gr_1^0\mathcal{H}^q\big(s_1(r)_{X|D}\big)}\longrightarrow\gr_0^0\mathcal{H}^q\big(s_1(r)_{X|D}\big) \longrightarrow \omega_{Y|D,\log}^q\longrightarrow 0,\end{equation*}
\begin{equation*}
\quad\quad\quad\quad\quad\quad\quad\quad\quad\quad\quad\quad\quad\quad\quad\quad\quad\quad\quad\quad\{x, a_1,\dotsc, a_{q-1}\}\mapsto d\log \overline{x} \land d\log \overline{ a_1} \land \cdots \land d\log \overline{a_q}.
\end{equation*}
Here $x \in (1+I_{D_2})^{\times}$, $a_1, \dotsc, a_{q-1} \in M_{X_2}^{gp}$ and $y \in \mathscr{O}_{X_2}(-D_2)$. We denote by $\overline{x}$ {\rm (resp. }$\overline{a_i}${\rm )} the image of $x$ {\rm (resp.} $a_i${\rm)} in $M_Y^{gp}$, and we denote by $\overline{y}$ the image of $y$ in $\mathscr{O}_Y(-D_s)$.

\begin{equation*} 
0\longrightarrow\mathfrak{R}\cap\gr_1^0\mathcal{H}^q\big(s_1(r)_{X|D}\big)\longrightarrow\gr_1^0\mathcal{H}^q\big(s_1(r)_{X|D}\big)\longrightarrow \omega_{Y|D,\log}^{q-1}
\longrightarrow 0,\end{equation*}
\begin{equation*}
\quad\quad\quad\quad\quad\quad\quad\quad\quad\quad\quad\quad\quad\quad\quad\quad\quad\quad\quad\quad\{x, a_1,\dotsc, a_{q-2},\pi\}\mapsto d\log \overline{x} \land  d\log \overline{ a_1} \land \cdots \land d\log \overline{a_{q-2}},
\end{equation*}
where 
\[\mathfrak{R}:=\Ker\Big(\gr_{\mathcal{U}}^0\big(\mathcal{H}^q(s_1(r)_{X|D})\big)\rightarrow  \Ker\big(Z^q\big( \mathscr{O}_Y \otimes_{\mathscr{O}_{Z_1}} \omega_{Z_1|\mathscr{D}_1}^{\cdot}\big)\xrightarrow{1-a_0^{p(r-q)}C^{-1}} \mathcal{H}^q\big( \mathscr{O}_Y \otimes_{\mathscr{O}_{Z_1}} \omega_{Z_1|\mathscr{D}_1}^{\cdot}\big)\big)\Big).\]

\item\;If $0 <m <pe/(p-1)$ and $p \nmid m$, then we have
\begin{equation*} 
\gr_0^m\mathcal{H}^q\big(s_1(r)_{X|D}\big) \cong \frac{\omega_{Y|D}^{q-1}}{B_{Y|D}^{q-1}},\end{equation*}
\begin{equation*}
\{1+\pi^my, a_1,\dotsc, a_{q-1}\}\mapsto  b_0^{-p(r-q)}\overline{y}d\log \overline{ a_1} \land \cdots \land d\log \overline{a_q},
\end{equation*}
\begin{equation*} 
\gr_1^m\mathcal{H}^q\big(s_1(r)_{X|D}\big) \cong \frac{\omega_{Y|D}^{q-2}}{Z_{Y|D}^{q-2}},\end{equation*}
\begin{equation*}
 \{1+\pi^my, a_1,\dotsc, a_{q-2}, \pi\}\mapsto  b_0^{-p(r-q)}\overline{y}d\log \overline{ a_1} \land \cdots \land d\log \overline{a_{q-2}}.
\end{equation*}
\item\;If $0 <m <pe/(p-1)$ and $p | m$, then we have short exact sequences
\begin{equation*} 
0\longrightarrow\frac{\mathfrak{L}}{\mathfrak{L}\cap \mathcal{H}^q\big(s_1(r)_{X|D}\big)}\longrightarrow\gr_0^m\mathcal{H}^q\big(s_1(r)_{X|D}\big) \longrightarrow \frac{\omega_{Y|D}^{q-1}}{Z_{Y|D}^{q-1}}\rightarrow 0,\end{equation*}
\begin{equation*}
 \quad\quad\quad\quad\quad\quad\quad \quad\quad\quad\quad\quad\quad\quad\ \ \{1+\pi^my, a_1,\dotsc, a_{q-1}\}\mapsto b_0^{-p(r-q)}\overline{y}d\log \overline{ a_1} \land \cdots \land d\log \overline{a_q},
\end{equation*}

\begin{equation*} 
0\longrightarrow  \mathfrak{L}\cap \mathcal{H}^q\big(s_1(r)_{X|D}\big)\longrightarrow\gr_1^m\mathcal{H}^q\big(s_1(r)_{X|D}\big) \longrightarrow \frac{\omega_{Y|D}^{q-2}}{Z_{Y|D}^{q-2}}\longrightarrow 0,\end{equation*}
\begin{equation*}
 \quad\quad\quad\quad\quad\quad\quad \quad\quad\quad\quad\quad\quad\quad\ \ \{1+\pi^my, a_1,\dotsc, a_{q-2}, \pi\}\mapsto  b_0^{-p(r-q)}\overline{y}d\log \overline{ a_1} \land \cdots \land d\log \overline{a_{q-2}},
\end{equation*}
where  $\mathfrak{L}:=\Ker\Big(\gr_{\mathcal{U}}^m\big(\mathcal{H}^q(s_1(r)_{X|D})\big) \longrightarrow  B^q\left(\frac{T^{ep(r-q)/(p-1)+m}\mathscr{O}_{Z_1}}{T^{ep(r-q)/(p-1)+m+1}\mathscr{O}_{Z_1}}\otimes \omega_{Z_1|\mathscr{D}_1}\right) \Big)$.
\item\;If $m \geq pe/(p-1)$, $\mathcal{U}^m\mathcal{H}^q\big(s_1(r)_{X|D}\big)=0$.
\end{enumerate}
\end{Cor}
\begin{proof}
We put $\mathcal{H}^q:=\mathcal{H}^q(s_1(r)_{X|D})$ for simplicity.

If $m=0$, by Proposition \ref{general} $(2)$, we have the following diagram of the short exact sequences:
\[
\xymatrix@M=10pt{
0 \ar[r] & \gr_1^0\mathcal{H}^q \ar@{->>}[d] \ar[r]&\gr_{\mathcal{U}}^0\mathcal{H}^q \ar@{->>}[d] \ar[r]& \gr_0^0\mathcal{H}^q\ar[d] \ar[r]&0\\
0 \ar[r]&\omega_{Y|D_s, \log}^{q-1} \ar[r]& \mathcal{K} \ar[r] & \omega_{Y|D_s, \log}^{q} \ar[r] & 0,
 }\]
where the surjectivity of the left and middle vertical arrows are from Proposition \ref{Prop r not q} $(1)$. Here we put
\[\mathcal{K}:=\Ker\left(Z^q\big( \mathscr{O}_Y \otimes_{\mathscr{O}_{Z_1}} \omega_{Z_1|\mathscr{D}_1}^{\cdot}\big)\xrightarrow{1-a_0^{p(r-q)}C^{-1}} \mathcal{H}^q\big( \mathscr{O}_Y \otimes_{\mathscr{O}_{Z_1}} \omega_{Z_1|\mathscr{D}_1}^{\cdot}\big)\right).\]
By the snake lemma, we have two short exact sequences in the assertion $(1)$. 
If $0< m < pe/(p-1)$ and $p\nmid m$, by Proposition \ref{general} $(3)$ (a), we obtain the following diagram
{\footnotesize\[
\xymatrix@M=10pt{
0 \ar[r] & \gr_1^m\mathcal{H}^q\ar@{->>}[d] \ar[r]&\gr_{\mathcal{U}}^m\mathcal{H}^q\ar[d]^-{\cong} \ar[r]& \gr_0^m\mathcal{H}^q\ar[d] \ar[r]&0\\
0 \ar[r]& \frac{\omega_{Y|D_s}^{q-2}}{Z_{Y|D_s}^{q-2}} \ar[r]&B^q \left(\left(\frac{T^{\frac{ep(r-q)}{p-1}+m}\mathscr{O}_{Z_1}}{T^{\frac{ep(r-q)}{p-1}+m+1}\mathscr{O}_{Z_1}}\right)\otimes \omega_{Z_1|\mathscr{D}_1}^{\cdot}\right)\ar[r]& \frac{\omega_{Y|D_s}^{q-1}}{B_{Y|D_s}^{q-1}} \ar[r] & 0,
 }\]}
where the left vertical arrow is surjective and the middle vertical arrow is an isomorphism by Proposition \ref{Prop r not q} $(2)$. By the snake lemma,  we obtain the isomorphisms in the assertion $(2)$.
If $0< m < pe/(p-1)$ and $p | m$, Proposition \ref{general} $(3)$ (b), we have the following diagram
{\footnotesize \[
\xymatrix@M=10pt{
0 \ar[r] & \gr_1^m\mathcal{H}^q \ar@{->>}[d] \ar[r]&\gr_{\mathcal{U}}^m\mathcal{H}^q \ar@{->>}[d] \ar[r]& \gr_0^m\mathcal{H}^q\ar[d] \ar[r]&0\\
0 \ar[r]& \frac{\omega_{Y|D_s}^{q-2}}{Z_{Y|D_s}^{q-2}} \ar[r]&B^q \left(\left(\frac{T^{\frac{ep(r-q)}{p-1}+m}\mathscr{O}_{Z_1}}{T^{\frac{ep(r-q)}{p-1}+m+1}\mathscr{O}_{Z_1}}\right)\otimes \omega_{Z_1|\mathscr{D}_1}^{\cdot}\right)\ar[r] & \frac{\omega_{Y|D_s}^{q-1}}{Z_{Y|D_s}^{q-1}} \ar[r] & 0.
 }\]}
Here the surjectivity of the left and middle vertical arrows are from Lemma Proposition \ref{Prop r not q} $(2)$. 
The lower exact sequence, we use the identification $(\mathfrak{Z})$ in the proof of Proposition 5.3 (3) $(b)$.
By the snake lemma, we obtain the assertion $(3)$. 
 Since $U^{pe}\mathcal{H}^q(s_1(q)_{X|D})\big)=0$ by Lemma \ref{Lem3} and Corollary \ref{Lem6}, this implies ($4$). This completes the proof of this Proposition.
\end{proof}

Next we do not assume that $K$ contains a primitive $p$-th root  of unity. Let $\mathscr{O}_{K'}$ be a totally ramified extension of $\mathscr{O}_{K}$ of degree $w$.
We denote $(S', N')$ the scheme $\Spec{\mathscr{O}_{K'}}$ with the log structure defined by the closed point. Assume that there exists a prime $\pi'$ of $\mathscr{O}_{K'}$ such that $\pi'^w=\pi$. We choose such a prime $\pi'$.
Let $(V',M_{V'})$ be the scheme $\Spec(W[\mathbb{N}])=\Spec(W[T'])$ endowed with the log structure associated to the inclusion $\mathbb{N}\hookrightarrow W[\mathbb{N}]$.
 We define the exact closed immersion $i_{V'_n}:(S'_n, M_{S'_n})\hookrightarrow (V',M_{V'})$ in the same way as $i_{V_n}$, by using $\pi'$ (see the argument before Lemma \ref{Lem1}). 
We have a cartesian diagram:

\[
\xymatrix@M=10pt{  
(S'_n, M_{S'_n})\ar@{}[rd]|{\square} \ar[r]\ar[d]& (V'_n,M_{V'_n})\ar[d]^{(\clubsuit)}&\\
(S_n,N_n)\ar[r]&(V_n,M_{V_n}),&
}\]

where the morphism $(\clubsuit)$ is defined by the multiplication by $d$ on $\mathbb{N}$. We define $(X', M_{X'}):=(X,M_X)\times_{(S,N)} (S',N')$, $\tilde{D}:=D\times_S S'$, and denote $\tilde{\mathscr{D}}_n$, $(Z'_n, M_{Z'_n})$ and $\{F_{Z'_n}\}$ the base changes of $\mathscr{D}_n$, $(Z_n, M_{Z_n})$ and $\{F_{Z_n}\}$ under the morphism $(\clubsuit)$ above. Then one can apply the above arguments to $\mathscr{O}_{K'}$, $\pi'$, $(X',M_{X'})$, $(Z'_n, M_{Z'_n})$ and $\{F_{Z'_n}\}$. We denote by $'$ the corresponding things.  Since $(Y',M_{Y'}):=(Y,M_Y)\times_{(s, M_s)} (s',M'_s)$, $s=s'$ and $Y'=Y$, then we have $\omega_{Y|D_s}^{\cdot}\xrightarrow{\cong} \omega_{Y'|\tilde{D}_s}^{\cdot}$. Thus we obtain the following relations of the filtrations $\tilde{\mathfrak{U}}^m$ on $\mathcal{H}^q(s_1(r)_{X|D})$ and $\mathcal{H}^q(s_1(r)'_{X|D})$ from Proposition \ref{general} and $(*2)$ :

\begin{Lem}\label{comp}(cf. \cite[Lemma A18]{Tsu2})
Let $r$ and $q$ be integers such that $0 \leq q \leq r \leq p-2$. Then there exists a canonical morphism 
\[\mathcal{H}^q(s_1(r)_{X|D}) \longrightarrow \mathcal{H}^q(s_1(r)_{X|D}'),\]
which sends $\tilde{\mathfrak{U}}^m$ into $\tilde{\mathfrak{U}}^{wm}$ for $m \in \mathbb{N}$.  If $ep(r-q)/(p-1) \leq m< ep(r-q+1)/(p-1)$, then we have the following commutative diagram:

$$
\begin{CD}   
0@>>>\mathfrak{W}_1@>>>\frac{\gr_{\Tilde{\mathfrak{U}}}^m\mathcal{H}^q\big(s_1(r)_{X|D}\big)}{\mathfrak{K}}@>>>\mathfrak{W}_2@>>> 0\quad\cdots(\mathscr{P}_1) \\
@.@VVw\cdot id_{\mathfrak{W}}V@VVV@VVpr.V @. \\
0@>>>\mathfrak{W}_1@>>>\frac{\gr_{\Tilde{\mathfrak{U}}}^{dm}\mathcal{H}^q\big(s_1(r)_{X|D}'\big)}{\mathfrak{K}'}@>>>\mathfrak{W}'_2@>>>0\quad\cdots(\mathscr{P}_2) ,
\end{CD}
$$\\
where the horizontal rows $(\mathscr{P}_1) $ and $(\mathscr{P}_2)$ are exact sequence.
Here
\[\mathfrak{W}_1:= \Ker\big(1-a_0^{p(r-q)}C^{-1}:Z_{Y|D}^{q-1} \longrightarrow \mathcal{H}^{q-1}(\omega_{Y|D}^{\cdot})\big) \quad(resp.\;\frac{\omega_{Y|D}^{q-2}}{Z_{Y|D}^{q-2}},\; resp.\;\frac{\omega_{Y|D}^{q-2}}{Z_{Y|D}^{q-2}}), \]
\[\mathfrak{W}_2:= \Ker\big(1-a_0^{p(r-q)}C^{-1}:Z_{Y|D}^{q} \longrightarrow \mathcal{H}^{q}(\omega_{Y|D}^{\cdot})\big) \quad(resp.\;\frac{\omega_{Y|D}^{q-1}}{B_{Y|D}^{q-1}},\; resp.\; \frac{\omega_{Y|D}^{q-1}}{Z_{Y|D}^{q-1}}), \]
{\rm if}     \[m=ep(r-q)/(p-1)\quad (resp.\;m > ep(r-q)/(p-1),\;p\nmid m,\;resp.\;m > ep(r-q)/(p-1),\;p|m),\]
\[\mathfrak{W}'_2:=\frac{\omega_{Y|D}^{q-1}}{Z_{Y|D}^{q-1}}\quad {\rm if}\;m > ep(r-q)/(p-1),\;p|wm,\]
and $\mathfrak{W}'_2=\mathfrak{W}_2$ otherwise, 
\[\mathfrak{K}:= \mathfrak{K}^q\quad(resp.\;0,\;resp.\;\mathfrak{K}^q)\]
\[{\rm if}\;m=ep(r-q)/(p-1)\quad(resp.\;m > ep(r-q)/(p-1),\;p\nmid m,\;resp.\;m > ep(r-q)/(p-1),\;p|m),\]
\[\mathfrak{K}':= \mathfrak{K}'^q\quad(resp.\;0,\;resp.\;\mathfrak{K}'^q)\]
\[{\rm if}\;m=ep(r-q)/(p-1)\quad(resp.\;m > ep(r-q)/(p-1),\;p\nmid wm,\;resp.\;m > ep(r-q)/(p-1),\;p|wm).\]
Here $\mathfrak{K}^q:=\frac{\mathscr{O}_Y\otimes_{\mathscr{O}_{Z_1}}\omega_{Z_1|\mathscr{D'}_1}^{q-1}}{\mathscr{O}_Y\otimes_{\mathscr{O}_{Z_1}}\omega_{Z_1|\mathscr{D}_1}^{q-1}}$, $\mathfrak{K}'^q:=\frac{\mathscr{O}_Y\otimes_{\mathscr{O}_{Z'_1}}\omega_{Z'_1|\tilde{\mathscr{D}}'_1}^{q-1}}{\mathscr{O}_Y\otimes_{\mathscr{O}_{Z'_1}}\omega_{Z'_1|\tilde{\mathscr{D}}_1}^{q-1}}$. We denote by $pr.$ the canoical projection or the identity.
If $K'$ is tamely ramified filed over $K$, we have an isomorphism:
\[ \frac{\gr_{\Tilde{\mathfrak{U}}}^m\mathcal{H}^q\big(s_1(r)_{X|D}\big)}{\mathfrak{K}} \xrightarrow{\cong} \frac{\gr_{\Tilde{\mathfrak{U}}}^{wm}\mathcal{H}^q\big(s_1(r)_{X|D}'\big)}{\mathfrak{K}'},\]

\end{Lem}
\begin{proof}
The horizontal rows of the diagram $(\mathscr{P}_1) $ and $(\mathscr{P}_2)$ are obtained by Proposition 5.3 for each cases.
The first claim is trivial by $T=T'^w$. We show the second claim. The horizontal rows of diagrams of the second claim are obtained by \eqref{a^p-ses} in Lemma \ref{generalLem1}, 
by $(3)$ in Lemma \ref{general} and by $(*2)$. We prove the commutativity of these diagrams below.\\
$(1)$ $m=ep(r-q)/(p-1)$ case:  We have the following diagram
{\tiny\[
\xymatrix@M=3pt{ 
0 \ar[r]&\Ker\big(1-a_0^{p(r-q)}C^{-1}\big)\ar[r]\ar[d]^-{w\cdot id} &\frac{\gr_{\Tilde{\mathfrak{U}}}^m\mathcal{H}^q\big(s_1(r)_{X|D}\big)}{\mathfrak{K}^q} \ar[d]\ar[r]&\Ker\big(1-a_0^{p(r-q)}C^{-1}\big)\ar[r]\ar[d]^-{id}&0\\ 
0 \ar[r]&\Ker\big(1-a_0^{p(r-q)}C^{-1}\big)\ar[r] &\frac{\gr_{\Tilde{\mathfrak{U}}}^{dm}\mathcal{H}^q\big(s_1(r)'_{X|D}\big)}{\mathfrak{K}'^{q}} \ar[r]&\Ker\big(1-a_0^{p(r-q)}C^{-1}\big)\ar[r]&0.}\]}
The image of {\tiny$x d\log(\overline{y}) \land d\log(\overline{a_2}) \land \cdots \land d\log(\overline{a_{q-1}}) \in \Ker\big(1-a_0^{p(r-q)}C^{-1}\big)$} in $\frac{\gr_{\Tilde{\mathfrak{U}}}^{dm}\mathcal{H}^q\big(s_1(r)'_{X|D}\big)}{\mathfrak{K}}$ is 
\begin{align*}
x\cdot (d\log(y) \land &d\log(a_2) \land \cdots \land d\log(a_{q-1})\land d\log(T) \mod T)\\
=wx\cdot (d\log(y) \land &d\log(a_2) \land \cdots \land d\log(a_{q-1})\land d\log(T') \mod T'^w),
\end{align*}
where we use that $d\log T=w \cdot d\log T'$. On the other hand, the image of $w\cdot x d\log(\overline{y}) \land d\log(\overline{a_2}) \land \cdots \land d\log(\overline{a_{q-1}}) \in \Ker\big(1-a_0^{p(r-q)}C^{-1}\big)$ in
$\frac{\gr_{\Tilde{\mathfrak{U}}}^{dm}\mathcal{H}^q\big(s_1(r)'_{X|D}\big)}{\mathfrak{K}}$ is \[wx\cdot (d\log(y) \land d\log(a_2) \land \cdots \land d\log(a_{q-1})\land d\log(T') \mod T'^w)\] by calculating counterclockwise. Thus the left square is commutative. The commutativity of the right square is obvious.\\
$(2)$ $m>ep(r-q)/(p-1)$ case: We have the following diagram:
{\footnotesize\[
\xymatrix@M=10pt{ 
0 \ar[r]&\frac{\omega_{Y|D}^{q-2}}{Z_{Y|D}^{q-2}}\ar[r]\ar[d]^-{w\cdot id} &\frac{\gr_{\Tilde{\mathfrak{U}}}^m\mathcal{H}^q\big(s_1(r)_{X|D}\big)}{\mathfrak{K}} \ar[d]\ar[r]&\frac{\omega_{Y|D}^{q-1}}{Z_{Y|D}^{q-1}} \ar[r]\ar[d]^-{pr}&0\\ 
0 \ar[r]&\frac{\omega_{Y|D}^{q-2}}{Z_{Y|D}^{q-2}}\ar[r] &\frac{\gr_{\Tilde{\mathfrak{U}}}^{dm}\mathcal{H}^q\big(s_1(r)'_{X|D}\big)}{\mathfrak{K}'} \ar[r]&\mathfrak{W}'_2\ar[r]&0.}\]}
The image of $\overline{x} d\log(\overline{y}) \land d\log(\overline{a_2}) \land \cdots \land d\log(\overline{a_{q-2}}) \in \frac{\omega_{Y|D}^{q-2}}{Z_{Y|D}^{q-2}}$ in
$\frac{\gr_{\Tilde{\mathfrak{U}}}^{dm}\mathcal{H}^q\big(s_1(r)'_{X|D}\big)}{\mathfrak{K}'}$ is 
\begin{align*}
d(T^mx\cdot d\log(y) \land &d\log(a_2) \land \cdots \land d\log(a_{q-2})\land d\log(T) \mod T^{m+1})\\
=d(T'^{wm}dx\cdot d\log(y) \land &d\log(a_2) \land \cdots \land d\log(a_{q-2})\land d\log(T') \mod T^{w(m+1)}),
\end{align*}
where we use that $d\log T=w \cdot d\log T'$. On the other hand, the image of $w\cdot \overline{x} d\log(\overline{y}) \land d\log(\overline{a_2}) \land \cdots \land d\log(\overline{a_{q-2}}) \in \frac{\omega_{Y|D}^{q-2}}{Z_{Y|D}^{q-2}}$ in $\frac{\gr_{\Tilde{\mathfrak{U}}}^{dm}\mathcal{H}^q\big(s_1(r)'_{X|D}\big)}{\mathfrak{K}'}$ is 
\[d(T'^{wm}wx\cdot d\log(y) \land d\log(a_2) \land \cdots \land d\log(a_{q-2})\land d\log(T') \mod T^{w(m+1)}).\]
Hence the left square is commutative. The commutativity of the right square is obvious. Finally, if $K'$ is tamely ramified filed over $K$, we have $p \nmid w$. Then the above all cases, the morphisms $w\cdot id$ and $pr(=id)$ are an isomorphism. Thus we have the isomorphism 
\[ \frac{\gr_{\Tilde{\mathfrak{U}}}^m\mathcal{H}^q\big(s_1(r)_{X|D}\big)}{\mathfrak{K}} \xrightarrow{\cong} \frac{\gr_{\Tilde{\mathfrak{U}}}^{dm}\mathcal{H}^q\big(s_1(r)'_{X|D}\big)}{\mathfrak{K}'}\] by using snake lemma. This completes the proof.
\end{proof}

\begin{Cor} If $0\leq m \leq ep/(p-1)$ and $K'$ is tamely ramified filed over $K$,  the kenel and the cokernel of \[\gr_{\tilde{\mathfrak{U}}}^{m}\mathcal{H}^q\big(s_1(r)_{X|D}\big) \longrightarrow \gr_{\tilde{\mathfrak{U}}}^{dm}\mathcal{H}^q\big(s_1(r)_{X|D}'\big)\]
are Mittag-Leffler zero  with respect to the multiplicities of the prime components of $D$. \end{Cor}
\begin{proof} We consider a commutative diagram

$$
\begin{CD}   
0 @>>> \mathfrak{K}@>>> \gr_{\tilde{\mathfrak{U}}}^{m}\mathcal{H}^q\big(s_1(r)_{X|D}\big)@>>>\frac{\gr_{\tilde{\mathfrak{U}}}^{m}\mathcal{H}^q\big(s_1(r)_{X|D}\big)}{\mathfrak{K}}@>>>0\\
@.@VVV @VVV@VVV\\
0@>>>\mathfrak{K}'@>>>\gr_{\tilde{\mathfrak{U}}}^{dm}\mathcal{H}^q\big(s_1(r)_{X|D}'\big)@>>> \frac{\gr_{\tilde{\mathfrak{U}}}^{dm}\mathcal{H}^q\big(s_1(r)_{X|D}'\big)}{\mathfrak{K}'}@>>>0.
\end{CD}
$$\\

\noindent
From Lemma \ref{comp}, the right vertical arrow is an isomorphism. The kernel and the cokernel of the left vertical arrow are Mittag-Leffler zero because $\mathfrak{K}$ and $\mathfrak{K}'$ are Mittag-Leffler zero.
Thus we obtain the claim by the snake lemma.
\end{proof}

By the same arguments as in Lemma \ref{CokernelSymb} and Lemma \ref{compU}, we have the following Proposition:
\begin{Prop}\label{kercok}(cf. {\rm Lemma \ref{compU}})
The kernel and the cokernel of the morphism \[\Phi_{m,D}:\ \ \gr_{\mathcal{U}}^m\mathcal{H}^q\big(s_1(r)_{X|D}\big) \rightarrow \gr_{\Tilde{\mathfrak{U}}}^{m+ep(r-q)/(p-1)}\mathcal{H}^q\big(s_1(r)_{X|D}\big)\]
are Mittag-Leffler zero with respect to the multiplicities of the prime components of $D$.
\end{Prop}
To prove the above Proposition, we need the following Lemma:
\begin{Lem}(cf. {\rm Lemma \ref{CokernelSymb}})\label{Cokernel Symmb*}
The cokernel of the morphism
\[\mathfrak{gr}^m\Symb_{X|D}: \gr_{U}^m\left((1+I_{D_2})^{\times}\otimes (M_{X_2}^{gp})^{\otimes q-1}\right)\longrightarrow \gr_{\tilde{\mathfrak{U}}}^{m+ep(r-q)/(p-1)}\mathcal{H}^q\big(s_1(r)_{X|D}\big)\]
is Mittag-Leffler zero with respect to the multiplicities of the prime components of $D$.
\end{Lem}
\begin{proof}This proof is the same as the proof of Lemma \ref{CokernelSymb}.
We have the following commutative diagram:\\
{\footnotesize  \[
\xymatrix{
&&0 \ar[d]&& \\
&&\mathfrak{L}'^m\ar@{->>}[ld]_{(\diamond 1)} \ar[d]&&\\
0&\Coker(\mathfrak{gr}^m\Symb_{X|D})\ar[l]&\gr_{\Tilde{\mathfrak{U}}}^{m+\frac{ep(r-q)}{(p-1)}}\mathcal{H}^q(s_1(r)_{X|D})\ar[l]\ar[d]&\gr_{U}^m\big((1+I_{D_2})^{\times}\otimes(M_{X_2}^{gp})^{\otimes(q-1)}\big)\ar@{->>}[ld]^{(\diamond 2)} \ar[l]&\\
&&\mathfrak{E} \ar[d]&&\\
&&0&&
}\]
}
where the vertical and horizontal sequences is exact.   Here we put 
 \[ \mathfrak{E}:=\begin{cases}  \Ker\Big(Z^q\big( \mathscr{O}_Y \otimes_{\mathscr{O}_{Z_1}} \omega_{Z_1|\mathscr{D}_1}^{\cdot}\big)\xrightarrow{1-a_0^{p(r-q)}C^{-1}} \mathcal{H}^q\big( \mathscr{O}_Y \otimes_{\mathscr{O}_{Z_1}} \omega_{Z_1|\mathscr{D}_1}^{\cdot} \big)\Big)\ {\rm if}\ m=0,&\\
 B^q\left((T^{m+\frac{ep(r-q)}{p-1}}\mathscr{O}_{Z_1}/T^{m+\frac{ep(r-q)}{p-1}+1}\mathscr{O}_{Z_1})\otimes \omega^{\cdot}_{Z_1|\mathscr{D}_1}\right)\ \ \ \ \ \ {\rm if}\ 0 <m<pe/(p-1),&\\
 
 \end{cases}\]

 {\tiny\[ \mathfrak{L}_0^{'m}:=\begin{cases}
\Ker\left(\mathcal{H}^q(\gr_{\tilde{\mathfrak{U}}}^0s_1(r)_{X|D}) \twoheadrightarrow \Ker\Big(Z^q\big( \mathscr{O}_Y \otimes_{\mathscr{O}_{Z_1}} \omega_{Z_1|\mathscr{D}_1}^{\cdot}\big)\xrightarrow{1-a_0^{p(r-q)}C^{-1}} \mathcal{H}^q\big( \mathscr{O}_Y \otimes_{\mathscr{O}_{Z_1}} \omega_{Z_1|\mathscr{D}_1}^{\cdot} \big)\Big)\right)\ {\rm if}\ m=0,& \\
\Ker\left(\mathcal{H}^q(\gr_{\tilde{\mathfrak{U}}}^ms_1(r)_{X|D}) \twoheadrightarrow  B^q\left(\left(\frac{T^{m+\frac{ep(r-q)}{p-1}}\mathscr{O}_{Z_1}}{T^{m+\frac{ep(r-q)}{p-1}+1}\mathscr{O}_{Z_1}}\right)\otimes \omega^{\cdot}_{Z_1|\mathscr{D}_1}\right)\right) \ \ \ {\rm if}\ 0 <m<pe/(p-1),\  p|m,&\\
0 \ \ \  {\rm if}\ 0 <m<pe/p-1,\  p \nmid m,&
\end{cases}\]}
\[\mathfrak{L}'^m:= \mathfrak{L}_0^{'m} \cap \gr_{\Tilde{\mathfrak{U}}}^{m+\frac{ep(r-q)}{p-1}}\mathcal{H}^q(s_1(r)_{X|D}).\]
The morphism $(\diamond 2)$ is constructed in Proposition \ref{Prop r not q}. It is surjective by the explicit assignments in Proposition \ref{Prop r not q}. 
We have 
\[\mathfrak{L}_0^{'m} \cong \frac{\mathscr{O}_Y\otimes_{\mathscr{O}_{Z_1}}\omega^{q-1}_{Z_1|\mathscr{D}'_1}}{\mathscr{O}_Y\otimes_{\mathscr{O}_{Z_1}}\omega^{q-1}_{Z_1|\mathscr{D}_1}}\ \ ({\rm if}\ 0 \leq m \leq pe/(p-1),\ p|m)\]
by the similar argument as the proof of Lemma \ref{Lem5} $(1)$ and $(3)$. Thus $\mathfrak{L}_0^{'m}$ is Mittag-Leffler zero with respect to the multiplicities of the prime components of $D$. Since $(\diamond 2)$ is surjective, $(\diamond 1)$ is also surjective. Since $\mathfrak{L}'^m$ is Mittag-Leffler zero, so is $\Coker(\gr^m{\rm Symb}_{X|D})$. \end{proof}

\textit{Proof of Proposition \ref{kercok}} : We put $\mathcal{H}^q:=\mathcal{H}^q(s_1(r)_{X|D})$ for simplicity.
We consider the following commutative diagram:\\
$$
\begin{CD}   
0 @>>> \mathcal{U}^{m+1}\mathcal{H}^{q}@>>>\mathcal{U}^{m}\mathcal{H}^{q}@>>>\gr_{\mathcal{U}}^{m}\mathcal{H}^{q}@>>>0\\
@.@VVV @VVV@VV\Phi_{m,D}V\\
0@>>>\Tilde{\mathfrak{U}}^{m+\frac{pe(r-q)}{p-1}+1}\mathcal{H}^{q}@>>> \Tilde{\mathfrak{U}}^{m+\frac{pe(r-q)}{p-1}}\mathcal{H}^{q}@>>> \gr_{\Tilde{\mathfrak{U}}}^{m+\frac{pe(r-q)}{p-1}}\mathcal{H}^{q}@>>>0
\end{CD}
$$\\
The left and central vertical morphism is injective. If $m > pe/(p-1)$, the claim is trivial. We assume that $0 \leq m \leq pe/(p-1)$. If $m=pe/(p-1)$, the right vertical morphism is injective by the same argument as Lemma \ref{Lem6}  and the cokernel of $\Phi_{m,D}$ is Mittag-Leffler zero from Lemma \ref{Cokernel Symmb*}. We can easily show the assertion by induction on $m$. $\square$

\section{Acknowledgements}
I am thankful to Professor Kanetomo Sato for his great help and discussion. This problem was suggested to me by him. 
I thank Takashi Suzuki for his advice and comment, especially on the proof of the existence of a nice hyper covering.
I would like to thank the referee for his/her numerous valuable comments and suggestions to improve the quality of  this paper.

\end{document}